\documentclass[oneside]{amsart}
\usepackage{graphicx}
\usepackage{amssymb}
\usepackage{amsthm}
\usepackage{amsmath}
\usepackage{amsaddr}

\newtheorem{thm}{Theorem}

\newtheorem{lem}{Lemma}
\newtheorem{prop}{Proposition}
\newtheorem{rem}{Remark}
\newtheorem*{prop*}{Proposition}
\newtheorem*{thm*}{Theorem}
\numberwithin{thm}{section}
\numberwithin{cor}{section}
\numberwithin{lem}{section}
\numberwithin{prop}{section}
\numberwithin{rem}{section}
\numberwithin{equation}{section}

\renewcommand{\Re}{\text{Re\,}}
\renewcommand{\Im}{\text{Im\,}}
\newcommand{\Ai}{\text{Ai}}
\newcommand{\Bi}{\text{Bi}}

\begin{document}
\title[Zeroes of the spectral density]
{Zeroes of the spectral density of the
\\
Schr\"odinger operator with the
\\
slowly decaying Wigner--von Neumann potential}

\author{Sergey Simonov}
\address{%
Dublin Institute of Technology\\
School of Mathematical Sciences\\
Kevin Street, Dublin 8\\
Ireland}
\email{sergey.a.simonov@gmail.com}


\maketitle

\begin{abstract}
We consider the Schr\"odinger operator $\mathcal L_{\alpha}$ on the half-line with a periodic background potential and a perturbation which consists of two parts: a summable potential and the slowly decaying Wigner--von Neumann potential $\frac{c\sin(2\omega x+\delta)}{x^{\gamma}}$, where $\gamma\in(\frac12,1)$. The continuous spectrum of this operator has the same band-gap structure as the continuous spectrum of the unperturbed periodic operator. In every band there exist two points, called critical, where the eigenfunction equation has square summable solutions. Every critical point $\nu_{cr}$ is an eigenvalue of the operator $\mathcal L_{\alpha}$ for some value of the boundary parameter $\alpha=\alpha_{cr}$, specific to that particular point. We prove that for $\alpha\neq\alpha_{cr}$ the spectral density of the operator $\mathcal L_{\alpha}$ has a zero of the exponential type at $\nu_{cr}$.
\end{abstract}

\tableofcontents

\section{Introduction}

Wigner--von Neumann potential $\frac{c\sin(2\omega x)}x$ gives probably the simplest example of an eigenvalue embedded into the continuous spectrum of the Schr\"odinger operator on the half-line. The eigenvalue can appear at the point $\omega^2$  which is called critical or resonance point, and appears for only one boundary condition and only if $|c|>2|\omega|$. At this point a square summable solution of the eigenfunction equation exists, and for one value of the boundary parameter it satisfies the boundary condition. This situation is very unstable: the eigenvalue disappears if one slightly changes the boundary condition or adds a summable perturbation to the potential. This eigenvalue can be related to resonances, which are poles on the unphysical sheet of the analytic continuation of the Green's function, but may or may not exist. It is meaningful to consider objects that are stable under small perturbations. One can study properties of the Weyl's $m$-function or the spectral density $\rho'$ of the operator (which is the derivative of the spectral function \cite{Levitan-Sargsyan-1975}).

In the present paper we study properties of the spectral density of the Schr\"odinger operator $\mathcal L_{\alpha}$ defined by the differential expression
\begin{equation}\label{L}
    L:=-\frac{d^2}{dx^2}
    +q(x)+q_{WN}(x,\gamma)+q_1(x)
\end{equation}
on the positive half-line with the boundary condition
\begin{equation}\label{bc}
    \psi(0)\cos\alpha-\psi'(0)\sin\alpha=0,
\end{equation}
which means that $\mathcal L_{\alpha}$ acts as
\begin{equation}\label{rule}
    \mathcal L_{\alpha}:\psi\mapsto L\psi
\end{equation}
on the domain
\begin{equation}\label{domain}
    \text{dom}\,\mathcal L_{\alpha}=\{\psi\in L_2(\mathbb R_+)\cap H^2_{loc}(\mathbb R_+): L\psi\in L_2(\mathbb R_+), \psi(0)\cos\alpha-\psi'(0)\sin\alpha=0\}
\end{equation}
in the Hilbert space $L_2(\mathbb R_+)$. We assume that
\begin{equation}\label{conditions}
    \begin{array}{l}
    \cdot\ q\text{ is periodic with the period }a\text{ and }q\in L_1(0,a),
    \\
    \cdot\ c,\omega,\delta\in\mathbb R,\gamma\in(\frac12,1],\frac{2a\omega}{\pi}\notin\mathbb Z,
    \\
    \cdot\ q_{WN}(x,\gamma):=
    \left\{
    \begin{array}{l}
    \frac{c\sin(2\omega x+\delta)}{x^{\gamma}},\text{ if }\gamma\in(\frac12,1),
    \\
    \frac{c\sin(2\omega x+\delta)}{x+1},\text{ if }\gamma=1,
    \end{array}
    \right.
    \\
    \cdot\ q_1\in L_1(\mathbb R_+),
    \\
    \cdot\ \alpha\in[0,\pi).
    \end{array}
\end{equation}
Under these assumptions the operator $\mathcal L_{\alpha}$ is self-adjoint \cite{Kurasov-Simonov-2013}.

As it was shown by Kurasov and Naboko in \cite{Kurasov-Naboko-2007}, the absolutely continuous spectrum of the operator given by the expression $L$ on the whole real line coincides with the spectrum of the corresponding unperturbed periodic operator on the whole line,
\begin{equation}\label{L-per}
    \mathcal L_{per}:\psi\mapsto-\psi''+q\psi,
\end{equation}
\begin{equation*}
    \text{dom}\,\mathcal L_{per}=\{\psi\in L_2(\mathbb R)\cap H^2_{loc}(\mathbb R): (-\psi''+q\psi)\in L_2(\mathbb R)\},
\end{equation*}
which means it has a band-gap structure:
\begin{equation}\label{prelim spectrum of periodic operator}
    \sigma(\mathcal L_{per})=\bigcup\limits_{j=0}^{\infty}([\lambda_{2j},\mu_{2j}]\cup[\mu_{2j+1},\lambda_{2j+1}]),
\end{equation}
where
\begin{equation}\label{prelim edges of spectrum}
    \lambda_0<\mu_0\le\mu_1<\lambda_1\le\lambda_2<\mu_2\le\mu_3<\lambda_3\le\lambda_4<...
\end{equation}
In turn, the absolutely continuous spectrum of $\mathcal L_{\alpha}$ coincides as a set with the spectrum of $\mathcal L_{per}$, although the latter has multiplicity two, whereas the former is simple. Moreover, Kurasov and Naboko in \cite{Kurasov-Naboko-2007} showed that in every band $[\lambda_j,\mu_j]$ or $[\mu_j,\lambda_j]$  there exist two critical points $\nu_{j+}$ and $\nu_{j-}$. The type of asymptotics of generalised eigenvectors (solutions of the eigenfunction equation) at these points is different from that type in other points of the absolutely continuous spectrum. At each of the critical points there exists a subordinate solution and hence each of these points can be an eigenvalue of the operator $\mathcal L_{\alpha}$ as long as this solution belongs to $L_2(\mathbb R_+)$ and satisfies the boundary condition. Locations of the points $\nu_{j\pm}$ are determined by the conditions \cite{Kurasov-Naboko-2007}
\begin{equation}\label{intro critical points}
    k(\nu_{j+})=\pi\left(j+1-\left\{\frac{a\omega}{\pi}\right\}\right),
    \
    k(\nu_{j-})=\pi\left(j+\left\{\frac{a\omega}{\pi}\right\}\right),
    \
    j\ge0,
\end{equation}
where $k$ is the quasi-momentum of the periodic operator $\mathcal L_{per}$  and $\{\cdot\}$ is the standard fractional part function. The condition
$\frac{2a\omega}{\pi}\notin\mathbb Z$ ensures that critical points do not coincide with the endpoints of bands and with each other.

In  \cite{Naboko-Simonov-2012} we have studied the asymptotic behaviour of the spectral density of the operator $\mathcal L_{\alpha}$  near critical points for the case $\gamma=1$. In this paper we consider $\gamma\in(\frac12,1)$, and this case differs significantly. Let us state both results right away.

Let $\psi_+(x,\lambda)$ and $\psi_-(x,\lambda)$ be the Bloch solutions of the periodic equation $-\psi''(x)+q(x)\psi(x)=\lambda\psi(x)$ chosen so that they are complex conjugate to each other whenever $\lambda$ is inside a spectral band. Let $\varphi_{\alpha}(x,\lambda)$ be the solution of the Cauchy problem
\begin{equation}\label{phi}
    L\varphi_{\alpha}=\lambda\varphi_{\alpha},\ \varphi_{\alpha}(0,\lambda)=\sin\alpha,\ \varphi'_{\alpha}(0,\lambda)=\cos\alpha.
\end{equation}
Denote by $W\{\psi_+,\psi_-\}(\lambda):=\psi_+'(x,\lambda)\psi_-(x,\lambda)-\psi_+(x,\lambda)\psi_-'(x,\lambda)$ the Wronskian of two Bloch solutions. In \cite{Kurasov-Naboko-2007} Kurasov and Naboko have found the following asymptotics of generalised eigenvectors at critical points.

\begin{prop}[\cite{Kurasov-Naboko-2007}]\label{prop Kurasov-Naboko}
Let the operator $\mathcal L_{\alpha}$ be defined by \eqref{rule} and \eqref{domain} where $L$ is given by \eqref{L}, and let the conditions \eqref{conditions} hold. Let $\nu_{j\pm}$ and $\varphi_{\alpha}$ be defined by \eqref{intro critical points} and \eqref{phi}, respectively. For every critical point $\nu_{cr}\in\{\nu_{j+},\nu_{j-},j\ge0\}$ there exist $\alpha_{cr}(c,\omega,\delta,\gamma)\in[0,\pi)$, $\beta_{cr}(c,\omega,\delta,\gamma)\ge0$, $\phi_{cr}(c,\omega,\delta,\gamma)\in\mathbb R$ and $d_{cr\pm}(c,\omega,\delta,\gamma)\in\mathbb R\backslash\{0\}$ such that, as $x\rightarrow+\infty$,
\begin{multline}\label{intro phi critical - asympt gamma=1}
    \varphi_{\alpha_{cr}}(x,\nu_{cr})=d_{cr-}x^{-\beta_{cr}}
    \\
    \times
    (ie^{\frac i2\phi_{cr}}\psi_-(x,\nu_{cr})-ie^{-\frac i2\phi_{cr}}\psi_+(x,\nu_{cr})+o(1)),\text{ if }\gamma=1,
\end{multline}
\begin{multline}\label{intro phi critical - asympt gamma<1}
    \varphi_{\alpha_{cr}}(x,\nu_{cr})=d_{cr-}\exp\left(-\frac{\beta_{cr}x^{1-\gamma}}{1-\gamma}\right)
    \\
    \times
    (ie^{\frac i2\phi_{cr}}\psi_-(x,\nu_{cr})-ie^{-\frac i2\phi_{cr}}\psi_+(x,\nu_{cr})+o(1)),\text{ if }\gamma\in\left(\frac12,1\right),
\end{multline}
and for every $\alpha\neq\alpha_{cr}$
\begin{multline}\label{intro phi critical + asympt gamma=1}
    \varphi_{\alpha}(x,\nu_{cr})=d_{cr+}\sin(\alpha-\alpha_{cr})x^{\beta_{cr}}
    \\
    \times
    (e^{\frac i2\phi_{cr}}\psi_-(x,\nu_{cr})+e^{-\frac i2\phi_{cr}}\psi_+(x,\nu_{cr})+o(1)),\text{ if }\gamma=1,
\end{multline}
\begin{multline}\label{intro phi critical + asympt gamma<1}
    \varphi_{\alpha}(x,\nu_{cr})=d_{cr+}\sin(\alpha-\alpha_{cr})\exp\left(\frac{\beta_{cr}x^{1-\gamma}}{1-\gamma}\right)
    \\
    \times
    (e^{\frac i2\phi_{cr}}\psi_-(x,\nu_{cr})+e^{-\frac i2\phi_{cr}}\psi_+(x,\nu_{cr})+o(1)),\text{ if }\gamma\in\left(\frac12,1\right).
\end{multline}
\end{prop}

\begin{rem}
In fact, formulae for $\beta_{cr}$ and $\phi_{cr}$ are known:
\begin{equation}\label{beta cr}
    \beta_{j\pm}:=\frac{|c|}{2a|W\{\psi_+,\psi_-\}(\nu_{j\pm})|}
    \left|\int\limits_0^a\psi^2_{\pm}(t,\nu_{j\pm})e^{2i\omega t}dt\right|,
\end{equation}
\begin{equation}\label{phinotvar}
    \phi_{j\pm}:=\pm\left(\delta+\text{arg}\,\int_0^a\psi_{\pm}^2(t,\nu_{j\pm})e^{2i\omega t}dt\right).
\end{equation}
The expression \eqref{beta cr} was found in \cite{Naboko-Simonov-2012} and the expression \eqref{phinotvar} for the case $\gamma\in(\frac12,1)$ is given by Lemma \ref{lem reduction} of the present paper (for the case $\gamma=1$ a slight modification of this lemma is needed).
\end{rem}

In \cite{Naboko-Simonov-2012} we had the following result.

\begin{prop}[\cite{Naboko-Simonov-2012}]\label{thm gamma=1}
Let the operator $\mathcal L_{\alpha}$ be defined by \eqref{rule} and \eqref{domain} where $L$ is given by \eqref{L}, and let the conditions \eqref{conditions} hold with $\gamma=1$. Let $\rho_{\alpha}'$ be the spectral density of $\mathcal L_{\alpha}$, let $\nu_{j\pm}$ and $\varphi_{\alpha}$ be defined by \eqref{intro critical points} and \eqref{phi}, respectively. Let $\nu_{cr}\in\{\nu_{j+},\nu_{j-},j\ge0\}$ be one of the critical points and $\alpha_{cr}$ be defined in Proposition \ref{prop Kurasov-Naboko}. If $\alpha\neq\alpha_{cr}$, then there exist two non-zero one-side limits
\begin{equation*}
    \lim_{\lambda\rightarrow(\nu_{cr})^{\pm}}\frac{\rho'_{\alpha}(\lambda)}
    {|\lambda-\nu_{cr}|^{2\beta_{cr}}},
\end{equation*}
where $\beta_{cr}$ is given by the expression \eqref{beta cr}.
\end{prop}

This means that the spectral density of the operator $\mathcal L_{\alpha}$ at a critical point in the generic situation $\alpha\neq\alpha_{cr}$ has zeroes of the power type, and the power is twice the rate of decay of the subordinate solution at this critical point. The main result of the present paper is the following theorem.

\begin{thm}\label{thm gamma<1}
Let the operator $\mathcal L_{\alpha}$ be defined by \eqref{rule} and \eqref{domain} where $L$ is given by \eqref{L}, and let the conditions \eqref{conditions} hold with $\gamma\in(\frac12,1)$. Let the potential $q_1$ satisfy the estimate
\begin{equation}\label{restriction on q1}
    |q_1(x)|<\frac{c_1}{x^{1+\alpha_1}},\ x\in[0,+\infty)
\end{equation}
with some $\alpha_1,c_1>0$. Let $\nu_{j\pm}$ and $\varphi_{\alpha}$ be defined by \eqref{intro critical points} and \eqref{phi}, respectively. Let $\nu_{cr}\in\{\nu_{j+},\nu_{j-},j\ge0\}$ be one of the critical points and $\alpha_{cr}$ be defined in Proposition \ref{prop Kurasov-Naboko}. If $\alpha\neq\alpha_{cr}$, then the spectral density $\rho_{\alpha}'$ of the operator $\mathcal L_{\alpha}$ has the following asymptotics:
\begin{equation}\label{answer}
    \rho'_{\alpha}(\lambda)=\frac{a_{cr}}{d_{cr+}^2\sin^2(\alpha-\alpha_{cr})}
    \exp\left(-\frac{2c_{cr}}{|\lambda-\nu_{cr}|^{\frac{1-\gamma}{\gamma}}}\right)(1+o(1))
    \text{ as }\lambda\rightarrow\nu_{cr},
\end{equation}
where
\begin{equation}\label{c cr}
    c_{cr}:=\frac{(2\beta_{cr})^{\frac1{\gamma}}}{4\gamma}B\left(\frac32,\frac{1-\gamma}{2\gamma}\right)
    \left(\frac{a}{2\pi k'(\nu_{cr})}\right)^{\frac{1-\gamma}{\gamma}}
\end{equation}
and
\begin{multline}\label{a cr}
    a_{cr}:=\frac1{\pi|W\{\psi_+,\psi_-\}(\nu_{cr})|}
    \exp\Biggl(-\int_0^{\frac{(2\beta_{cr})^{\frac1{\gamma}}}2}
    \frac{\gamma\left(1-\sqrt{1-\frac{t^{2\gamma}}{4\beta_{cr}^2}}\right)}{t\left(1-\frac{t^{2\gamma}}{4\beta_{cr}^2}\right)}dt
    \\
    +
    \int_{\frac{(2\beta_{cr})^{\frac1{\gamma}}}2}^{(2\beta_{cr})^{\frac1{\gamma}}}
    \frac{\gamma dt}{t\sqrt{1-\frac{t^{2\gamma}}{4\beta_{cr}^2}}}
    -\text{v.p.}\int_{\frac{(2\beta_{cr})^{\frac1{\gamma}}}2}^{+\infty}
    \frac{\gamma dt}{t\left(1-\frac{t^{2\gamma}}{4\beta_{cr}^2}\right)}\Biggr),
\end{multline}
$\beta_{cr}$ is given by the expression \eqref{beta cr}, $d_{cr+}$ is defined in \eqref{intro phi critical + asympt gamma<1},
$B$ is the beta function and $k$ is the quasi-momentum of the unperturbed periodic operator $\mathcal L_{per}$.
\end{thm}

Note that the subordinate solution at a critical point decays as $\exp\left(-\frac{\beta_{cr}x^{1-\gamma}}{1-\gamma}\right)$, while the spectral density vanishes as
\begin{equation*}
    \exp\left(-\frac{(2\beta_{cr})^{\frac1{\gamma}}B\left(\frac32,\frac{1-\gamma}{2\gamma}\right)}{2\gamma}
    \left(\frac{a}{2\pi k'(\nu_{cr})|\lambda-\nu_{cr}|}\right)^{\frac{1-\gamma}{\gamma}}\right).
\end{equation*}
Relation between, on the one hand, the behaviour at infinity of the subordinate solution (whenever it exists) compared to the behaviour of a non-subordinate one and, on the other hand, the normal boundary behaviour of the Weyl's $m$-function was, in a general situation, established by Jitomirskaya--Last \cite{Jitomirskaya-Last-1996} and Remling \cite{Remling-1997}. However, in our case it yields only a trivial result: that for $\alpha\neq\alpha_{cr}$ one has $|m_{\alpha}(\nu_{cr}+i\varepsilon)|=O(1)$ and for $\alpha=\alpha_{cr}$ one has $|m_{\alpha}(\nu_{cr}+i\varepsilon)|\asymp\frac1{\varepsilon}$ as $\varepsilon\rightarrow0^+$. Spectral density is related to the boundary behaviour of the Weyl's $m$-function, and one can say in these terms that we study $\rho_{\alpha}'(\nu_{cr}+\varepsilon)=\frac1{\pi}\Im m_{\alpha}(\nu_{cr}+\varepsilon+i0)$ as $\varepsilon\rightarrow0$, which, clearly, can behave quite differently to $\Im m_{\alpha}(\nu_{cr}+i\varepsilon)$.

The case $\gamma=1$ was earlier considered in detail by Hinton--Klaus--Shaw \cite{Hinton-Klaus-Shaw-1991}, Klaus \cite{Klaus-1991} and Behncke \cite{Behncke-1991-I,Behncke-1991-II,Behncke-1994}. They included no periodic background potentials. Hinton, Klaus and Shaw considered instead an infinite sum of Wigner--von Neumann terms, while Behncke added Dirac operators into consideration. Their methods are specific to the models they study, while in \cite{Naboko-Simonov-2012} we proposed an approach which is based on reducing the problem to a certain discrete linear system with small parameter which essentially models the behaviour of solutions of the Schr\"odinger equation (we call it the ``model problem''). This allowed for stating the result in certain generality and using it to study a discrete analogue, a Jacobi matrix \cite{Janas-Simonov-2010,Simonov-2012}.

Condition \eqref{restriction on q1} in the formulation of the result is of technical nature. We would like to note here that in \cite{Hinton-Klaus-Shaw-1991} for the case $\gamma=1$ an analogous condition was imposed on the summable part $q_1$ of the potential, however, with $\alpha_1=1$. It was left as an open question whether this condition could be weakened to include every $q_1\in L_1(\mathbb R_+)$, see the question (\emph{4}) in the final section of \cite{Hinton-Klaus-Shaw-1991}. For the case $\gamma=1$ the answer is positive, as our analysis in \cite{Naboko-Simonov-2012} suggests. For the case $\gamma\in(\frac12,1)$ we do not have the answer.

In the case $\gamma\in(\frac12,1)$ we reduce the problem to a model linear differential, rather than discrete, system. In fact, it is possible to do the same in the case $\gamma=1$, and then the idea of the method of \cite{Naboko-Simonov-2012} would work and should lead to the same result as there, but in an easier way (however, on this way we would lose the discrete case). For $\gamma<1$ this idea does not produce the result anymore, but can only be used at one of the stages (in Section \ref{section I}). On the whole, the method of the present paper is different from the method of \cite{Naboko-Simonov-2012}, which is insufficient for the situation considered here.

Wigner--von Neumann potentials originally aroused interest as giving an explicit example of an eigenvalue embedded into continuous spectrum \cite{Wigner-von-Neumann-1929}. Since then they have been studied by many authors (for example, the papers \cite{Behncke-1991-I, Behncke-1991-II, Behncke-1994, Buslaev-Matveev-1970, Buslaev-Skriganov-1974, Cruz-Sampedro-Herbst-Martinez-Avendano-2002, Hinton-Klaus-Shaw-1991, Klaus-1991, Kurasov-1992, Kurasov-1996, Kurasov-Naboko-2007, Lukic-2010, \cite{Lukic-Ong-2015}, Matveev-1973, Nesterov-2007} and many more). Embedded eigenvalues created by such potentials have been observed in experiment, see \cite{Capasso-et-al-1992}. Zeroes of density divide the absolutely continuous spectrum into independent parts and for this reason are sometimes called pseudogaps.

A paper by Kreimer, Last and Simon \cite{Kreimer-Last-Simon-2009} should be also mentioned as an example of the spectral density analysis (in that case of the discrete Schr\"odinger operator with slowly decaying potential, near the endpoints of the absolutely continuous spectrum). Results on the behaviour of the spectral density have been recently used by Lukic \cite{Lukic-2013} to construct test sequences in the proof of higher-order Szeg\H{o} theorems for CMV matrices.

The paper is organized  as follows.
In Section \ref{section preliminaries} we introduce the Titchmarsh--Weyl formula for the spectral density of the operator $\mathcal L_{\alpha}$ which was proved in \cite{Kurasov-Simonov-2013}. The formula expresses $\rho_{\alpha}'(\lambda)$ in terms of asymptotic coefficients of the solution $\varphi_{\alpha}(x,\lambda)$ as $x\rightarrow+\infty$.
In Section \ref{section reduction} we rewrite the eigenfunction equation of $\mathcal L_{\alpha}$ as a model system with a small parameter $\varepsilon$ and express the spectral density in terms of a certain solution of that system.
In Section \ref{section the model problem} we transform the model system to a new form  and determine five regions of the positive coordinate half-line where asymptotic analysis of solutions as $\varepsilon\rightarrow0$ should be carried out in different ways.
In Sections \ref{section I}--\ref{section III} we consider each of the regions separately and find asymptotics of solutions of the model system.
In Section \ref{section matching} we match the results in order to obtain the double asymptotics, in the coordinate and in the small parameter, of the solution of the model system in terms of which the spectral density of $\mathcal L_{\alpha}$ was earlier expressed.
In Section \ref{section proof} we prove the main result of the paper, Theorem \ref{thm gamma<1}.

We denote by $M^{2\times2}(\mathbb R)$ and $M^{2\times2}(\mathbb C)$ matrices of two rows and two columns with real and, respectively, complex entries and use the following notation for two basic vectors in $\mathbb C^2$:
\begin{equation}\label{e +-}
    e_+:=
    \left(
      \begin{array}{c}
        1 \\
        0 \\
      \end{array}
    \right),\
    e_-:=
    \left(
      \begin{array}{c}
        0 \\
        1 \\
      \end{array}
    \right).
\end{equation}

\section{Preliminaries}\label{section preliminaries}
The spectrum of the periodic Schr\"odinger operator $\mathcal L_{per}$ consists of infinitely many intervals \eqref{prelim spectrum of periodic operator}, see \cite[Theorem 2.3.1]{Eastham-1973}, where $ \lambda_j $ and $ \mu_j $ are the eigenvalues of the Sturm--Liouville problem defined by the differential expression $L$ on the interval $[0,a]$ with, respectively, periodic and antiperiodic boundary conditions.  Let us denote by $\partial$ the set of the endpoints of the spectral bands of $\mathcal L_{per}$,
\begin{equation}\label{prelim boundary of the spectrum}
    \partial:=\{\lambda_j,\mu_j,\ j\ge0\},
\end{equation}
including the case when the endpoints of the neighbouring bands coincide. Spectral properties of the operator $\mathcal L_{per}$ are related to the quasi-momentum
\begin{equation*}
    k(\lambda):=-i\ln\left(\frac{D(\lambda)+\sqrt{D^2(\lambda)-4}}2\right),
\end{equation*}
where the entire function $D$ is the trace of the monodromy matrix of the periodic equation. We can choose the branch of $k$ so that
\begin{equation*}
\begin{array}{c}
	k(\lambda_0)=0,k(\mu_0)=k(\mu_1)=\pi,k(\lambda_1)=k(\lambda_2)=2\pi,...,
	\\
    \begin{array}{rl}
    k(\lambda)\in\mathbb R\text{ and }k'(\lambda)>0,&\text{if }\lambda\in\sigma(\mathcal L_{per})\backslash\partial,
    \\
    k(\lambda)\in\mathbb C_+,&\text{if }\lambda\in\mathbb C_+,
    \end{array}
    \end{array}
\end{equation*}
see \cite[Theorem 2.3.1]{Eastham-1973}. The eigenfunction equation of $\mathcal L_{per}$,
\begin{equation*}
    -\psi''(x)+q(x)\psi(x)=\lambda\psi(x),
\end{equation*}
has two solutions $\psi_+(x,\lambda)$ and $\psi_-(x,\lambda)$ (Bloch solutions) which satisfy the quasi-periodic conditions:
\begin{equation}\label{prelim Bloch solutions}
    \psi_{\pm}(x+a,\lambda)\equiv e^{\pm ik(\lambda)}\psi_{\pm}(x,\lambda),
\end{equation}
Each of them is defined uniquely up to multiplication by a coefficient which depends on $\lambda$.
It is possible to choose these coefficients so that for every $x\ge0$ the functions $\psi_{\pm}(x,\cdot)$ and $\psi_{\pm}'(x,\cdot)$, and hence their Wronskian $W\{\psi_+,\psi_-\}(\cdot)$, are analytic in $\mathbb C_+$ and continuous up to the set
$\sigma(\mathcal L_{per})\backslash\partial$. Moreover, it is possible for every $\lambda\in\sigma(\mathcal L_{per})\backslash\partial$ to have
\begin{equation*}
    \psi_+(x,\lambda)\equiv\overline{\psi_-(x,\lambda)}\text{ and }iW\{\psi_+,\psi_-\}(\lambda)<0.
\end{equation*}
In what follows we assume that such a choice of the coefficients is made. Bloch solutions can be written in another form:
\begin{equation*}
    \psi_{\pm}(x,\lambda)=e^{\pm ik(\lambda)\frac xa}p_{\pm}(x,\lambda),
\end{equation*}
where the functions $p_+(x,\lambda)$ and $p_-(x,\lambda)$ are periodic with the period $a$ in the variable $x$ and have the same properties as $\psi_{\pm}(x,\lambda)$ in the variable $\lambda$.

It is well known that some of spectral properties of one-dimensional Schr\"odinger operators can be written in terms of the asymptotic behaviour of their generalised eigenvectors (see, for example, \cite{Gilbert-Pearson-1987}). In particular, the spectral density of the operator can be expressed in terms of the Jost function by the Titchmarsh--Weyl formula. The eigenfunction equation $L\psi=\lambda\psi$ is a small perturbation of the periodic equation $-\psi''(x)+q(x)\psi(x)=\lambda\psi(x)$ in the sense that asymptotically $\varphi_{\alpha}$ is some linear combination of two Bloch solutions.

\begin{prop}[\cite{Kurasov-Simonov-2013}]\label{prop Titchmarsh--Weyl formula}
Let the operator $\mathcal L_{\alpha}$ be defined by \eqref{rule} and \eqref{domain} where $L$ is given by \eqref{L}, and let the conditions \eqref{conditions} hold. Let $\rho_{\alpha}'$ be the spectral density of the operator $\mathcal L_{\alpha}$, and $\nu_{j\pm}$, $\varphi_{\alpha}$, $\sigma(\mathcal L_{per})$, $\partial$, $\psi_{\pm}$ be defined by \eqref{intro critical points}, \eqref{phi}, \eqref{prelim spectrum of periodic operator}, \eqref{prelim boundary of the spectrum}, \eqref{prelim Bloch solutions}, respectively. For every fixed $\lambda\in\sigma(\mathcal L_{per})\backslash(\partial\cup\{\nu_{j+},\nu_{j-},j\ge0\})$ there exists $A_{\alpha}(\lambda)$ such that, as $x\rightarrow+\infty$,
\begin{equation}\label{asymptotics of phi-alpha}
\begin{array}{l}
    \varphi_{\alpha}(x,\lambda)=A_{\alpha}(\lambda)\psi_-(x,\lambda)+
    \overline{A_{\alpha}(\lambda)}\psi_+(x,\lambda)+o(1),
    \\
    \varphi_{\alpha}'(x,\lambda)=A_{\alpha}(\lambda)\psi_-'(x,\lambda)+
    \overline{A_{\alpha}(\lambda)}\psi_+'(x,\lambda)+o(1),
\end{array}
\end{equation}
and the following equality (the Titchmarsh--Weyl formula) holds:
\begin{equation}\label{Titchmatsh-Weyl formula}
    \rho'_{\alpha}(\lambda)=\frac1{2\pi|W\{\psi_+,\psi_-\}(\lambda)| \; |A_{\alpha}(\lambda)|^2}.
\end{equation}
\end{prop}

Here $A_{\alpha}$ is essentially the Jost function $M_{\alpha}$ (up to a coefficient: $A_{\alpha}(\lambda)=-\frac{M_{\alpha}(\lambda)}{W\{\psi_+,\psi_-\}(\lambda)}$). This result is a generalisation of the classic Titchmarsh--Weyl (or Kodaira) formula \cite{Kodaira-1949,Titchmarsh-1946-2}. A variant of this formula for the Schr\"odinger operator with Wigner--von Neumann potential without the periodic background ($q(x)\equiv0$) follows from results of \cite{Matveev-1973}. In the case of discrete Schr\"odinger operator with the Wigner--von Neumann potential an analogous formula is also known, see \cite{Damanik-Simon-2006,Janas-Simonov-2010}.

\section{Reduction to the model problem}\label{section reduction}
In this section we reduce the eigenfunction equation $L\psi=\lambda\psi$ to a specially constructed linear differential system and express the modulus of the coefficient $A_{\alpha}$ from the Titchmarsh--Weyl formula \eqref{Titchmatsh-Weyl formula} in terms of a certain solution of this new system.

\begin{lem}\label{lem reduction}
Let the conditions of Theorem \ref{thm gamma<1} hold and let $\nu_{cr}\in\{\nu_{j+},\nu_{j-},j\ge0\}$ be one of the critical points.
\\
\textit{1.}
There exist the following objects which are determined by the data of the problem ($q$, $q_1$, $c$, $\omega$, $\gamma$, $\delta$ and $\alpha$) and have the following properties:
\\
$\cdot$ the neighbourhood $U_{cr}$ of the point $\nu_{cr}$ such that its closure lies inside the spectral band and does not contain the second critical point of that band,
\\
$\cdot$ the bijective real-valued function $\varepsilon_{cr}(\lambda)$ such that
\begin{equation}\label{epsilon 0}
    \varepsilon_{cr}(\lambda)=\frac{2\pi k'(\nu_{cr})}{a}(\lambda-\nu_{cr})+O((\lambda-\nu_{cr})^2)\text{ as }\lambda\rightarrow\nu_{cr},
\end{equation}
\\
$\cdot$ the $M^{2\times2}(\mathbb R)$-valued function $R_{cr}(x,\lambda)$ such that for every $x\in[0,+\infty)$ it is continuous in $U_{cr}$ as a function of $\lambda$, and for every $\lambda\in U_{cr}$ and $x\in[0,+\infty)$ satisfies the estimate
\begin{equation*}
    \|R_{cr}(x,\lambda)\|<c_2\left(|q_1(x)|+\frac1{(x^2+x)^{\gamma}}\right)
\end{equation*}
with some $c_2>0$,
\\
$\cdot$ the vector $g_{cr,\alpha}\in\mathbb R^2\backslash\{0\}$,
\\
$\cdot$ the solution $w_{cr,\alpha}(x,\lambda)$ of the linear differential system
\begin{equation}\label{system w}
    w_{cr}'(x)=
    \left(
    \frac{\beta_{cr}}{x^{\gamma}}
    \left(
    \begin{array}{cc}
    \cos(\varepsilon_{cr}(\lambda)x) & \sin(\varepsilon_{cr}(\lambda)x)
    \\
    \sin(\varepsilon_{cr}(\lambda)x) & -\cos(\varepsilon_{cr}(\lambda)x)
    \end{array}
    \right)
    +
    R_{cr}(x,\lambda)
    \right)
    w_{cr}(x),
\end{equation}
where $\beta_{cr}$ is given by \eqref{beta cr}, and such that $w_{cr,\alpha}(0,\cdot)$ is a continuous function in $U_{cr}$ and
$w_{cr,\alpha}(0,\nu_{cr})=g_{cr,\alpha}$.
\\
\textit{2.}
For every $\lambda\in U_{cr}\backslash\{\nu_{cr}\}$ the limit
$\lim\limits_{x\rightarrow+\infty}w_{cr,\alpha}(x,\lambda)$ exists
and
\begin{equation}\label{A=eta}
    |A_{\alpha}(\lambda)|=\left\|\lim_{x\rightarrow+\infty}w_{cr,\alpha}(x,\lambda)\right\|,
\end{equation}
where $A_{\alpha}$ is defined in Proposition \ref{prop Titchmarsh--Weyl formula}.
\\
\textit{3.}
With the same $\alpha_{cr}$, $d_{cr-}$ and $d_{cr+}$ as in Proposition \ref{prop Kurasov-Naboko} and with $\phi_{cr}$ given by \eqref{phinotvar} one has
\begin{equation}\label{asympt w-}
    w_{cr,\alpha_{cr}}(x,\nu_{cr})=
    d_{cr-}
    \exp\left(-\frac{\beta_{cr}x^{1-\gamma}}{1-\gamma}\right)
    (e_-+o(1))
\end{equation}
as $x\rightarrow+\infty$ and the asymptotics \eqref{intro phi critical - asympt gamma<1} holds. For every $\alpha\neq\alpha_{cr}$ one has
\begin{equation}\label{asympt w+}
    w_{cr,\alpha}(x,\nu_{cr})=
    d_{cr+}\sin(\alpha-\alpha_{cr})
    \exp\left(\frac{\beta_{cr}x^{1-\gamma}}{1-\gamma}\right)
    (e_++o(1))
\end{equation}
as $x\rightarrow+\infty$ and the asymptotics \eqref{intro phi critical + asympt gamma<1} holds. Vectors $e_{\pm}$ are defined by \eqref{e +-}.
\end{lem}

\begin{proof}\textit{1--2.}
Let us start with the eigenfunction equation $L\psi=\lambda\psi$ and write it in the vector form,
\begin{equation*}
\left(
\begin{array}{c}
\psi(x)
\\
\psi'(x)
\end{array}
\right)'
=
\left(
  \begin{array}{cc}
    0 & 1
    \\
    q(x)+\frac{c\sin(2\omega x+\delta)}{x^{\gamma}}+q_1(x)-\lambda & 0
  \end{array}
\right)
\left(
  \begin{array}{c}
    \psi(x)
    \\
    \psi'(x)
  \end{array}
\right).
\end{equation*}
Let us make variation of parameters:
\begin{equation}\label{eta}
    \left(%
    \begin{array}{c}
    \psi(x) \\
    \psi'(x) \\
    \end{array}%
    \right)
    =
    \left(%
    \begin{array}{cc}
    \psi_-(x,\lambda) & \psi_+(x,\lambda) \\
    \psi_-'(x,\lambda) & \psi_+'(x,\lambda) \\
    \end{array}%
    \right)
    \widehat w(x).
\end{equation}
This leads to the system
\begin{equation}\label{system w hat}
    \widehat w\,'(x)=
    \frac{\frac{c\sin(2\omega x+\delta)}{x^{\gamma}}+q_1(x)}{W\{\psi_+,\psi_-\}(\lambda)}
    \left(%
    \begin{array}{cc}
    -\psi_+(x,\lambda)\psi_-(x,\lambda) & -\psi_+^2(x,\lambda) \\
    \psi_-^2(x,\lambda) & \psi_+(x,\lambda)\psi_-(x,\lambda) \\
    \end{array}%
    \right)
    \widehat w(x).
\end{equation}
Denote the summable part of its coefficient matrix as
\begin{equation}\label{R hat}
    \widehat R(x,\lambda):=\frac{q_1(x)}{W\{\psi_+,\psi_-\}(\lambda)}
    \left(%
    \begin{array}{cc}
    -\psi_+(x,\lambda)\psi_-(x,\lambda) & -\psi_+^2(x,\lambda) \\
    \psi_-^2(x,\lambda) & \psi_+(x,\lambda)\psi_-(x,\lambda) \\
    \end{array}%
    \right).
\end{equation}
For the remaining part we use Fourier series decompositions of the periodic functions $p_+p_-,p_+^2$ and $p_-^2=\overline{p_+^2}$ to write the entries of the matrix as follows:
\begin{equation}
\begin{array}{l}
    \psi_+(x,\lambda)\psi_-(x,\lambda)=\sum\limits_{n\in\mathbb Z}b_n(\lambda)e^{2i\pi n\frac xa},
    \\
    \psi_+^2(x,\lambda)=\sum\limits_{n\in\mathbb Z}b_n^+(\lambda)e^{2i(\pi n+k(\lambda))\frac xa},
    \\
    \psi_-^2(x,\lambda)=\sum\limits_{n\in\mathbb Z}\overline{b_{-n}^+(\lambda)}e^{2i(\pi n-k(\lambda))\frac xa}.
\end{array}
\end{equation}
Now let us choose some band of the spectrum with the index $j$ and one of two critical points in it, $\nu_{cr}=\nu_{j+}$. We give the detailed proof for the choice of the sign ``$+$'', and for the sign ``$-$'' formulae should be modified in a natural way. Take the neighbourhood of the critical point $U_{j+}$ so that its closure lies inside the band with the index $j$ and does not contain the point $\nu_{j-}$. Let
\begin{equation*}
    n_{j+}:=-\left(j+1+\left\lfloor\frac{a\omega}{\pi}\right\rfloor\right),
\end{equation*}
so that
\begin{equation*}
    2i(\pi n_{j+}+k(\lambda))x+2ia\omega x=2i(k(\lambda)+\pi n_{j+}+a\omega)x,
\end{equation*}
and
\begin{equation}\label{n j+}
    \pi n_{j+}+a\omega=-\pi\left(j+1-\left\{\frac{a\omega}{\pi}\right\}\right)=-k(\nu_{j+}).
\end{equation}
Then we can write
\begin{multline*}
    \frac{c\sin(2\omega x+\delta)}{x^{\gamma}W\{\psi_+,\psi_-\}(\lambda)}
    \left(%
    \begin{array}{cc}
    -\psi_+(x,\lambda)\psi_-(x,\lambda) & -\psi_+^2(x,\lambda) \\
    \psi_-^2(x,\lambda) & \psi_+(x,\lambda)\psi_-(x,\lambda) \\
    \end{array}%
    \right)
\\
=\frac{c(e^{2i\omega x+i\delta}-e^{-2i\omega x-i\delta})}{2ix^{\gamma}W\{\psi_+,\psi_-\}(\lambda)}
\sum_{n\in\mathbb Z}
\left(
  \begin{array}{cc}
    -b_n(\lambda)e^{2i\pi n\frac xa}
    &
    -b_n^+(\lambda)e^{2i(\pi n+k(\lambda))\frac xa}
    \\
    \overline{b_n^+(\lambda)}e^{-2i(\pi n+k(\lambda))\frac xa}
    &
    b_n(\lambda)e^{2i\pi n\frac xa}\\
  \end{array}
\right)
\\
=\sum_{n\in\mathbb Z}S_{n}(x,\lambda)=S_{j+}^{\,(1)}(x,\lambda)+S_{j+}^{\,(2)}(x,\lambda),
\end{multline*}
where
\begin{multline}\label{S n}
    S_{n}(x,\lambda):=\frac{c(e^{2i\omega x+i\delta}-e^{-2i\omega x-i\delta})}{2ix^{\gamma}W\{\psi_+,\psi_-\}(\lambda)}
    \\
    \times
    \left(
    \begin{array}{cc}
    -b_n(\lambda)e^{2i\pi n\frac xa}
    &
    -b_n^+(\lambda)e^{2i(\pi n+k(\lambda))\frac xa}
    \\
    \overline{b_n^+(\lambda)}e^{-2i(\pi n+k(\lambda))\frac xa}
    &
    b_n(\lambda)e^{2i\pi n\frac xa}\\
    \end{array}
    \right),
\end{multline}
\begin{multline}\label{S j+ 1}
    S_{j+}^{\,(1)}(x,\lambda):=\frac{c}{2ix^{\gamma}W\{\psi_+,\psi_-\}(\nu_{j+})}
    \\
    \times
    \left(
  \begin{array}{cc}
    0
    &
    -b_{n_{j+}}^+(\nu_{j+})e^{i\delta}e^{2i\pi(k(\lambda)-k(\nu_{j+}))\frac xa}
    \\
    -\overline{b_{n_{j+}}^+(\nu_{j+})e^{i\delta}}e^{-2i\pi(k(\lambda)-k(\nu_{j+}))\frac xa}
    &
    0
  \end{array}
\right).
\end{multline}
and
\begin{equation}\label{S j+2}
    S_{j+}^{\,(2)}(x,\lambda):=\sum_{n\in\mathbb Z\backslash\{n_{j+}\}}S_{n}(x,\lambda)+(S_{n_{j+}}(x,\lambda)-S_{j+}^{\,(1)}(x,\lambda)).
\end{equation}
In this notation the system \eqref{system w hat} reads as
\begin{equation}\label{system w hat 2}
    \widehat w\,'=(S^{\,(1)}_{j,+}+S^{\,(2)}_{j,+}+\widehat R)\,\widehat w.
\end{equation}
To eliminate the non-resonating term $S^{\,(2)}_{j,+}$ from this system we use the Harris--Lutz transformation based on the matrix
\begin{equation}\label{T w}
    \widehat T_{j+}(x,\lambda)=-\int_x^{+\infty}S_{j+}^{\,(2)}(t,\lambda)dt.
\end{equation}
First we need to see that this integral is convergent.

\begin{lem}\label{lem reduction estimate of T}
The integral in the definition \eqref{T w} is convergent, the function $\widehat T_{j+}(x,\lambda)$ is continuous in $\lambda$ for every $x\in[0,+\infty)$ and satisfies the estimate
\begin{equation*}
    \|\widehat T_{j+}(x,\lambda)\|\le\frac{c_3}{(x+1)^{\gamma}}
\end{equation*}
for every $x\in[0,+\infty)$ and $\lambda\in U_{j+}$ with some $c_3>0$.
\end{lem}

\begin{proof}
Coefficients $b_n$ and $b_n^+$ have the same analyticity properties as the function $k$, and satisfy the following estimates (see, for example, \cite{Kurasov-Simonov-2013}): there exists $c_4>0$ such that for every $\lambda\in U_{j+}$ and $n\in\mathbb Z$
\begin{equation}\label{estimate for S 1}
    |b_n(\lambda)|,|b_n^+(\lambda)|<\frac{c_4}{n^2+1}.
\end{equation}
One can choose, if necessary, $c_4$ large enough to ensure that for every $\lambda\in U_{j+}$
\begin{equation}\label{estimate for S 2}
    \frac1{|W\{\psi_+,\psi_-\}(\lambda)|}<c_4.
\end{equation}
We also need the following rough estimate: if $N_1<N_2$, then
\begin{equation}\label{estimate for S 3}
    \left|\int_{N_1}^{N_2}\frac{e^{i\xi t}dt}{t^{\gamma}}\right|\le\frac{2^{\gamma}\left(\frac2{|\xi|}+\frac1{1-\gamma}\right)}{(N_1+1)^{\gamma}}.
\end{equation}
To see this let us consider three cases.
\\\textit{1.}
If $0\le N_1<N_2\le1$, then
\begin{equation*}
    \left|\int_{N_1}^{N_2}\frac{e^{i\xi t}dt}{t^{\gamma}}\right|\le\int_{0}^{1}\frac{dt}{t^{\gamma}}
    =\frac1{1-\gamma}\le\frac{2^{\gamma}\left(\frac2{|\xi|}+\frac1{1-\gamma}\right)}{(N_1+1)^{\gamma}}.
\end{equation*}
\textit{2.}
If $1\le N_1<N_2$, then
\begin{equation*}
    \left|\int_{N_1}^{N_2}\frac{e^{i\xi t}dt}{t^{\gamma}}\right|\le\frac2{|\xi|N_1^{\gamma}}
    \le\frac{2^{\gamma}\left(\frac2{|\xi|}+\frac1{1-\gamma}\right)}{(N_1+1)^{\gamma}}
\end{equation*}
from integrating by parts.
\\\textit{3.}
If $0\le N_1<1<N_2$, then
\begin{equation*}
    \left|\int_{N_1}^{N_2}\frac{e^{i\xi t}dt}{t^{\gamma}}\right|
    \le\left|\int_{N_1}^1\frac{e^{i\xi t}dt}{t^{\gamma}}\right|+\left|\int_{1}^{N_2}\frac{e^{i\xi t}dt}{t^{\gamma}}\right|
    \le\frac1{1-\gamma}+\frac2{|\xi|}
    \le\frac{2^{\gamma}\left(\frac2{|\xi|}+\frac1{1-\gamma}\right)}{(N_1+1)^{\gamma}}
\end{equation*}
using the intermediate estimates from the first case for the first summand and from the second case for the second.

Using the estimates \eqref{estimate for S 1}, \eqref{estimate for S 2} and \eqref{estimate for S 3} we see that there exists $c_5>0$ such that for every $\lambda\in U_{j+}, n\in\mathbb Z$ and $N_1<N_2$
\begin{equation*}
    \left\|\int_{N_1}^{N_2}S_{n}(t,\lambda)dt\right\|
    \le\frac{c_5}{(n^2+1)(N_1+1)^{\gamma}\min\limits_{\lambda\in U_{j+}}\{|\pi n\pm a\omega|, |k(\lambda)+\pi n\pm a\omega|\}}.
\end{equation*}
This estimate works for $n\neq n_{j+}$, because
\begin{equation*}
    \min\limits_{n\in\mathbb Z\backslash\{n_{j+}\},\lambda\in U_{j+}}\{|\pi n\pm a\omega|, |k(\lambda)+\pi n\pm a\omega|\}>0,
\end{equation*}
while for $n=n_{j+}$ one has
\begin{equation*}
    \min\limits_{\lambda\in U_{j+}}\{|k(\lambda)+\pi n_{j+}\pm a\omega|\}=|k(\nu_{j+})+\pi n_{j+}\pm a\omega|=0.
\end{equation*}
One can choose $c_5$ so large that for every $\lambda\in U_{j+}$ and $N_1<N_2$ the estimate
\begin{equation*}
    \left\|\int_{N_1}^{N_2}S_{n}(t,\lambda)dt\right\|\le\frac{c_5}{(n^2+1)(N_1+1)^{\gamma}}
\end{equation*}
holds for every $n\in\mathbb Z\backslash\{n_{j+}\}$, and so with some $c_6>0$ one has
\begin{equation}\label{estimate almost S 2}
    \left\|\int_{N_1}^{N_2}\sum_{n\in\mathbb Z\backslash\{n_{j+}\}}S_{n}(t,\lambda)dt\right\|\le\frac{c_6}{(N_1+1)^{\gamma}}.
\end{equation}
In order to estimate $S_{j+}^{\,(2)}(x,\lambda)$  consider the difference
\begin{multline*}
    S_{n_{j+}}(x,\lambda)-S_{j+}^{\,(1)}(x,\lambda)
    \\
    =\frac{ce^{i(2\omega x+\delta)}}{2ix^{\gamma}W\{\psi_+,\psi_-\}(\lambda)}
    \left(
    \begin{array}{cc}
    -b_{n_{j+}}(\lambda)e^{2i\pi n_{j+}\frac xa}
    &
    0
    \\
    \overline{b_{n_{j+}}^+(\lambda)}e^{-2i(\pi n_{j+}+k(\lambda))\frac xa}
    &
    b_{n_{j+}}(\lambda)e^{2i\pi n_{j+}\frac xa}
    \end{array}
\right)
\\
+  \frac{ce^{-i(2\omega x+\delta)}}{2ix^{\gamma}W\{\psi_+,\psi_-\}(\lambda)}
    \left(
    \begin{array}{cc}
    b_{n_{j+}}(\lambda)e^{2i\pi n_{j+}\frac xa}
    &
    b_{n_{j+}}^+(\lambda)e^{2i(\pi n_{j+}+k(\lambda))\frac xa}
    \\
    0
    &
    -b_{n_{j+}}(\lambda)e^{2i\pi n_{j+}\frac xa}
    \end{array}
\right)
\\
+\frac{ce^{i\delta}}{2ix^{\gamma}}
\left(\frac{b_{n_{j+}}^+(\nu_{j+})}{W\{\psi_+,\psi_-\}(\nu_{j+})}-\frac{b_{n_{j+}}^+(\lambda)}{W\{\psi_+,\psi_-\}(\lambda)}\right)
e^{2i\pi(k(\lambda)-k(\nu_{j+}))\frac xa}
\left(
  \begin{array}{cc}
    0 & 1 \\
    0 & 0 \\
  \end{array}
\right)
\\
+\frac{ce^{-i\delta}}{2ix^{\gamma}}
\left(\frac{\overline{b_{n_{j+}}^+(\nu_{j+})}}{W\{\psi_+,\psi_-\}(\nu_{j+})}-\frac{\overline{b_{n_{j+}}^+(\lambda)}}{W\{\psi_+,\psi_-\}(\lambda)}\right)
e^{-2i\pi(k(\lambda)-k(\nu_{j+}))\frac xa}
\left(
  \begin{array}{cc}
    0 & 0 \\
    1 & 0 \\
  \end{array}
\right).
\end{multline*}
The first two summands can be estimated using \eqref{estimate for S 1}--\eqref{estimate for S 3} and added into \eqref{estimate almost S 2} under the integral with, possibly, a change of $c_6$. The integral of the third and the fourth summands from $N_1$ to $N_2$ can be estimated using \eqref{estimate for S 2} and \eqref{estimate for S 3} by
\begin{equation*}
    \frac{c_7\left|\frac{b_{n_{j+}}^+(\nu_{j+})}{W\{\psi_+,\psi_-\}(\nu_{j+})}-\frac{b_{n_{j+}}^+(\lambda)}{W\{\psi_+,\psi_-\}(\lambda)}\right|}
    {|k(\lambda)-k(\nu_{j+})|(N_1+1)^{\gamma}}
\end{equation*}
with some $c_7>0$. This can in turn be estimated by $\frac{c_8}{(N_1+1)^{\gamma}}$ with some $c_8>0$, since both functions $\frac{b_n^+}{W\{\psi_+,\psi_-\}}$ and $k$ are differentiable at the point $\nu_{j+}$ and since $k\,'(\nu_{j+})\neq0$. Combining this and \eqref{estimate almost S 2} we see that there exists $c_8>0$ such that for every $\lambda\in U_{j+}$ and $N_1<N_2$ we have the estimate
\begin{equation}\label{estimate S 2}
        \left\|\int_{N_1}^{N_2}S_{j+}^{\,(2)}(t,\lambda)dt\right\|\le\frac{c_3}{(N_1+1)^{\gamma}}
\end{equation}
with some $c_3>0$. This means that the integral $\int_{x}^{+\infty}S_{j+}^{\,(2)}(t,\lambda)dt$ exists for every $x\in[0,+\infty)$, the definition \eqref{T w} of $\widehat T_{j+}$ is correct and that $\widehat T_{j+}(x,\lambda)$ is continuous in $\lambda$ and is estimated in norm by $\frac{c_3}{(x+1)^{\gamma}}$.
\end{proof}

From now on we start writing the index ``cr'' instead of ``$j+$'', because all the formulae remain valid  if one changes the index to ``$j-$'', that is, for the second type of critical points. Let us make the Harris-Lutz transformation by substituting
\begin{equation}\label{eta widehat}
    \widehat w(x)=\exp(\widehat T_{cr}(x,\lambda))\widetilde w_{cr}(x)
\end{equation}
to the system \eqref{system w hat}. This leads, with the use of the equality $\widehat T_{cr}'=S^{\,(2)}_{cr}$, to the system
\begin{equation*}
    \widetilde w_{cr}'=e^{-\widehat T_{cr}}((S^{\,(1)}_{cr}+S^{\,(2)}_{cr}+\widehat R)e^{\widehat T_{cr}}
    -(e^{\widehat T_{cr}})')\widetilde w_{cr}=(S^{\,(1)}_{cr}+\widetilde R_{cr})\widetilde w_{cr}
\end{equation*}
with
\begin{equation}\label{R tilde}
    \widetilde R_{cr}=e^{-\widehat T_{cr}}(S^{\,(1)}_{cr}+S^{\,(2)}_{cr})e^{\widehat T_{cr}}
    -(S^{\,(1)}_{cr}+S^{\,(2)}_{cr})+e^{-\widehat T_{cr}}\widehat Re^{\widehat T_{cr}}
    -e^{-\widehat T_{cr}}((e^{\widehat T_{cr}})'-\widehat T_{cr}').
\end{equation}
Let us estimate this remainder. For every $\lambda\in U_{cr}$ and $x\in(0,+\infty)$ we have $\|S^{\,(1)}_{cr}+S^{\,(2)}_{cr}\|<\frac{c_9}{x^{\gamma}}$ with some $c_9>0$ and $\|\widehat T_{cr}(x,\lambda)\|<\frac{c_3}{(x+1)^{\gamma}}$. Therefore
\begin{equation*}
    \left\|e^{-\widehat T_{cr}(x,\lambda)}(S^{\,(1)}_{cr}(x,\lambda)+S^{\,(2)}_{cr}(x,\lambda))e^{\widehat T_{cr}(x,\lambda)}
    -(S^{\,(1)}_{cr}(x,\lambda)+S^{\,(2)}_{cr}(x,\lambda))\right\|
    <\frac{c_{10}}{(x^2+x)^{\gamma}}
\end{equation*}
and
\begin{multline*}
    \left\|(e^{\widehat T_{cr}(x,\lambda)})'-\widehat T_{cr}'(x,\lambda)\right\|
    =\left\|\sum_{n=2}^{\infty}\frac{(\widehat T_{cr}^n(x,\lambda))'}{n!}\right\|
    \le\left\|\widehat T_{cr}'(x,\lambda)\right\|\sum_{n=2}^{\infty}\frac{\|\widehat T_{cr}(x,\lambda)\|^{n-1}}{(n-1)!}
    \\
    =\|S^{\,(2)}_{cr}(x,\lambda)\|(e^{\|\widehat T_{cr}(x,\lambda)\|}-1)
    <\frac{c_{11}}{(x^2+x)^{\gamma}}.
\end{multline*}
Combining all this and the estimate of the summable term \eqref{R hat} we see that there exists $c_{12}>0$ such that for every $\lambda\in U_{cr}$ and $x\in[0,+\infty)$ one has
\begin{equation*}
    \|\widetilde R_{cr}(x,\lambda)\|<c_{12}\left(\frac1{(x^2+x)^{\gamma}}+|q_1(x)|\right).
\end{equation*}
Now we can rewrite $S_{cr}^{\,(1)}$ so that the system on $\widetilde w_{cr}$ reads:
\begin{equation}\label{system w tilde}
    \widetilde w\,'_{cr}(x)=\left(\frac1{x^{\gamma}}
    \left(
      \begin{array}{cc}
        0 & z_{cr}e^{i\varepsilon_{cr}(\lambda)} \\
        \overline{z_{cr}}e^{-i\varepsilon_{cr}(\lambda)} & 0 \\
      \end{array}
    \right)
    +\widetilde R_{cr}(x,\lambda)
    \right)
    \widetilde w(x),
\end{equation}
where
\begin{equation*}
    \varepsilon_{cr}(\lambda):=\frac{2\pi(k(\lambda)-k(\nu_{cr}))}a
\end{equation*}
and
\begin{equation}\label{z cr}
    z_{cr}:=-\frac{cb_{n_{cr}}^+(\nu_{cr})e^{i\delta}}{2iW\{\psi_+,\psi_-\}(\nu_{cr})}.
\end{equation}
One can check that
\begin{equation*}
    |z_{cr}|=\beta_{cr},\ \text{arg}\,z_{cr}=\phi_{cr},
\end{equation*}
with $\beta_{cr}$ and $\phi_{cr}$ given by \eqref{beta cr} and \eqref{phinotvar}, by substituting the expression for the Fourier coefficient
\begin{equation*}
    b_{n}^+(\lambda)=\frac1a\int_0^a\psi_+^2(t,\lambda)e^{-2i(k(\lambda)+\pi n)\frac ta}dt
\end{equation*}
into \eqref{z cr} and using the relation \eqref{n j+} involving $n_{cr}$.

Let us substitute $\widetilde w_{cr}$ in the form
\begin{equation*}
    \widetilde w_{cr}(x)=
    \left(
      \begin{array}{cc}
        e^{\frac i2\phi_{cr}} & 0 \\
        0 & e^{-\frac i2\phi_{cr}} \\
      \end{array}
    \right)
    \left(
      \begin{array}{cc}
        1 & i \\
        1 & -i \\
      \end{array}
    \right)
    w_{cr}(x),
\end{equation*}
into the system \eqref{system w tilde}. We get for $w_{cr}$ the system \eqref{system w}, and
\begin{multline}\label{R cr}
    R_{cr}(x,\lambda)=
    \left(
      \begin{array}{cc}
        1 & i \\
        1 & -i \\
      \end{array}
    \right)^{-1}
    \left(
      \begin{array}{cc}
        e^{-\frac i2\phi_{cr}} & 0 \\
        0 & e^{\frac i2\phi_{cr}} \\
      \end{array}
    \right)
    \\
    \times
    \widetilde R_{cr}(x,\lambda)
    \left(
      \begin{array}{cc}
        e^{\frac i2\phi_{cr}} & 0 \\
        0 & e^{-\frac i2\phi_{cr}} \\
      \end{array}
    \right)
    \left(
      \begin{array}{cc}
        1 & i \\
        1 & -i \\
      \end{array}
    \right).
\end{multline}
This expression for every $x\in[0,+\infty)$ is continuous in $U_{cr}$ as a function of $\lambda$. For every $\lambda\in U_{cr}$ and $x\in[0,+\infty)$ it can be estimated in norm by $c_2\left(|q_1(x)|+\frac1{(x^2+x)^{\gamma}}\right)$ with some $c_2>0$.

Let us check that the matrix $R_{cr}$ has real entries. Fix some $x\in[0,\infty)$ and $\lambda\in U_{cr}$. From the expression \eqref{R hat} for $\widehat R$ using that $q_1(x)\in\mathbb R$ and $iW\{\psi_+,\psi_-\}(\lambda)<0$ we see that the matrix $\widehat R(x,\lambda)$ has the following conjugation property:
\begin{equation*}
    \widehat R_{11}(x,\lambda)=\overline{\widehat R_{22}(x,\lambda)},\ \widehat R_{12}(x,\lambda)=\overline{\widehat R_{21}(x,\lambda)}.
\end{equation*}
This property is preserved for sums or products as well as for real analytic functions of such matrices. As it is clear from the formulae \eqref{S n}, \eqref{S j+ 1} and \eqref{S j+2}, this property holds for the matrices $S_n(x,\lambda)$, $S_{cr}^{\,(1)}(x,\lambda)$ and $S_{cr}^{\,(2)}(x,\lambda)$, and therefore, due to \eqref{T w}, for $\widehat T_{cr}(x,\lambda)$ as well as for $\exp(\widehat T_{cr}(x,\lambda))$ and for $(\exp(\widehat T_{cr}(x,\lambda)))'$. Hence it holds for $\widetilde R_{cr}(x,\lambda)$ given by \eqref{R tilde} and for
\begin{equation*}
        \left(
      \begin{array}{cc}
        e^{-\frac i2\phi_{cr}} & 0 \\
        0 & e^{\frac i2\phi_{cr}} \\
      \end{array}
    \right)
    \widetilde R_{cr}(x,\lambda)
    \left(
      \begin{array}{cc}
        e^{\frac i2\phi_{cr}} & 0 \\
        0 & e^{-\frac i2\phi_{cr}} \\
      \end{array}
    \right).
\end{equation*}
Since for every $a,b\in\mathbb C$
\begin{equation*}
    \left(
      \begin{array}{cc}
        1 & i \\
        1 & -i \\
      \end{array}
    \right)^{-1}
    \left(
      \begin{array}{cc}
        a & b \\
        \overline b & \overline a \\
      \end{array}
    \right)
    \left(
      \begin{array}{cc}
        1 & i \\
        1 & -i \\
      \end{array}
    \right)
    =
    \left(
      \begin{array}{cc}
        \Re a+\Re b & \Im b-\Im a \\
        \Im a+\Im b & \Re a-\Re b \\
      \end{array}
    \right),
\end{equation*}
we see from the expression \eqref{R cr} that the entries of $R_{cr}(x,\lambda)$ are real-valued. As a result $R_{cr}(\cdot,\lambda)\in L_1(\mathbb R_+,M^{2\times2}(\mathbb R))$ for every $\lambda\in U_{cr}$.

Consider the solution $w_{cr,\alpha}$ of the system \eqref{system w} which corresponds to the solution $\varphi_{\alpha}$ of the eigenfunction equation,
\begin{equation}\label{w cralpha}
    w_{cr,\alpha}(x,\lambda):=T_{cr}(x,\lambda)
    \left(
      \begin{array}{c}
        \varphi_{\alpha}(x,\lambda) \\
        \varphi_{\alpha}'(x,\lambda) \\
      \end{array}
    \right),
\end{equation}
where
\begin{multline*}
    T_{cr}(x,\lambda):=
    \left(
      \begin{array}{cc}
        1 & i \\
        1 & -i \\
      \end{array}
    \right)^{-1}
    \left(
      \begin{array}{cc}
        e^{-\frac i2\phi_{cr}} & 0 \\
        0 & e^{\frac i2\phi_{cr}} \\
      \end{array}
    \right)
    \exp(-\widehat T_{cr}(x,\lambda))
    \\
    \times
    \left(%
    \begin{array}{cc}
    \psi_-(x,\lambda) & \psi_+(x,\lambda) \\
    \psi_-'(x,\lambda) & \psi_+'(x,\lambda) \\
    \end{array}%
    \right)^{-1}.
\end{multline*}
The matrix
\begin{equation*}
    \left(
     \begin{array}{cc}
       e^{-\frac i2\phi_{cr}} & 0 \\
       0 & e^{\frac i2\phi_{cr}} \\
     \end{array}
   \right)
   \exp(-\widehat T_{cr}(x,\lambda))
\end{equation*}
has the same conjugation property as the matrix $\widehat R$. The matrix
\begin{equation*}
    \left(%
    \begin{array}{cc}
    \psi_-(x,\lambda) & \psi_+(x,\lambda) \\
    \psi_-'(x,\lambda) & \psi_+'(x,\lambda) \\
    \end{array}%
    \right)^{-1}
    =
    \frac1{W\{\psi_+,\psi_-\}(\lambda)}
    \left(%
    \begin{array}{cc}
    \psi_+'(x,\lambda) & -\psi_+(x,\lambda) \\
    -\psi_-'(x,\lambda) & \psi_-(x,\lambda) \\
    \end{array}%
    \right)
\end{equation*}
has the first row complex conjugate to the second row. Since for every $a,b,c\in\mathbb C$
\begin{equation*}
    \left(
      \begin{array}{cc}
        a & b \\
        \overline b & \overline a \\
      \end{array}
    \right)
    \left(
      \begin{array}{c}
        c \\
        \overline c \\
      \end{array}
    \right)
    =
    \left(
      \begin{array}{c}
        ac+b\overline c \\
        \overline a\overline c+\overline b c \\
      \end{array}
    \right),
    \
    \left(
      \begin{array}{cc}
        1 & i \\
        1 & -i \\
      \end{array}
    \right)^{-1}
    \left(
      \begin{array}{c}
        c \\
        \overline c \\
      \end{array}
    \right)
    =
    \left(
      \begin{array}{c}
        \Re c \\
        \Im c \\
      \end{array}
    \right),
\end{equation*}
the matrix $T_{cr}(x,\lambda)$ has real entries and therefore $w_{cr,\alpha}(x,\lambda)\in\mathbb R^2$ for every $x\in[0,+\infty)$, $\lambda\in U_{cr}$. The initial condition for the solution $w_{cr,\alpha}$ is
\begin{equation*}
    w_{cr,\alpha}(0,\lambda)=T_{cr}(0,\lambda)
    \left(
      \begin{array}{c}
        \sin\alpha \\
        \cos\alpha \\
      \end{array}
    \right),
\end{equation*}
which is continuous in $\lambda$. In particular,
\begin{equation*}
    w_{cr,\alpha}(0,\nu_{cr})=T_{cr}(0,\nu_{cr})
    \left(
      \begin{array}{c}
        \sin\alpha \\
        \cos\alpha \\
      \end{array}
    \right)=:g_{cr,\alpha}.
\end{equation*}
The matrix $T_{cr}(0,\nu_{cr})$ is non-degenerate, and hence the vector $g_{cr,\alpha}$ runs over all the directions in $\mathbb R^2$ as $\alpha$ runs over the interval $[0,\pi)$.

From the asymptotics \eqref{asymptotics of phi-alpha} due to the relation \eqref{eta} for every $\lambda\in U_{cr}\backslash\{\nu_{cr}\}$ we have
\begin{equation*}
    \lim\limits_{x\rightarrow+\infty}\widehat w_{cr,\alpha}(x,\lambda)
    =
    \left(
           \begin{array}{c}
             A_{\alpha}(\lambda) \\
             \overline{A_{\alpha}(\lambda)} \\
           \end{array}
    \right).
\end{equation*}
Since $\widehat T_{cr}(x,\lambda)\rightarrow0$ as $x\rightarrow+\infty$,
\begin{equation*}
    \lim\limits_{x\rightarrow+\infty}w_{cr,\alpha}(x,\lambda)
    =
    \left(
      \begin{array}{cc}
        1 & i \\
        1 & -i \\
      \end{array}
    \right)^{-1}
    \left(
      \begin{array}{cc}
        e^{-\frac i2\phi_{cr}} & 0 \\
        0 & e^{\frac i2\phi_{cr}} \\
      \end{array}
    \right)
   \left(
       \begin{array}{c}
         A_{\alpha}(\lambda) \\
         \overline{A_{\alpha}(\lambda)} \\
       \end{array}
     \right).
\end{equation*}
The matrix
$\frac1{\sqrt2}\left(
  \begin{array}{cc}
    1 & i \\
    1 & -i \\
  \end{array}
\right)$
is unitary. Therefore
\begin{equation*}
    \left\|\lim_{x\rightarrow+\infty}w_{cr,\alpha}(x,\lambda)\right\|
    =\frac{|A_{\alpha}(\lambda)|}{\sqrt2}
    \left\|
    \left(
      \begin{array}{c}
        1 \\
        1 \\
      \end{array}
    \right)
    \right\|
    =|A_{\alpha}(\lambda)|.
\end{equation*}
In the opposite direction, the relation \eqref{w cralpha} means that
\begin{multline}\label{varphi via w}
    \varphi_{\alpha}(x,\lambda)=(T_{cr}^{-1}(x,\lambda)w_{cr,\alpha}(x,\lambda),e_+)_{\mathbb C^2}
    \\
    =\left(w_{cr,\alpha}(x,\lambda),\left(
        \begin{array}{c}
        e^{-\frac i2\phi_{cr}}\overline{\psi_-(x,\lambda)}+e^{\frac i2\phi_{cr}}\overline{\psi_+(x,\lambda)}
        \\
        -ie^{-\frac i2\phi_{cr}}\overline{\psi_-(x,\lambda)}+ie^{\frac i2\phi_{cr}}\overline{\psi_+(x,\lambda)}
        \end{array}
      \right)+o(1)
    \right)_{\mathbb C^2}.
\end{multline}
\textit{3.} Consider the system \eqref{system w} for $\lambda=\nu_{cr}$:
\begin{equation*}
    w_{cr}'(x)=\left(
    \frac{\beta_{cr}}{x^{\gamma}}
    \left(
      \begin{array}{cc}
        1 & 0 \\
        0 & -1 \\
      \end{array}
    \right)
    +R_{cr}(x,\lambda)\right)w_{cr}(x).
\end{equation*}
By the asymptotic Levinson theorem \cite[Theorem 8.1]{Coddington-Levinson-1955} this system has two solutions $w_{cr}^{\pm}$ with the asymptotics
\begin{equation}\label{asympt wpm}
    w_{cr}^{\pm}(x)=\exp\left(\pm\frac{\beta_{cr}x^{1-\gamma}}{1-\gamma}\right)(e_{\pm}+o(1))
    \text{ as }x\rightarrow+\infty.
\end{equation}
Since the coefficients of the system are real-valued, $w_{cr}^-(x)\in\mathbb R^2$ for every $x\in[0,+\infty)$ (since $\overline{w_{cr}^-(x)}$ is also a solution which has the same asymptotics and hence is proportional to $w_{cr}^-(x)$ with the coefficient one). So there exists the unique $\alpha_{cr}\in[0,\pi)$ such that the vector $g_{cr,\alpha_{cr}}$ is proportional to $w_{cr}^-(0)$:
\begin{equation*}
    g_{cr,\alpha_{cr}}=d_{cr-}w_{cr}^-(0),
\end{equation*}
and hence
\begin{equation*}
    w_{cr,\alpha_{cr}}(x,\nu_{cr})=d_{cr-}w_{cr}^-(x),
\end{equation*}
From this and \eqref{asympt wpm} the asymptotics \eqref{asympt w-} follows, and from it using the relation \eqref{varphi via w} we get the asymptotics \eqref{intro phi critical - asympt gamma<1} of the solution $\varphi_{\alpha_{cr}}(x,\nu_{cr})$ as $x\rightarrow+\infty$. For every $\alpha\neq\alpha_{cr}$ due to \eqref{asympt wpm} and since $w_{cr,\alpha}(0,\nu_{cr})\nparallel w_{cr}^-(0)$ we have:
\begin{equation*}
    w_{cr,\alpha}(x,\nu_{cr})=d_{cr}(\alpha)\exp\left(\frac{\beta_{cr}x^{1-\gamma}}{1-\gamma}\right)(e_++o(1))
    \text{ as }x\rightarrow+\infty.
\end{equation*}
The coefficient $d_{cr}(\alpha)$ is a linear functional on $\mathbb R^2$, the space of initial conditions. Hence it can be expressed in terms of the scalar product with some fixed vector, or in terms of the angle between this vector and the vector
$\left(
  \begin{array}{c}
    \sin\alpha \\
    \cos\alpha \\
  \end{array}
\right)$.
Clearly this functional vanishes for $\alpha_{cr}$, therefore it should be
\begin{equation*}
    d_{cr}(\alpha)=d_{cr+}\sin(\alpha-\alpha_{cr})
\end{equation*}
with some $d_{cr+}\in\mathbb R$. From this we get the asymptotics \eqref{asympt w-}, and using the relation \eqref{varphi via w} the asymptotics \eqref{intro phi critical + asympt gamma<1} of $\varphi_{\alpha}(x,\nu_{cr})$ as $x\rightarrow+\infty$. This completes the proof.
\end{proof}

\section{The model problem}\label{section the model problem}
In this section we study the system \eqref{system w} in the general setting and use only the objects and the properties that are listed in Lemma \ref{lem reduction}. We pass from the spectral parameter $\lambda$ to the small parameter $\varepsilon_0$ supposing that the positive constant $\beta$, the remainder matrix $R(x,\varepsilon_0)$ with possibly complex entries and the vector of the initial condition $f\in\mathbb C^2$ are given. In such a setting we are able to establish the asymptotics as $x\rightarrow+\infty$ and then as $\varepsilon_0\rightarrow0$ of solutions of the model system
\begin{equation}\label{system model problem}
    u'(x)=
    \left(
    \frac{\beta}{x^{\gamma}}
    \left(
    \begin{array}{cc}
    \cos(\varepsilon_0x) & \sin(\varepsilon_0x)
    \\
    \sin(\varepsilon_0x) & -\cos(\varepsilon_0x)
    \end{array}
    \right)
    +
    R(x,\varepsilon_0)
    \right)
    u(x).
\end{equation}
We consider this system for $\varepsilon_0\in U_0$, where $U_0$ is some interval with the midpoint zero. Let $\alpha_r,c_r>0$ and
\begin{equation}\label{model problem r}
    r(x):=\frac{c_r}{(x+1)^{1+\alpha_r}}.
\end{equation}
We assume the following:
\begin{equation}\label{model problem conditions}
    \left\{
    \begin{array}{l}
    \beta>0,
    \\
    \gamma\in(\frac12,1),
    \\
    R(x,\cdot)\text{ for every }x\in[0,+\infty)\text{ is continuous in }U_0,
    \\
    \|R(x,\varepsilon)\|<r(x)\text{ for every }x\in[0,+\infty)\text{ and }\varepsilon_0\in U_0.
    \end{array}
    \right.
\end{equation}
Let us define for every $\varepsilon_0\in U_0$ and $f\in\mathbb C^2$ the solution $u(x,\varepsilon_0,f)$ of the Cauchy problem for the system \eqref{system model problem} with the initial condition
\begin{equation}\label{model problem initial condition}
    u(0,\varepsilon_0,f)=f.
\end{equation}
First we need to establish asymptotics of this solution as $x\rightarrow+\infty$ for every fixed $\varepsilon_0\in U_0$. We do this in the same way as in Lemma \ref{lem reduction}.

\begin{lem}\label{lem model problem individual asymptotics}
Let the conditions \eqref{model problem conditions} and \eqref{model problem r} hold, let $f\in\mathbb C^2$ and let $u(x,\varepsilon_0,f)$ be the solution of the system \eqref{system model problem} with the initial condition \eqref{model problem initial condition}.
\\
1. For every $\varepsilon_0\neq 0$ and $f\in\mathbb C^2$ there exists a finite non-zero limit $\lim\limits_{x\rightarrow+\infty}u(x,\varepsilon_0,f)$.
\\
2. For $\varepsilon_0=0$ and every $f\in\mathbb C^2$ the following asymptotics holds:
\begin{equation*}
    u(x,0,f)=\exp\left(\frac{\beta x^{1-\gamma}}{1-\gamma}\right)(\Phi(f)e_++o(1))
    \text{ as }x\rightarrow+\infty,
\end{equation*}
where $\Phi$ is a linear functional in $\mathbb C^2$. This functional has a one-dimensional kernel which consists of the vector $f_-$ (and its multiples) such that
\begin{equation*}
    u(x,0,f_-)=\exp\left(-\frac{\beta x^{1-\gamma}}{1-\gamma}\right)(e_-+o(1))
    \text{ as }x\rightarrow+\infty.
\end{equation*}
\end{lem}

\begin{proof}
\textit{1.} Just as in the previous section we make a Harris--Lutz transformation $u(x)=\exp(\widehat T_u(x,\varepsilon_0))u_1(x)$ with
\begin{equation*}
    \widehat T_u(x,\varepsilon_0):=-\int_x^{+\infty}
    \left(
    \begin{array}{cc}
    \cos(\varepsilon_0x') & \sin(\varepsilon_0x')
    \\
    \sin(\varepsilon_0x') & -\cos(\varepsilon_0x')
    \end{array}
    \right)
    \frac{dx'}{x'^{\gamma}}.
\end{equation*}
For $\varepsilon_0\neq0$ one has $\widehat T_u(x,\varepsilon_0)=O\left(\frac1{x^{\gamma}}\right)$ as $x\rightarrow+\infty$ and using the same kind of estimates as in the proof of Lemma \ref{lem reduction} we arrive at the system
\begin{equation*}
    u_1'(x)=R_1(x,\varepsilon_0)u_1(x)
\end{equation*}
with
\begin{equation*}
    R_1(x,\varepsilon_0)=O\left(\frac1{x^{2\gamma}}+\frac1{x^{1+\alpha_1}}\right)\text{ as }x\rightarrow+\infty,
\end{equation*}
which means that $R_1(\cdot,\varepsilon_0)\in L_1(\mathbb R_+,M^{2\times2}(\mathbb C))$. The asymptotic Levinson theorem \cite[Theorem 8.1]{Coddington-Levinson-1955} is applicable to this system and yields the existence of two solutions which have limits $e_+$ and $e_-$ as $x\rightarrow+\infty$. Hence every solution $u_1(x)$ has a non-zero limit as $x\rightarrow+\infty$. The same is true for every solution $u(x)$, because $\widehat T_u(x,\varepsilon_0)$ goes to zero at infinity.
\\
\textit{2.} If $\varepsilon_0=0$, then the asymptotic Levinson theorem is directly applicable to the system \eqref{system model problem}. Using it we conclude that there exists a solution with the asymptotics
\begin{equation*}
    \exp\left(-\frac{\beta x^{1-\gamma}}{1-\gamma}\right)(e_-+o(1))
    \text{ as }x\rightarrow+\infty.
\end{equation*}
Let us take as $f_-$ the value of this solution at zero. Again by the Levinson theorem there exists another solution, with the asymptotics
\begin{equation*}
    \exp\left(\frac{\beta x^{1-\gamma}}{1-\gamma}\right)(e_++o(1))
    \text{ as }x\rightarrow+\infty.
\end{equation*}
Therefore
\begin{equation*}
    u(x,\varepsilon_0,f)=\exp\left(\frac{\beta x^{1-\gamma}}{1-\gamma}\right)(\Phi(f)e_++o(1))\text{ as }x\rightarrow+\infty.
\end{equation*}
with some coefficient $\Phi(f)$ which depends on $f$ linearly and which is such that $\Phi(f_-)=0$ and $\text{dim}\,\text{ker}\,\Phi=1$. From \cite[Lemma 4.2, case (2)]{Kurasov-Simonov-2013} one can write out the formula for $\Phi$:
\begin{equation}\label{Phi}
    \Phi(f)=\left(\left(f+\int_0^{+\infty}\exp\left(-\frac{\beta x^{1-\gamma}}{1-\gamma}\right)R(x,0)u(x,0,f)dx\right),e_+\right)_{\mathbb C^2}.
\end{equation}
This completes the proof.
\end{proof}

The main result concerning the model problem is given by the following theorem which establishes the behaviour of the limit $\lim\limits_{x\rightarrow+\infty}u(x,\varepsilon_0,f)$ as $\varepsilon_0\rightarrow0$ and which will be proved in Section \ref{section matching}.

\begin{thm}\label{thm model problem}
Let the conditions \eqref{model problem conditions} and \eqref{model problem r} hold, let $f\in\mathbb C^2\backslash\{0\}$ and let $u(x,\varepsilon_0,f)$ be the solution of the system \eqref{system model problem} with the initial condition \eqref{model problem initial condition}. The following asymptotic holds:
\begin{equation}\label{answer u}
    \lim_{x\rightarrow+\infty}\|u(x,\varepsilon_0,f)\|
    =
    C_{mp}(\beta,\gamma)\exp
    \left(
    \frac1{|\varepsilon_0|^{\frac{1-\gamma}{\gamma}}}\int_0^{(2\beta)^{\frac1{\gamma}}}\sqrt{\frac{\beta^2}{t^{2\gamma}}-\frac14}dt
    \right)(|\Phi(f)|+o(1))
\end{equation}
as $\varepsilon_0\rightarrow0$, where $\Phi$ is defined in Lemma \ref{lem model problem individual asymptotics} and
\begin{multline}\label{C mp}
    C_{mp}(\beta,\gamma):=\frac1{\sqrt2}
    \exp\Biggl(\int_0^{\frac{(2\beta)^{\frac1{\gamma}}}2}
    \frac{\gamma\left(1-\sqrt{1-\frac{\tau^{2\gamma}}{4\beta^2}}\right)}{2\tau\left(1-\frac{\tau^{2\gamma}}{4\beta^2}\right)}d\tau
    \\
    -
    \int_{\frac{(2\beta)^{\frac1{\gamma}}}2}^{(2\beta)^{\frac1{\gamma}}}
    \frac{\gamma d\tau}{2\tau\sqrt{1-\frac{\tau^{2\gamma}}{4\beta^2}}}
    +\text{v.p.}\int_{\frac{(2\beta)^{\frac1{\gamma}}}2}^{+\infty}
    \frac{\gamma d\tau}{2\tau\left(1-\frac{\tau^{2\gamma}}{4\beta^2}\right)}\Biggr).
\end{multline}
\end{thm}

\begin{rem}\textit{1.} The case when $f$ is proportional to $f_-$ is included, and in such a case this is of course an estimate, not an asymptotics.
\\
\textit{2.}
The integral in the exponent in the asymptotics \eqref{answer u} can be expressed in terms of the beta function:
\begin{multline}\label{exponent}
    \int_0^{(2\beta)^{\frac1{\gamma}}}\sqrt{\frac{\beta^2}{t^{2\gamma}}-\frac14}dt
    =\frac{(2\beta)^{\frac1{\gamma}}}2\int_0^1t_1^{-\gamma}\sqrt{1-t_1^{2\gamma}}dt_1
    \\
    =\frac{(2\beta)^{\frac1{\gamma}}}{4\gamma}\int_0^1t_2^{\frac1{2\gamma}-\frac32}(1-t_2)^{\frac12}dt_2
    =\frac{(2\beta)^{\frac1{\gamma}}}{4\gamma}B\left(\frac32,\frac{1-\gamma}{2\gamma}\right).
\end{multline}
\end{rem}

Note that the asymptotics of $\lim\limits_{x\rightarrow+\infty}u(x,\varepsilon_0,f)$ contains $\Phi(f)$ which is responsible for the behaviour of solutions for the fixed value of the parameter, $\varepsilon_0=0$. This is explained by the multiscale nature of the problem: $\Phi(f)$ may be considered as the result of development of the solution in the ``fast'' variable $x$ which then enters the initial condition for the scaled system in the ``slow'' variable $t$. We consider this scaled system below.

\subsection{Reformulation}
The system \eqref{system model problem} was used in \cite{Naboko-Simonov-2012} for study of the case $\gamma=1$. Let us see why in the present form it is no longer suitable for the case $\gamma\in(\frac12,1)$. Consider $\varepsilon_0>0$. Scaling of the independent variable $x=\frac{t}{\varepsilon_0}$ leads to the system which has in some sense the following limit as $\varepsilon_0\rightarrow0^+$:
\begin{equation*}
    v'(t)=
    \frac{\beta}{t}
    \left(
    \begin{array}{cc}
    \cos t & \sin t
    \\
    \sin t & -\cos t
    \end{array}
    \right)
    v(t),
\end{equation*}
and the convergence is uniform in $t$. This system can be further analysed as $t\rightarrow+\infty$.
If we did the same for the case $\gamma\in(\frac12,1)$, we would come the system
\begin{equation*}
    \varepsilon_0^{1-\gamma}v'(t)=
    \frac{\beta}{t^{\gamma}}
    \left(
    \begin{array}{cc}
    \cos t & \sin t
    \\
    \sin t & -\cos t
    \end{array}
    \right)
    v(t),
\end{equation*}
which contains a small parameter at the derivative. This is why the difference between the cases $\gamma=1$ and $\gamma\in(\frac12,1)$ is essential. Moreover, in our case one needs to consider a different scale of the independent variable rather than $x=\varepsilon_0^{-1}t$, namely $x=\varepsilon_0^{-\frac1{\gamma}}t$. This is not immediately clear: such a substitution does not eliminate the small parameter from the derivative and leads to growing oscillations in the coefficient matrix:
\begin{equation*}
    \varepsilon_0^{\frac{1-\gamma}{\gamma}}v'(t)=
    \frac{\beta}{t^{\gamma}}
    \left(
    \begin{array}{cc}
    \cos\Bigl(\varepsilon_0^{-\frac{1-\gamma}{\gamma}}t\Bigr) & \sin\Bigl(\varepsilon_0^{-\frac{1-\gamma}{\gamma}}t\Bigr)
    \\
    \sin\Bigl(\varepsilon_0^{-\frac{1-\gamma}{\gamma}}t\Bigr) & -\cos\Bigl(\varepsilon_0^{-\frac{1-\gamma}{\gamma}}t\Bigr)
    \end{array}
    \right)
    v(t).
\end{equation*}

In view of the above let us begin not with scaling, but with getting rid of the oscillations: make the substitution
\begin{equation*}
    u(x)=
    \left(
    \begin{array}{cc}
    \cos\left(\frac{\varepsilon_0 x}2\right) & \sin\left(\frac{\varepsilon_0 x}2\right)
    \\
    \sin\left(\frac{\varepsilon_0 x}2\right) & -\cos\left(\frac{\varepsilon_0 x}2\right)
    \end{array}
    \right)
    u_1(x),
\end{equation*}
which leads to the system
\begin{equation*}
    u_1'(x)=
    \left(
    \frac{\beta}{x^{\gamma}}
    \left(
    \begin{array}{cc}
    1 & 0
    \\
    0 &-1
    \end{array}
    \right)
    +
    \frac{\varepsilon_0}2
    \left(
    \begin{array}{cc}
    0 & -1
    \\
    1 & 0
    \end{array}
    \right)
    +R_1(x,\varepsilon_0)
    \right)
    u_1(x),
\end{equation*}
where
\begin{equation*}
    R_1(x,\varepsilon_0)=
    \left(
    \begin{array}{cc}
    \cos\left(\frac{\varepsilon_0 x}2\right) & \sin\left(\frac{\varepsilon_0 x}2\right)
    \\
    \sin\left(\frac{\varepsilon_0 x}2\right) & -\cos\left(\frac{\varepsilon_0 x}2\right)
    \end{array}
    \right)
    R(x,\varepsilon_0)
    \left(
    \begin{array}{cc}
    \cos\left(\frac{\varepsilon_0 x}2\right) & \sin\left(\frac{\varepsilon_0 x}2\right)
    \\
    \sin\left(\frac{\varepsilon_0 x}2\right) & -\cos\left(\frac{\varepsilon_0 x}2\right)
    \end{array}
    \right).
\end{equation*}
Now let us scale the independent variable so as to make the first and the second terms in the coefficient matrix of the same order,
\begin{equation}\label{t}
    x=\frac{t}{|\varepsilon_0|^{\frac1{\gamma}}},
\end{equation}
and substitute
\begin{equation*}
    u_1(x)=u_2\left(|\varepsilon_0|^{\frac1{\gamma}}x\right).
\end{equation*}
This leads to two systems
\begin{equation*}
    |\varepsilon_0|^{\frac{1-\gamma}{\gamma}}u_2'(t)=
    \left(
    \left(
    \begin{array}{cc}
    \frac{\beta}{t^{\gamma}} & \mp\frac12
    \\
    \pm\frac12 & -\frac{\beta}{t^{\gamma}}
    \end{array}
    \right)
    +\frac1{|\varepsilon_0|}R_1\left(|\varepsilon_0|^{-\frac1{\gamma}}t,\varepsilon_0\right)
    \right)
    u_2(t),
\end{equation*}
for two possible signs of the parameter $\varepsilon_0$: one has to consider the cases $\varepsilon_0\rightarrow0^+$ and $\varepsilon_0\rightarrow0^-$ separately. Here the signs in $\pm$ and $\mp$ correspond to the sign of $\varepsilon_0$.
Now we can define a new small positive parameter
\begin{equation}\label{epsilon}
  \varepsilon:=|\varepsilon_0|^{\frac{1-\gamma}{\gamma}}
\end{equation}
from the set
\begin{equation}\label{U}
  U:=\bigl\{|\varepsilon_0|^{\frac{1-\gamma}{\gamma}}, \varepsilon_0\in U_0\bigr\}\backslash\{0\},
\end{equation}
and write these systems as
\begin{equation}\label{system u 2}
    \varepsilon {u_2^+}'(t)=(A_2^+(t)+R_2^+(t,\varepsilon))u_2^+(t)
\end{equation}
and
\begin{equation}\label{system u 2-}
    \varepsilon {u_2^-}'(t)=(A_2^-(t)+R_2^-(t,\varepsilon))u_2^-(t),
\end{equation}
where
\begin{equation*}
    A_2^{\pm}(t):=
    \left(
    \begin{array}{cc}
    \frac{\beta}{t^{\gamma}} & \mp\frac12
    \\
    \pm\frac12 & -\frac{\beta}{t^{\gamma}}
    \end{array}
    \right)
\end{equation*}
and
\begin{multline}\label{R 2 pm}
    R_2^{\pm}(t,\varepsilon):=\varepsilon^{-\frac{\gamma}{1-\gamma}}
    R_1\left(\varepsilon^{-\frac1{1-\gamma}}t,\pm\varepsilon^{\frac{\gamma}{1-\gamma}}\right)
    \\
    =
    \varepsilon^{-\frac{\gamma}{1-\gamma}}
    \left(
    \begin{array}{cc}
    \cos\left(\frac{t}{2\varepsilon}\right) & \sin\left(\frac{t}{2\varepsilon}\right)
    \\
    \sin\left(\frac{t}{2\varepsilon}\right) & -\cos\left(\frac{t}{2\varepsilon}\right)
    \end{array}
    \right)
    R\left(\varepsilon^{-\frac1{1-\gamma}}t,\pm\varepsilon^{\frac{\gamma}{1-\gamma}}\right)
    \left(
    \begin{array}{cc}
    \cos\left(\frac{t}{2\varepsilon}\right) & \sin\left(\frac{t}{2\varepsilon}\right)
    \\
    \sin\left(\frac{t}{2\varepsilon}\right) & -\cos\left(\frac{t}{2\varepsilon}\right)
    \end{array}
    \right).
\end{multline}
Let us introduce solutions $u_2^+(t,\varepsilon,f)$ of the system \eqref{system u 2} and $u_2^-(t,\varepsilon,f)$ of the system \eqref{system u 2-}, which correspond to $u(x,\varepsilon_0,f)$, by the formula
\begin{equation}\label{u 2 pm}
    u_2^{\pm}(t,\varepsilon,f):=
    \left(
    \begin{array}{cc}
    \cos\left(\frac{t}{2\varepsilon}\right) & \sin\left(\frac{t}{2\varepsilon}\right)
    \\
    \sin\left(\frac{t}{2\varepsilon}\right) & -\cos\left(\frac{t}{2\varepsilon}\right)
    \end{array}
    \right)
    u\left(\varepsilon^{-\frac1{1-\gamma}}t,\pm\varepsilon^{\frac{\gamma}{1-\gamma}},f\right),
\end{equation}
so that they have initial conditions
\begin{equation}\label{u 2 pm initial condition}
    u_2^{\pm}(0,\varepsilon,f)=
    \left(
      \begin{array}{cc}
        1 & 0 \\
        0 & -1 \\
      \end{array}
    \right)
    f.
\end{equation}

Each of the systems \eqref{system u 2} and \eqref{system u 2-} has an analytic part ($A_2^+$ or $A_2^-$) of the coefficient matrix and a remainder ($R_2^+$ or $R_2^-$) which is small in some sense. We will show that these remainders can be ignored away from zero. If there were no remainders, the well developed analytic theory would work here, see \cite[Chapter VIII]{Wasow-1965}. The eigenvalues of both matrices $A_2^{\pm}(t)$ are the same: $\sqrt{\frac{\beta^2}{t^{2\gamma}}-\frac14}$ and $-\sqrt{\frac{\beta^2}{t^{2\gamma}}-\frac14}$, and
\begin{equation}\label{t 0}
    t_0:=(2\beta)^{\frac1{\gamma}}
\end{equation}
is the turning point for both systems \eqref{system u 2}. At this point eigenvalues of $A_2^{\pm}$ coincide, and each of two matrices is similar to a Jordan block. The behaviour as $\varepsilon\rightarrow0^+$ of solutions of both systems \eqref{system u 2} has different character in the intervals $(0,t_0)$ and $(t_0,+\infty)$, so one needs to consider these intervals separately. In order to match the results in these two intervals we consider a small neighbourhood of the turning point and introduce a different (now again ``fast'') variable $z$ there. However, this is still not enough: we need to consider intermediate regions at both sides of the turning point and use a different method there to treat the remainders $R_2^{\pm}$. Only then matching of all the results can be done to trace the behaviour of $u_2^{\pm}(t,\varepsilon,f)$ from $t=0$ to $t=+\infty$.

In the statement of Theorem \ref{thm model problem} one can rewrite the expression \eqref{answer u} in the following way:
\begin{equation}\label{answer u 3}
    \exp
    \left(
    -\frac1{\varepsilon}\int_0^{t_0}\sqrt{\frac{\beta^2}{t^{2\gamma}}-\frac14}dt
    \right)
    \lim_{t\rightarrow+\infty}\|u_2^{\pm}(t,\varepsilon,f)\|
    \rightarrow
    C_{mp}(\beta,\gamma)|\Phi(f)|
\end{equation}
as $\varepsilon\rightarrow0^+$. Note that due to the oscillations in \eqref{u 2 pm} $u_2^{\pm}(t,\varepsilon,f)$ cannot have limits as $t\rightarrow+\infty$, and only limits of their norms exist. This can be interpreted in the sense that the vector $\Phi(f)e_+$ plays the role of the initial condition for the systems \eqref{system u 2}, and the growth of norms of the solutions $u_2^{\pm}$ takes place on the interval $(0,t_0)$ at the rate determined by the positive eigenvalue of the matrices $A_2^{\pm}$.

\subsection{Regions of the half-line}
In the next five sections we analyse the asymptotic behaviour of solutions of the system \eqref{system u 2} in five different regions of the positive half-line. In each of these regions we need to use appropriate transformations in order to simplify the system. Then we combine the results to match asymptotics, and this gives the asymptotic behaviour of the limit of the norm of the solution $u_2^+(t,\varepsilon,f)$ at $t=+\infty$ as $\varepsilon\rightarrow0^+$. In Section \ref{section matching} we will see that there is no need to perform a parallel study for the second system \eqref{system u 2-} and the solution $u_2^-$. We take points $t_{I-II}\in(0,t_0)$ and $t_{IV-V}\in(t_0,+\infty)$ sufficiently close to the turning point $t_0$. On the interval $[t_{I-II},t_{IV-V}]$ we use the variable $z=\varepsilon^{-\frac23}\left(1-\frac{t^{2\gamma}}{4\beta^2}\right)$ for which this interval corresponds to $[-Z_2(\varepsilon),Z_2(\varepsilon)]$ (travelled in the opposite direction). We further divide this interval into three regions by the points $\pm Z_0$ which in the variable $t$ correspond to the points $t_{II-III}(\varepsilon)$ and $t_{III-IV}(\varepsilon)$. The regions are displayed on the following figure:
\\
\includegraphics[width=\textwidth]{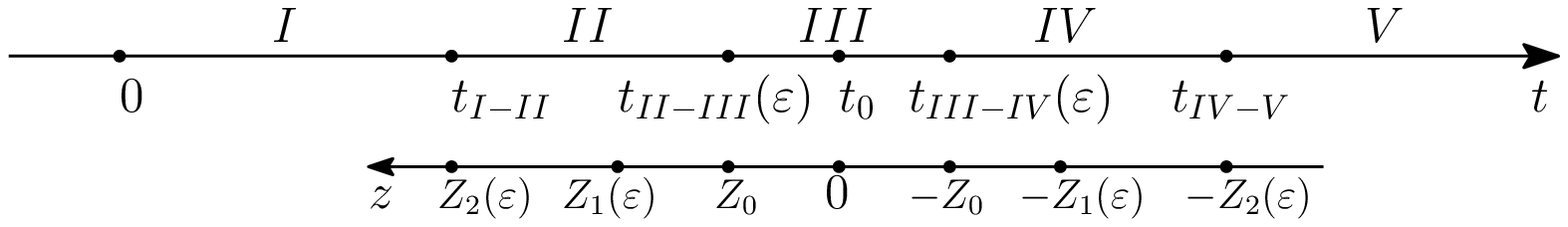}

In the region $I$ we obtain the asymptotics of the solution $u_2^+$, for the regions $II$ and $IV$ we find bases of solutions with known asymptotics, for the region $III$ we find a matrix solution and its asymptotics, and for the region $V$ we find a family of solutions determined by their behaviour as $t\rightarrow+\infty$ and establish their asymptotics as $\varepsilon\rightarrow0^+$.

In section devoted to the regions $II$--$V$ we formulate results in a general form, that is for systems of the kind of \eqref{system u 2}, imposing different sufficient conditions on remainders of these systems. We as well check that these conditions are satisfied for the remainder $R_2^+$ of the system \eqref{system u 2} itself. In notation for each of these systems and for other related objects we use indices which correspond to the region that is considered.

\section{Neighbourhood of the origin (region $I$): hyperbolic case }\label{section I}
We start with the system
\begin{equation}\label{system u I}
    \varepsilon u_2^{+'}(t)=
    \left(
    \left(
    \begin{array}{cc}
    \frac{\beta}{t^{\gamma}} & -\frac12
    \\
    \frac12 & -\frac{\beta}{t^{\gamma}}
    \end{array}
    \right)
    +R_2^+(t,\varepsilon)
    \right)
    u_2^+(t).
\end{equation}
Let us diagonalise the main term of the coefficient matrix with the transformation
\begin{equation}\label{u I1}
    u_2^+(t)=T_I(t)u_{I,1}(t),
\end{equation}
where
\begin{equation}\label{T I}
    T_I(t):=
    \left(
    \begin{array}{cc}
    1 & \frac{t^{\gamma}}{4\beta} \\
    \frac{t^{\gamma}}{2\beta\left(1+\sqrt{1-\frac{t^{2\gamma}}{4\beta^2}}\right)}
    &
    \frac12\left(1+\sqrt{1-\frac{t^{2\gamma}}{4\beta^2}}\right) \
    \end{array}
    \right)
\end{equation}
(the eigenvector in the second column is chosen so that it does not have a singularity at $t=0$). The substitution gives:
\begin{equation}\label{system u I1}
    u_{I,1}'(t)=\left(\frac{\lambda_I(t)}{\varepsilon}
    \left(
    \begin{array}{cc}
    1 & 0 \\
    0 & -1 \\
    \end{array}
    \right)
    +S_I(t)+R_{I,1}(t,\varepsilon)\right)u_{I,1}(t),
\end{equation}
where
\begin{equation}\label{lambda I}
    \lambda_I(t):=\sqrt{\frac{\beta^2}{t^{2\gamma}}-\frac14},
\end{equation}
\begin{equation}\label{S I}
    S_I(t):=\frac{\gamma}{8\beta t^{1-\gamma}\left(1-\frac{t^{2\gamma}}{4\beta^2}\right)}
    \left(
    \begin{array}{cc}
    \frac{t^{\gamma}}{\beta\left(1+\sqrt{1-\frac{t^{2\gamma}}{4\beta^2}}\right)}
    &
    -1-\sqrt{1-\frac{t^{2\gamma}}{4\beta^2}}
    \\
    -\frac{4}{1+\sqrt{1-\frac{t^{2\gamma}}{4\beta^2}}}
    &
    \frac{t^{\gamma}\left(1+2\sqrt{1-\frac{t^{2\gamma}}{4\beta^2}}\right)}{\beta\left(1+\sqrt{1-\frac{t^{2\gamma}}{4\beta^2}}\right)}
    \end{array}
    \right),
\end{equation}
\begin{equation}\label{R I1}
    R_{I,1}(t,\varepsilon):=\frac{T_I^{-1}(t)R_2^+(t,\varepsilon)T_I(t)}{\varepsilon}.
\end{equation}

The result for the region $I$ is given by the following lemma.

\begin{lem}\label{lem I result}
Let the conditions \eqref{model problem conditions} and \eqref{model problem r} hold and let for $f\in\mathbb C^2$ the function $u(x,\varepsilon_0,f)$ be the solution of the system \eqref{system model problem} with the initial condition \eqref{model problem initial condition}. Let $u_2^+(t,\varepsilon,f)$ be given by \eqref{u 2 pm} with the use of the definitions \eqref{t} of $t$ and \eqref{epsilon} of $\varepsilon$, and thus be a solution of the system \eqref{system u I} where $R_2^+(t,\varepsilon)$ is given by \eqref{R 2 pm}. For every $t\in(0,t_0)$ and $\varepsilon\in U$ the following asymptotics holds:
\begin{equation}\label{I asymptotics}
    u_2^+(t,\varepsilon,f)
    =
     T_I(t)\exp\left(
    \int_0^t\left(
    \frac{\lambda_I(\tau)}{\varepsilon}+
    S_{I,+}(\tau)
    \right)d\tau\right)
    \left(
    \Phi(f)e_+
    +o(1)
    \right),
\end{equation}
as $\varepsilon\rightarrow0^+$, where $T_I,\lambda_I,S_I$ and $e_+$ are given by the expressions \eqref{T I}, \eqref{lambda I}, \eqref{S I} and \eqref{e +-}, respectively, $S_{I,+}$ is the upper-left entry of the matrix $S_I$ and $\Phi$ is defined in Lemma \ref{lem model problem individual asymptotics}.
\end{lem}

First let us prove an a priori estimate.

\begin{lem}\label{lem I a priori estimate}
Let $t_I\in(0,t_0)$. Under the conditions of Lemma \ref{lem I result} there exists $c_{13}>0$ such that for every $t\in[0,t_I]$, $\varepsilon\in U$ and $f\in\mathbb C^2$
\begin{equation}\label{a priori estimate}
    \|u_2^+(t,\varepsilon,f)\|<c_{13}\exp\left(\frac1{\varepsilon}\int_0^t\lambda_I\right)\|f\|,
\end{equation}
where $\lambda_I$ is given by \eqref{lambda I}.
\end{lem}

\begin{proof}
Rough estimate of the norm of the coefficient matrix of the system \eqref{system u I1} immediately gives:
\begin{equation*}
   \|u_{I,1}(t,\varepsilon,f)\|<\exp\left(\int_0^t\left(\frac{\lambda_I(\tau)}{\varepsilon}+\|S_I(\tau)\|+\|R_{I,1}(\tau,\varepsilon)\|\right)d\tau\right)\|f\|.
\end{equation*}
Since $S_I\in L_1((0,t_I),M^{2\times2}(\mathbb C))$, we need the estimate $\int_0^{t_I}\|R_{I,1}(\tau,\varepsilon)\|d\tau=O(1)$ as $\varepsilon\rightarrow0^+$. To see this first note that $T_I$ and $T_I^{-1}$ are bounded in $[0,t_I]$, and so, with some $c_{14}>0$,
\begin{equation*}
    \int_0^{t_I}\|R_{I,1}(\tau,\varepsilon)\|d\tau<\frac{c_{14}}{\varepsilon}\int_0^{t_I}\|R_2^+(\tau,\varepsilon)\|d\tau,
\end{equation*}
and so using \eqref{R 2 pm}, the condition on $R$ from \eqref{model problem conditions} and summability of $r$ due to \eqref{model problem r} we have
\begin{multline*}
    \int_0^{t_I}\|R_{I,1}(\tau,\varepsilon)\|d\tau
    <\frac{c_{14}}{\varepsilon}\int_0^{t_I}\|R_2^+(\tau,\varepsilon)\|d\tau
    =\frac{c_{14}}{\varepsilon_0\varepsilon}\int_0^{t_I}\left\|R\left(\varepsilon_0^{-\frac1{\gamma}}\tau,\varepsilon_0\right)\right\|d\tau
    \\
    <\frac{c_{14}}{\varepsilon_0^{\frac1{\gamma}}}\int_0^{t_I}r\left(\varepsilon_0^{-\frac1{\gamma}}\tau\right)d\tau
    \le c_{14}\int_0^{+\infty}r(x)dx,
\end{multline*}
which is finite and does not depend on $\varepsilon$. Using again boundedness of $T_I$ and the relation \eqref{u I1} we complete the proof.
\end{proof}

\begin{rem}
Note that the estimate \eqref{a priori estimate} is not valid for $t\ge t_0$, because $S_I\notin L_1(0,t_0)$. Neighbourhood of the turning point requires special attention.
\end{rem}

Let us make the variation of parameters
\begin{equation}\label{u I2}
    u_{I,1}(t)=
    \left(
      \begin{array}{cc}
        \exp\left(\frac1{\varepsilon}\int_0^t\lambda_I\right) & 0 \\
        0 & \exp\left(-\frac1{\varepsilon}\int_0^t\lambda_I\right) \\
      \end{array}
    \right)
    u_{I,2}(t),
\end{equation}
which gives
\begin{equation*}
    u_{I,2}'(t)=
    \left(
      \begin{array}{cc}
        \exp\left(-\frac1{\varepsilon}\int_0^t\lambda_I\right) & 0 \\
        0 & \exp\left(\frac1{\varepsilon}\int_0^t\lambda_I\right) \\
      \end{array}
    \right)
    (S_I(t)+R_{I,1}(t,\varepsilon))u_{I,1}(t).
\end{equation*}
Integrating this from $0$ to $t$ and returning to $u_{I,1}$ we get the integral equation for the solution of the system \eqref{system u I1},
\begin{equation}\label{u I1f}
    u_{I,1}(t,\varepsilon,f):=T_{I}^{-1}(t)u_2^+(t,\varepsilon,f),
\end{equation}
\begin{multline}\label{eq u I1f}
    u_{I,1}(t,\varepsilon,f)=
    \left(
      \begin{array}{cc}
        \exp\left(\frac1{\varepsilon}\int_0^t\lambda_I\right) & 0 \\
        0 & -\exp\left(-\frac1{\varepsilon}\int_0^t\lambda_I\right) \\
      \end{array}
    \right)
    f
    \\
    +\int_0^t
    \left(
      \begin{array}{cc}
        \exp\left(\frac1{\varepsilon}\int_\tau^t\lambda_I\right) & 0 \\
        0 & \exp\left(-\frac1{\varepsilon}\int_\tau^t\lambda_I\right) \\
      \end{array}
    \right)
    (S_I(\tau)+R_{I,1}(\tau,\varepsilon))u_{I,1}(\tau,\varepsilon,f)d\tau.
\end{multline}
Scaling
\begin{equation}\label{u I3}
    u_{I,3}(t,\varepsilon,f):=\exp\left(-\frac1{\varepsilon}\int_0^t\lambda_I\right)u_{I,1}(t,\varepsilon,f)
\end{equation}
we come to another integral equation,
\begin{multline}\label{I integral eq}
    u_{I,3}(t,\varepsilon,f)=
    \left(
      \begin{array}{cc}
        1 & 0 \\
        0 & -\exp\left(-\frac2{\varepsilon}\int_0^t\lambda_I\right) \\
      \end{array}
    \right)
    f
    \\
    +\int_0^t
    \left(
      \begin{array}{cc}
        1 & 0 \\
        0 & \exp\left(-\frac2{\varepsilon}\int_\tau^t\lambda_I\right) \\
      \end{array}
    \right)
    (S_I(\tau)+R_{I,1}(\tau,\varepsilon))u_{I,3}(\tau,\varepsilon,f)d\tau.
\end{multline}
Rewrite it as
\begin{equation}\label{I integral eq with kernel}
    u_{I,3}(t,\varepsilon,f)=h_{I,3}(t,\varepsilon,f)+\int_0^tK_I(t,\tau,\varepsilon)u_{I,3}(\tau,\varepsilon,f)d\tau,
\end{equation}
where
\begin{multline}\label{h I3}
    h_{I,3}(t,\varepsilon,f):=
    \left(
      \begin{array}{cc}
        1 & 0 \\
        0 & -\exp\left(-\frac2{\varepsilon}\int_0^t\lambda_I\right) \\
      \end{array}
    \right)
    f
    \\
    +\int_0^t
    \left(
      \begin{array}{cc}
        1 & 0 \\
        0 & \exp\left(-\frac2{\varepsilon}\int_\tau^t\lambda_I\right) \\
      \end{array}
    \right)
    R_{I,1}(\tau,\varepsilon)u_{I,3}(\tau,\varepsilon,f)d\tau
\end{multline}
and
\begin{equation}\label{K I}
    K_I(t,\tau,\varepsilon):=
    \left(
      \begin{array}{cc}
        1 & 0 \\
        0 & \exp\left(-\frac2{\varepsilon}\int_\tau^t\lambda_I\right) \\
      \end{array}
    \right)
    S_I(\tau).
\end{equation}
Define also
\begin{equation}\label{K I(0)}
    K_I(t,\tau,0):=    \left(
      \begin{array}{cc}
        1 & 0 \\
        0 & 0\\
      \end{array}
    \right)
    S_I(\tau).
\end{equation}
Now fix an arbitrary point $t_I\in(0,t_0)$. Consider the equation \eqref{I integral eq with kernel} as an equation in the Banach space $L_{\infty}((0,t_I),\mathbb C^2)$
\begin{equation}\label{I integral eq with operator}
    u_{I,3}(\varepsilon,f)=h_{I,3}(\varepsilon,f)+\mathcal K_I(\varepsilon)u_{I,3}(\varepsilon,f),
\end{equation}
 where $\mathcal K_I(\varepsilon)$ is the Volterra operator
 \begin{equation*}
    \mathcal K_I(\varepsilon):u(t)\mapsto\int_0^tK_I(t,\tau,\varepsilon)u(\tau)d\tau
\end{equation*}
(which makes sense for $\varepsilon=0$ too).

\begin{lem}\label{lem I convergence of the free term}
Let the conditions of Lemma \ref{lem I result} hold and let $h_{I,3}$ be given by \eqref{h I3} with the use of the relations \eqref{u I1f} and \eqref{u I3}.
\\
1. For every $t\in(0,t_0)$ and $f\in\mathbb C^2$ the following limit exists:
\begin{equation}\label{h I3(0)}
    \lim_{\varepsilon\rightarrow0^+}h_{I,3}(t,\varepsilon,f)=\Phi(f)e_+=:h_{I,3}(t,0,f),
\end{equation}
where $\Phi$ is defined in Lemma \ref{lem model problem individual asymptotics}.
\\
2. There exists  $c_{15}>0$ such that for every $t\in(0,t_0)$, $\varepsilon\in U$ and $f\in\mathbb C^2$ one has $\|h_{I,3}(t,\varepsilon,f)\|<c_{15}\|f\|$.
\end{lem}

\begin{rem}
Note that there is no convergence $h_{I,3}(\varepsilon,f)\rightarrow h_{I,3}(0,f)$ as $\varepsilon\rightarrow0^+$ in the norm of $L_{\infty}((0,t_I),\mathbb C^2)$.
\end{rem}

\begin{proof}
Rewrite the second summand in \eqref{h I3} in the following way:
\begin{multline}\label{I eq free term}
    \int_0^t
    \left(
      \begin{array}{cc}
        1 & 0 \\
        0 & \exp\left(-\frac2{\varepsilon}\int_\tau^t\lambda_I\right) \\
      \end{array}
    \right)
    R_{I,1}(\tau,\varepsilon)u_{I,3}(\tau,\varepsilon,f)d\tau
    \\
    =
    \int_0^t
    \left(
      \begin{array}{cc}
        1 & 0 \\
        0 & \exp\left(-\frac2{\varepsilon}\int_\tau^t\lambda_I\right) \\
      \end{array}
    \right)
    T_I^{-1}(\tau)R_2^+(\tau,\varepsilon)T_I(\tau)u_{I,1}(\tau,\varepsilon,f)e^{-\frac1{\varepsilon}\int_0^\tau\lambda_I}\frac{d\tau}{\varepsilon}
    \\
    =
    \int_0^t
    \left(
      \begin{array}{cc}
        1 & 0 \\
        0 & \exp\left(-\frac2{\varepsilon}\int_\tau^t\lambda_I\right) \\
      \end{array}
    \right)
    T_I^{-1}(\tau)R_2^+(\tau,\varepsilon)u_2^+(\tau,\varepsilon,f)e^{-\frac1{\varepsilon}\int_0^\tau\lambda_I}\frac{d\tau}{\varepsilon}
    \\
    =
    \int_0^t
    \left(
      \begin{array}{cc}
        1 & 0 \\
        0 & \exp\left(-\frac2{\varepsilon}\int_\tau^t\lambda_I\right) \\
      \end{array}
    \right)
    T_I^{-1}(\tau)
    \\
    \times
    \left(
    \begin{array}{cc}
    \cos\left(\frac{\tau}{2\varepsilon}\right) & \sin\left(\frac{\tau}{2\varepsilon}\right)
    \\
    \sin\left(\frac{\tau}{2\varepsilon}\right) & -\cos\left(\frac{\tau}{2\varepsilon}\right)
    \end{array}
    \right)
    R\left(\varepsilon_0^{-\frac1{\gamma}}\tau,\varepsilon_0\right)u\left(\varepsilon_0^{-\frac1{\gamma}}\tau,\varepsilon_0,f\right)
    e^{-\frac1{\varepsilon}\int_0^\tau\lambda_I}\frac{d\tau}{\varepsilon_0^{\frac1{\gamma}}}
    \\
    =
    \int_0^{\varepsilon_0^{-\frac1{\gamma}}t}
    \left(
      \begin{array}{cc}
        1 & 0 \\
        0 & \exp\left(-\frac2{\varepsilon}\int_{\varepsilon_0^{\frac1{\gamma}}x}^t\lambda_I\right) \\
      \end{array}
    \right)
    T_I^{-1}(\varepsilon_0^{\frac1{\gamma}}x)
    \\
    \times
    \left(
    \begin{array}{cc}
    \cos\left(\frac{\varepsilon_0 x}{2}\right) & \sin\left(\frac{\varepsilon_0 x}{2}\right)
    \\
    \sin\left(\frac{\varepsilon_0 x}{2}\right) & -\cos\left(\frac{\varepsilon_0 x}{2}\right)
    \end{array}
    \right)
    R(x,\varepsilon_0)u(x,\varepsilon_0,f)
    \exp\left(-\frac1{\varepsilon}\int_0^{\varepsilon_0^{\frac1{\gamma}}x}\lambda_I\right)dx.
\end{multline}
Consider the expression $\exp\left(-\frac2{\varepsilon}\int_{\varepsilon_0^{\frac1{\gamma}}x}^t\lambda_I\right)$ which is positive and is less than one. For every fixed $x$ and sufficiently small $\varepsilon$ it is less then $\exp\left(-\frac2{\varepsilon}\int_{\frac t2}^t\lambda_I\right)$ which converges to zero as $\varepsilon\rightarrow0^+$. Moreover,
\begin{multline*}
    \exp\left(-\frac1{\varepsilon}\int_0^{\varepsilon_0^{\frac1{\gamma}}x}\lambda_I\right)
    =\exp\left(-\frac1{\varepsilon}\int_0^{\varepsilon_0^{\frac1{\gamma}}x}\sqrt{\frac{\beta^2}{t^{2\gamma}}-\frac14}dt\right)
    \\
    =\exp\left(-\varepsilon_0\int_0^x\sqrt{\frac{\beta^2}{\varepsilon_0^2y^{2\gamma}}-\frac14}dy\right)
    \rightarrow
    \exp\left(-\int_0^x\frac{\beta dy}{y^{\gamma}}\right)=\exp\left(-\frac{\beta x^{1-\gamma}}{1-\gamma}\right)
\end{multline*}
as $\varepsilon\rightarrow0^+$.  Since $T_I(0)=I$ and the functions $R(x,\cdot)$ and $u(x,\cdot,f)$ are continuous in $U_0$, the expression under the integral in the result of the calculation \eqref{I eq free term} for every fixed $x$ converges to
\begin{equation*}
    \left(
      \begin{array}{cc}
        1 & 0 \\
        0 & 0 \\
      \end{array}
    \right)
    R(x,0)u(x,0,f)
    \exp\left(-\frac{\beta x^{1-\gamma}}{1-\gamma}\right)dx.
\end{equation*}
Since
\begin{equation*}
    \|u(x,\varepsilon_0,f)\|=\left\|u_2^+\left(\varepsilon_0^{\frac1{\gamma}}x,\varepsilon,f\right)\right\|
    <c_{13}\exp\left(\frac1{\varepsilon}\int_0^{\varepsilon_0^{\frac1{\gamma}}x}\lambda_I\right)\|f\|,
\end{equation*}
by Lemma \ref{lem I a priori estimate}, the estimate $\|R(x,\varepsilon_0)\|<r(x)$ provides a summable majorant for the expression under the integral, and by the Lebesgue's dominated convergence theorem we get that for every $t\in(0,t_I]$ there exists a limit of $h_{I,3}(t,\varepsilon,f)$ as $\varepsilon\rightarrow0^+$ which equals to
\begin{equation*}
    \left(
      \begin{array}{cc}
        1 & 0 \\
        0 & 0 \\
      \end{array}
    \right)\left(f+
    \int_0^{+\infty}
    R(x,0)u(x,0,f)
    \exp\left(-\frac{\beta x^{1-\gamma}}{1-\gamma}\right)dx\right)
    =\Phi(f)e_+,
\end{equation*}
according to the formula \eqref{Phi} for $\Phi$. The uniform boundedness of the $L_{\infty}((0,t_I),\mathbb C^2)$ norm of $h_{I,3}(t,\varepsilon,f)$ also follows from the existence of a summable majorant. This completes the proof.
\end{proof}

We denote by $\mathcal B(L_{\infty}((0,t_I),\mathbb C^2))$ the Banach space of bounded operators in $L_{\infty}((0,t_I),\mathbb C^2)$.

\begin{lem}\label{lem I operator convergence}
Let the conditions of Lemma \ref{lem I result} hold and $K_I(\varepsilon)$ be given by \eqref{K I} for $\varepsilon\neq0$ and by \eqref{K I(0)} for $\varepsilon=0$. Let $h_{I,3}(\varepsilon,f)$ for $\varepsilon\neq0$ be given by \eqref{h I3} with the use of the relations \eqref{u I1f}, \eqref{u I3}, and for $\varepsilon=0$ be defined in Lemma \ref{lem I convergence of the free term}. Let $t_I\in(0,t_0)$. Then the following holds.
\\
1. $\mathcal K_I(\varepsilon)\rightarrow\mathcal K_I(0)$ as $\varepsilon\rightarrow0^+$ in the norm of $\mathcal B(L_{\infty}((0,t_I),\mathbb C^2))$.
\\
2. $\mathcal K_I(0)h_{I,3}(\varepsilon,f)\rightarrow\mathcal K_I(0)h_{I,3}(0,f)$ as $\varepsilon\rightarrow0^+$ in the norm of $L_{\infty}((0,t_I),\mathbb C^2)$.
\end{lem}

\begin{proof}
\emph{1.} It suffices to prove that
\begin{equation*}
    \max\limits_{t\in[0,t_I]}\int_0^t\|K_I(t,\tau,\varepsilon)-K_I(t,\tau,0)\|d\tau\rightarrow0\text{ as }\varepsilon\rightarrow0^+,
\end{equation*}
or that
\begin{equation*}
    \max\limits_{t\in[0,t_I]}\int_0^t\exp\left(-\frac2{\varepsilon}\int_\tau^t\lambda_I\right)\|S_I(\tau)\|d\tau\rightarrow0
    \text{ as }\varepsilon\rightarrow0^+.
\end{equation*}
On the interval $[0,t_I]$ we can use estimates $\lambda_I(t)>c_{16}$ and $\|S_I(t)\|<\frac{c_{17}}{t^{1-\gamma}}$ with some $c_{16},c_{17}>0$ which can be seen directly from the expressions \eqref{lambda I} and \eqref{S I}. Take an arbitrary small $\Delta>0$. Firstly, there exists $t(\Delta)$ such that
\begin{equation*}
    \int_0^{t(\Delta)}\|S_I(\tau)\|d\tau<\frac{\Delta}2.
\end{equation*}
Thus
\begin{equation}\label{I est 1}
    \max\limits_{t\in[0,t(\Delta)]}\int_0^t\exp\left(-\frac2{\varepsilon}\int_\tau^t\lambda_I\right)\|S_I(\tau)\|d\tau<\frac{\Delta}2.
\end{equation}
Secondly,
\begin{multline}\label{I est 2}
    \max\limits_{t\in[t(\Delta),t_I]}\int_0^t\exp\left(-\frac2{\varepsilon}\int_\tau^t\lambda_I\right)\|S_I(\tau)\|d\tau
    \\
    \le
    \frac{\Delta}2+\max\limits_{t\in[t(\Delta),t_I]}\int_{t(\Delta)}^t\exp\left(-\frac2{\varepsilon}\int_\tau^t\lambda_I\right)\|S_I(\tau)\|d\tau
    \\
    \le
    \frac{\Delta}2+\frac{c_{17}}{(t(\Delta))^{1-\gamma}}
    \max\limits_{t\in[0,t_I]}\int_0^t\exp\left(-\frac{2c_{16}}{\varepsilon}(t-\tau)\right)
    \le
    \frac{\Delta}2+\frac{c_{17}}{(t(\Delta))^{1-\gamma}}\frac{\varepsilon}{2c_{16}}.
\end{multline}
One can choose $\varepsilon(\Delta)>0$ so that for every $\varepsilon\in(0,\varepsilon(\Delta))$ holds $\frac{c_{17}}{(t(\Delta))^{1-\gamma}}\frac{\varepsilon}{2c_{16}}<\frac{\Delta}2$, and so, from the estimates \eqref{I est 1} and \eqref{I est 2},
\begin{equation*}
    \max\limits_{t\in[0,t_I]}\int_0^t\exp\left(-\frac2{\varepsilon}\int_\tau^t\lambda_I\right)\|S_I(\tau)\|d\tau<\Delta.
\end{equation*}
This proves the convergence.
\\
\emph{2.} We have:
\begin{multline*}
    \|K_I(0)(h_{I,3}(\varepsilon,f)-h_{I,3}(0,f))\|_{L_{\infty}(0,t_I)}
    \\
    \le
    \max\limits_{t\in[0,t_I]}\left\|\int_0^t
\left(
  \begin{array}{cc}
    1 & 0 \\
    0 & 0 \\
  \end{array}
\right)
S_I(\tau)(h_{I,3}(\tau,\varepsilon,f)-h_{I,3}(\tau,0,f))d\tau\right\|
\\
\le
\int_0^{t_I}\|S_I(\tau)\|\|h_{I,3}(\tau,\varepsilon,f)-h_{I,3}(\tau,0,f)\|d\tau.
\end{multline*}
In the expression under the integral $S_I$ is summable and $h_{I,3}$ is point-wise convergent to zero and uniformly bounded by Lemma \ref{lem I convergence of the free term}. Therefore, by the Lebegue's dominated convergence theorem, the integral converges to zero as $\varepsilon\rightarrow0^+$.
\end{proof}

Now we are able to prove convergence of the solution.

\begin{lem}\label{lem I convergence of the solution}
Let the conditions of Lemma \ref{lem I result} hold and let $u_{I,3}$ be defined by \eqref{u I3}. For every $t\in(0,t_0)$ there exists the limit
\begin{equation*}
    \lim_{\varepsilon\rightarrow0^+}u_{I,3}(t,\varepsilon,f)=:u_{I,3}(t,0,f),
\end{equation*}
which satisfies the following integral equation on the interval $[0,t_0)$ :
\begin{equation}\label{I integral eq in the limit}
    u_{I,3}(t,0,f)=h_{I,3}(t,0,f)+\int_0^tK_I(t,\tau,0)u_{I,3}(\tau,0,f)d\tau,
\end{equation}
where $K_I(0)$ is given by \eqref{K I(0)}, and $h_{I,3}(0,f)$ is defined in Lemma \ref{lem I convergence of the free term}.
\end{lem}

\begin{rem}
Note that again, as with $h_{I,3}(\varepsilon,f)$, there is no convergence in the norm of $L_{\infty}((0,t_I),\mathbb C^2)$. However, the difference $u_{I,3}(\varepsilon,f)-h_{I,3}(\varepsilon,f)$ converges in this norm.
\end{rem}

\begin{proof} Take some $t_I\in(0,t_0)$. Let us rewrite the equation \eqref{I integral eq with operator} as
\begin{equation*}
    u_{I,3}(\varepsilon,f)-h_{I,3}(\varepsilon,f)=\mathcal K_I(\varepsilon)h_{I,3}(\varepsilon,f)+\mathcal K_I(\varepsilon)(u_{I,3}(\varepsilon,f)-h_{I,3}(\varepsilon,f))
\end{equation*}
and then as
\begin{equation}\label{I eq 1}
    u_{I,3}(\varepsilon,f)-h_{I,3}(\varepsilon,f)=(I-\mathcal K_I(\varepsilon))^{-1}\mathcal K_I(\varepsilon)h_{I,3}(\varepsilon,f).
\end{equation}
By Lemma \ref{lem I operator convergence}, due to the boundedness of $h_{I,3}(\varepsilon,f)$ in the norm of $L_{\infty}((0,t_I),\mathbb C^2)$ provided by Lemma \ref{lem I convergence of the solution}, we have:
\begin{equation*}
    \mathcal K_I(\varepsilon)h_{I,3}(\varepsilon,f)
    =(\mathcal K_I(\varepsilon)-\mathcal K_I(0))h_{I,3}(\varepsilon,f)+\mathcal K_I(0)h_{I,3}(\varepsilon,f)
    \rightarrow\mathcal K_I(0)h_{I,3}(0,f)
\end{equation*}
 as $\varepsilon\rightarrow0^+$. Since $\mathcal K_I(\varepsilon)$ are Volterra operators, the norms $\|(I-\mathcal K_I(\varepsilon))^{-1}\|\le\exp(\|\mathcal K_I(\varepsilon)\|)$ are bounded as $\varepsilon\rightarrow0^+$. Therefore
\begin{equation*}
    (I-\mathcal K_I(\varepsilon))^{-1}=(I+(I-\mathcal K_I(\varepsilon))^{-1}(\mathcal K_I(\varepsilon)-\mathcal K_I(0)))(I-\mathcal K_I(0))^{-1}\rightarrow(I-\mathcal K_I(0))^{-1}
\end{equation*}
in the norm of $\mathcal B(L_{\infty}((0,t_I),\mathbb C^2))$. Hence in the equality \eqref{I eq 1} there exists the limit
\begin{equation*}
    \lim_{\varepsilon\rightarrow0^+}(u_{I,3}(\varepsilon,f)-h_{I,3}(\varepsilon,f))=(I-\mathcal K_I(0))^{-1}\mathcal K_I(0)h_{I,3}(0,f)
\end{equation*}
in the norm of $L_{\infty}((0,t_I),\mathbb C^2)$ which means that $u_{I,3}(t,\varepsilon,f)-h_{I,3}(t,\varepsilon,f)$ has a limit for every $t\in[0,t_I]$ and uniformly in $t$. Since $t_I$ was chosen arbitrarily, the limit exists for every $t\in(0,t_0)$, however, without uniformity. By Lemma \ref{lem I convergence of the free term} for every $t\in(0,t_0)$ the function $h_{I,3}(t,\varepsilon,f)$ has a limit as $\varepsilon\rightarrow0^+$, also not uniform in $t$. On the interval $[0,t_I]$ we have the equality
\begin{equation*}
    u_{I,3}(0,f)=h_{I,3}(0,f)+(I-\mathcal K_I(0))^{-1}\mathcal K_I(0)h_{I,3}(0,f).
\end{equation*}
Applying $I-\mathcal K_I(0)$ to both sides we arrive at the formula \eqref{I integral eq in the limit}.
\end{proof}

Now we can prove Lemma \ref{lem I result}.

\begin{proof}[Proof of Lemma \ref{lem I result}]
Equation \eqref{I integral eq in the limit} is in fact simpler than it looks: it is merely an equation for the upper component of $u_{I,3}$ which can be solved explicitly. Indeed, using the formulae \eqref{h I3(0)} and \eqref{K I(0)} for the initial condition and the kernel we get the equation
\begin{equation*}
    u_{I,3}(t,0,f)=\Phi(f)e_++\int_0^t
    \left(
      \begin{array}{cc}
        1 & 0 \\
        0 & 0 \\
      \end{array}
    \right)
    S_I(\tau)
    u_{I,3}(\tau,0,f)d\tau,
\end{equation*}
and the solution is given by the expression
\begin{equation}\label{I asymptotics u I3}
    u_{I,3}(t,0,f)=\Phi(f)\exp\left(\int_0^tS_{I,+}(\tau)d\tau\right)e_+.
\end{equation}
Putting this into \eqref{u I3} and \eqref{u I1} we obtain the asymptotics \eqref{I asymptotics}, which completes the proof.
\end{proof}

\section{Neighbourhood of the infinity (region $V$): elliptic case }\label{section V}
We start with the same system in different notation:
\begin{equation}\label{system u V}
    \varepsilon u_{V}'(t)=
    \left(
    \left(
    \begin{array}{cc}
    \frac{\beta}{t^{\gamma}} & -\frac12
    \\
    \frac12 & -\frac{\beta}{t^{\gamma}}
    \end{array}
    \right)
    +R_{V}(t,\varepsilon)
    \right)
    u_{V}(t).
\end{equation}
In the same way as in the region $I$ let us diagonalise the main term of the coefficient matrix by the transformation
\begin{equation}\label{u V1}
    u_V(t)=T_V(t)u_{V,1}(t),
\end{equation}
where
\begin{equation}\label{T V}
    T_V(t):=
    \left(
    \begin{array}{cc}
    1
    &
    1
    \\
    \frac{2\beta}{t^{\gamma}}+i\sqrt{1-\frac{4\beta^2}{t^{2\gamma}}}
    &
    \frac{2\beta}{t^{\gamma}}-i\sqrt{1-\frac{4\beta^2}{t^{2\gamma}}}
    \end{array}
    \right).
\end{equation}
Substitution into the system \eqref{system u V} gives:
\begin{equation}\label{system u V1}
    u_{V,1}'(t)=\left(\frac{\lambda_V(t)}{\varepsilon}
    \left(
    \begin{array}{cc}
    1 & 0 \\
    0 & -1 \\
    \end{array}
    \right)
    +S_V(t)+R_{V,1}(t,\varepsilon)\right)u_{I,1}(t),
\end{equation}
where
\begin{equation}\label{lambda V}
    \lambda_V(t):=-\frac i2\sqrt{1-\frac{4\beta^2}{t^{2\gamma}}},
\end{equation}
\begin{equation}\label{S V}
    S_V(t):=\frac{\beta\gamma}{\left(1-\frac{4\beta^2}{t^{2\gamma}}\right)t^{1+\gamma}}
    \left(
    \begin{array}{cc}
    -\frac{2\beta}{t^{\gamma}}-i\sqrt{1-\frac{4\beta^2}{t^{2\gamma}}}
    &
    \frac{2\beta}{t^{\gamma}}-i\sqrt{1-\frac{4\beta^2}{t^{2\gamma}}}
    \\
    \frac{2\beta}{t^{\gamma}}+i\sqrt{1-\frac{4\beta^2}{t^{2\gamma}}}
    &
    -\frac{2\beta}{t^{\gamma}}+i\sqrt{1-\frac{4\beta^2}{t^{2\gamma}}}
    \end{array}
    \right),
\end{equation}
\begin{equation}\label{R V1}
    R_{V,1}(t,\varepsilon):=\frac{T_V^{-1}(t)R_{V}(t,\varepsilon)T_V(t)}{\varepsilon}.
\end{equation}

Here we consider solutions which are defined not by their values at zero, but rather by their asymptotics at infinity. Asymptotics of solutions in the region $V$ is given by the following lemma.

\begin{lem}\label{lem V answer}
Let $\beta>0,\gamma\in(\frac12,1),t_0=(2\beta)^{\frac1{\gamma}}$ and $g\in\mathbb C^2$. If
\begin{equation}\label{condition R V integral}
    \int_{t_0}^{+\infty}\|R_V(t,\varepsilon)\|dt=o(\varepsilon)\text{ as }\varepsilon\rightarrow0^+,
\end{equation}
then for every $\varepsilon\in U$ there exists a solution $u_V(t,\varepsilon,g)$ of the system \eqref{system u V} on the interval $(t_0,+\infty)$ with the following asymptotics:
\begin{multline}\label{V asymptotics}
    u_V(t,\varepsilon,g)
    =
     T_V(t)
     \\
     \times
     \Biggl(
     \exp
     \Biggl(
    \int_{t_0}^t
    \frac{\lambda_V(\tau)}{\varepsilon}
    \left(
      \begin{array}{cc}
        1 & 0 \\
        0 & -1 \\
      \end{array}
    \right)
    d\tau
    -\int_t^{+\infty}
    \text{\emph{diag}}\,
    S_{V}(\tau)
    d\tau
    \Biggr)
    g
    +o(1)
    \Biggr)
\end{multline}
as $\varepsilon\rightarrow0^+$, where the convergence of the term $o(1)$ is uniform with respect to $t\in[t_V,+\infty)$ for every $t_V\in(t_0,+\infty)$. Moreover, for every $\varepsilon\in U$
\begin{equation}\label{V convergence u V t->infty}
    \|u_V(t,\varepsilon,g)\|\rightarrow\sqrt2\|g\|\text{ as }t\rightarrow+\infty.
\end{equation}
\end{lem}

Let us rewrite the condition \eqref{condition R V integral} in terms of the remainder $R_{V,1}$.

\begin{lem}\label{lem V remainder}
Under the conditions of Lemma \ref{lem V answer}, if for some $t_V\in(t_0,+\infty)$
\begin{equation*}
   \int_{t_V}^{+\infty}\|R_V(t,\varepsilon)\|dt=o(\varepsilon)\text{ as }\varepsilon\rightarrow0^+,
   \text{ then }
   \int_{t_V}^{+\infty}\|R_{V,1}(t,\varepsilon)\|dt\rightarrow0\text{ as }\varepsilon\rightarrow0^+.
\end{equation*}
\end{lem}

\begin{proof}
Since $T_V$ and $T_V^{-1}$ are bounded in $[t_V,+\infty) $, one has with some $c_{18}>0$ using the definition \eqref{R V1}:
\begin{equation*}
    \int_{t_V}^{+\infty}\|R_{V,1}(\tau,\varepsilon)\|d\tau<\frac{c_{18}}{\varepsilon}\int_{t_V}^{+\infty}\|R_V(\tau,\varepsilon)\|d\tau,
\end{equation*}
which goes to zero as $\varepsilon\rightarrow0^+$.
\end{proof}

Now let us see that the condition \eqref{condition R V integral} is satisfied if we take $R_V=R_2^+$.

\begin{lem}\label{lem V R 2+ estimate}
Let $R_2^+$ be given by \eqref{R 2 pm}, $r\in L_1(\mathbb R_+)$ and conditions \eqref{model problem conditions} hold. Then
\begin{equation*}
   \int_{t_0}^{+\infty}\|R_2^+(t,\varepsilon)\|dt=o(\varepsilon)\text{ as }\varepsilon\rightarrow0^+.
\end{equation*}
\end{lem}

\begin{proof}
Using the definition \eqref{R 2 pm}, the estimate of $R$ from \eqref{model problem conditions}, the relation \eqref{epsilon} and summability of the function $r$ we have:
\begin{multline}\label{V estimate R V1}
    \int_{t_0}^{+\infty}\|R_2^+(\tau,\varepsilon)\|d\tau
    =\frac1{\varepsilon_0}\int_{t_0}^{+\infty}\left\|R\left(\varepsilon_0^{-\frac1{\gamma}}\tau,\varepsilon_0\right)\right\|d\tau
    \\
    \le\varepsilon\varepsilon_0^{-\frac1{\gamma}}\int_{t_0}^{+\infty}r\left(\varepsilon_0^{-\frac1{\gamma}}\tau\right)d\tau
    =\varepsilon\int_{\varepsilon_0^{-\frac1{\gamma}}t_0}^{+\infty}r(x)dx=o(\varepsilon)\text{ as }\varepsilon\rightarrow0^+.
\end{multline}
This gives the result.
\end{proof}

It is more convenient for the later use to reformulate this result in different terms. Making the variation of parameters in the system \eqref{system u V1},
\begin{equation}\label{u V2}
    u_{V,1}(t)=
    \left(
      \begin{array}{cc}
        \exp\left(\frac1{\varepsilon}\int_{t_0}^t\lambda_V\right) & 0 \\
        0 & \exp\left(-\frac1{\varepsilon}\int_{t_0}^t\lambda_V\right) \\
      \end{array}
    \right)
    u_{V,2}(t),
\end{equation}
we come to the system
\begin{multline}\label{system u V2}
    u_{V,2}'(t)=
    \left(
      \begin{array}{cc}
        \exp\left(-\frac1{\varepsilon}\int_{t_0}^t\lambda_V\right) & 0 \\
        0 & \exp\left(\frac1{\varepsilon}\int_{t_0}^t\lambda_V\right) \\
      \end{array}
    \right)
    (S_V(t)+R_{V,1}(t,\varepsilon))
    \\
    \times
    \left(
      \begin{array}{cc}
        \exp\left(\frac1{\varepsilon}\int_{t_0}^t\lambda_V\right) & 0 \\
        0 & \exp\left(-\frac1{\varepsilon}\int_{t_0}^t\lambda_V\right) \\
      \end{array}
    \right)
    u_{V,2}(t).
\end{multline}
The following lemma will be used as the result for the region $V$  in Section \ref{section matching}.

\begin{lem}\label{lem V result for u V2}
Let $\beta>0,\gamma\in(\frac12,1),t_0=(2\beta)^{\frac1{\gamma}},g\in\mathbb C^2$ and
\begin{equation}\label{V condition on R V}
    \int_{t_0}^{+\infty}\|R_V(t,\varepsilon)\|dt=o(\varepsilon)\text{ as }\varepsilon\rightarrow0^+.
\end{equation}
There exists the solution $u_{V,2}(t,\varepsilon,g)$ of the system \eqref{system u V2} such that the following holds.
\\
1. For every  $\varepsilon\in U$
\begin{equation*}
    u_{V,2}(t,\varepsilon,g)\rightarrow g\text{ as }t\rightarrow+\infty.
\end{equation*}
2. For every $t\in(t_0,+\infty)$
\begin{equation}\label{V asymptotics u V2}
    u_{V,2}(t,\varepsilon,g)\rightarrow\exp\left(-\int_t^{+\infty}\text{\emph{diag}}\,S_V(\tau)d\tau\right)g
    \text{ as }\varepsilon\rightarrow0^+,
\end{equation}
and the limit is uniform with respect to $t\in[t_V,+\infty)$ for every $t_V\in(t_0,+\infty)$.
\end{lem}

\begin{proof}[Proof of the item 1.]
We can use directly the asymptotic Levinson theorem. The coefficient matrix of the system \eqref{system u V2} is summable near infinity: from \eqref{S V} it is clear that $S_V\in L_1((t_V,+\infty),M^{2\times2}(\mathbb C))$, and $R_{V,1}$ is summable due to Lemma \ref{lem V remainder}. Conditions of the asymptotic Levinson theorem \cite[Theorem 8.1]{Coddington-Levinson-1955} are satisfied which gives the result. We will prove the item \textit{2.} later.
\end{proof}

Let us integrate both sides of the equation \eqref{system u V2} from $t$ to $+\infty$:
\begin{multline}\label{V eq 2}
    u_{V,2}(t,\varepsilon,g)=g-
    \int_t^{+\infty}
    \left(
      \begin{array}{cc}
        \exp\left(-\frac1{\varepsilon}\int_{t_0}^\tau\lambda_V\right) & 0 \\
        0 & \exp\left(\frac1{\varepsilon}\int_{t_0}^\tau\lambda_V\right) \\
      \end{array}
    \right)
    \\
    \times
    (S_V(\tau)+R_{V,1}(\tau,\varepsilon))
    \left(
      \begin{array}{cc}
        \exp\left(\frac1{\varepsilon}\int_{t_0}^\tau\lambda_V\right) & 0 \\
        0 & \exp\left(-\frac1{\varepsilon}\int_{t_0}^\tau\lambda_V\right) \\
      \end{array}
    \right)
    u_{V,2}(\tau,\varepsilon,g)d\tau.
\end{multline}
Rewrite this as
\begin{equation}\label{V integral eq with kernel}
    u_{V,2}(t,\varepsilon,g)=g+\int_t^{+\infty}K_V(\tau,\varepsilon)u_{V,2}(\tau,\varepsilon,g)d\tau,
\end{equation}
where
\begin{multline}\label{K V}
    K_V(\tau,\varepsilon):=
    -
    \left(
      \begin{array}{cc}
        \exp\left(-\frac1{\varepsilon}\int_{t_0}^\tau\lambda_V\right) & 0 \\
        0 & \exp\left(\frac1{\varepsilon}\int_{t_0}^\tau\lambda_V\right) \\
      \end{array}
    \right)
    \\
    \times
    (S_V(\tau)+R_{V,1}(\tau,\varepsilon))
    \left(
      \begin{array}{cc}
        \exp\left(\frac1{\varepsilon}\int_{t_0}^\tau\lambda_V\right) & 0 \\
        0 & \exp\left(-\frac1{\varepsilon}\int_{t_0}^\tau\lambda_V\right) \\
      \end{array}
    \right).
\end{multline}
Define also
\begin{equation}\label{K V(0)}
    K_V(\tau,0):=-\text{diag}\,S_V(\tau).
\end{equation}
Take some $t_V\in(t_0,+\infty)$ and consider \eqref{V integral eq with kernel} as an equation in the Banach space $L_{\infty}((t_V,+\infty),\mathbb C^2)$
\begin{equation}\label{V integral eq with operator}
    u_{V,2}(\varepsilon,g)=g+\mathcal K_V(\varepsilon)u_{V,2}(\varepsilon,g),
\end{equation}
 where $\mathcal K_V(\varepsilon)$ is the Volterra operator
 \begin{equation}\label{K V operator}
    \mathcal K_V(\varepsilon):u(t)\mapsto\int_0^tK_V(\tau,\varepsilon)u(\tau)d\tau,
\end{equation}
which is also defined by this rule for  $\varepsilon=0$.

\begin{rem}
Until this moment we were following the scheme which was already used for the region $I$. This analogy does not go further: the operators $\mathcal K_V(\varepsilon)$ do not converge in the norm of $\mathcal B(L_{\infty}((t_V,+\infty),\mathbb C^2))$, while the solutions  $u_{V,2}(\varepsilon,g)$ do converge in the norm of $L_{\infty}((t_V,+\infty),\mathbb C^2)$.
\end{rem}

\begin{lem}\label{lem V convergence of operators}
Let the conditions of Lemma \ref{lem V result for u V2} hold, $t_V\in(t_0,+\infty)$ and $\mathcal K_V(\varepsilon)$ be the operator in the space $L_{\infty}((t_V,+\infty),\mathbb C^2)$ defined by \eqref{K V operator} with the use of expressions \eqref{K V} and \eqref{K V(0)}. Operator $\mathcal K_V(\varepsilon)$ converges in the strong sense to the operator $\mathcal K_V(0)$ as $\varepsilon\rightarrow0^+$, and $(I-\mathcal K_V(\varepsilon))^{-1}$ converges in the strong sense to $(I-\mathcal K_V(0))^{-1}$.
\end{lem}

\begin{proof}
Take some arbitrary $u\in L_{\infty}((t_V,+\infty),\mathbb C^2)$. We have:
\begin{multline*}
    \|(\mathcal K_V(\varepsilon)-\mathcal K_V(0))u\|_{L_{\infty}(t_V,+\infty)}
    \le
    \left\|\int_{t_V}^{+\infty}(K_V(\tau,\varepsilon)-K_V(\tau,0))u(\tau)d\tau\right\|
    \\
    \le
    \left|\int_{t_V}^{+\infty}S_{V,12}(\tau)u_2(\tau)\exp\left(-\frac2{\varepsilon}\int_{t_0}^\tau\lambda_V\right)d\tau\right|
    \\
    +
    \left|\int_{t_V}^{+\infty}S_{V,21}(\tau)u_1(\tau)\exp\left(\frac2{\varepsilon}\int_{t_0}^\tau\lambda_V\right)d\tau\right|
    \\
    +
    \|u\|_{L_{\infty}(t_V,+\infty)}\int_{t_V}^{+\infty}\|R_{V,1}(\tau,\varepsilon)\|d\tau,
\end{multline*}
where $S_{V,12}$ and $S_{V,21}$ are the upper-right and the lower-left entries of the matrix $S_V$, and $u_1$ and $u_2$ are the upper and the lower components of the vector $u$. From the estimate \eqref{V estimate R V1} it follows that the third term in the last formula goes to zero as $\varepsilon\rightarrow0^+$. For the first term after the change of the variable of integration to $s(\tau)=\int_{t_0}^\tau\sqrt{1-\frac{4\beta^2}{\tau'^{2\gamma}}}d\tau'$ we have
\begin{multline*}
    \int_{t_V}^{+\infty}S_{V,12}(\tau)u_2(\tau)\exp\left(-\frac 2{\varepsilon}\int_{t_0}^\tau\lambda_V(\tau')d\tau'\right)d\tau
    \\
    =\int_{t_V}^{+\infty}S_{V,12}(\tau)u_2(\tau)\exp\left(\frac i{\varepsilon}\int_{t_0}^\tau\sqrt{1-\frac{4\beta^2}{\tau'^{2\gamma}}}d\tau'\right)d\tau
    \\
    =
    \int_{s(t_V)}^{+\infty}\frac{S_{V,12}(\tau(s))u_2(\tau(s))}{\sqrt{1-\frac{4\beta^2}{\tau(s)^{2\gamma}}}}\exp\left(\frac{is}{\varepsilon}\right)ds
    \rightarrow0\text{ as }\varepsilon\rightarrow0^+,
\end{multline*}
by the  Riemann-Lebesgue lemma, because $S_Vu\in L_1((t_V,+\infty),M^{2\times2}(\mathbb C))$ and
\begin{equation*}
    \int_{s(t_V)}^{+\infty}\frac{|S_{V,12}(\tau(s))u_2(\tau(s))|}{\sqrt{1-\frac{4\beta^2}{\tau(s)^{2\gamma}}}}ds
    =\int_{t_V}^{+\infty}|S_{V,12}(\tau)u_2(\tau)|d\tau<\infty.
\end{equation*}
The second term goes to zero for the analogous reason and the third term by Lemma \ref{lem V remainder}. This proves strong convergence of operators $\mathcal K_V(\varepsilon)$.

It remains to prove  strong convergence of $(I-\mathcal K_V(\varepsilon))^{-1}=\sum_{n=0}^{\infty}\mathcal K_V^n(\varepsilon)$. This series converges in the norm of $\mathcal B(L_{\infty}((t_V,+\infty),\mathbb C^2))$ uniformly in $\varepsilon\in U\cup\{0\}$, because for every $\varepsilon\in U\cup\{0\}$ and $n\in\mathbb N$ one has
\begin{equation*}
    \|\mathcal K_V^n(\varepsilon)\|_{\mathcal B(L_{\infty})}\le\frac{(\|\mathcal K_V(\varepsilon)\|_{\mathcal B(L_{\infty})})^n}{n!}
    \text{ and }
    \|\mathcal K_V(\varepsilon)\|_{\mathcal B(L_{\infty})}\le\int_{t_V}^{+\infty}\|S_{V}(\tau)\|d\tau.
\end{equation*}

Take $u\in L_{\infty}((t_V,+\infty),\mathbb C^2)$. By induction one proves that for every $n\in\mathbb N$ it holds that $\mathcal K_V^n(\varepsilon)u\rightarrow\mathcal K_V^n(0)u$ as $\varepsilon\rightarrow0^+$:
\begin{equation*}
    \mathcal K_V^n(\varepsilon)u=\mathcal K_V(\varepsilon)(\mathcal K_V^{n-1}(\varepsilon)-\mathcal K_V^{n-1}(0))u+\mathcal K_V(\varepsilon)\mathcal K_V^{n-1}(0)u\rightarrow \mathcal K_V^n(0)u,
\end{equation*}
where the first term goes to zero by the induction hypothesis and the second converges to the result due to the strong convergence of $\mathcal K_V(\varepsilon)$. Take arbitrarily small $\Delta>0$. There exists $N(\Delta)$ such that for every $\varepsilon\in U\cup\{0\}$ and $N>N(\Delta)$ one has
\begin{equation*}
    \left\|\sum_{n=N}^{\infty}\mathcal K_V^n(\varepsilon)u\right\|<\frac{\Delta}3.
\end{equation*}
There also exists $\varepsilon(\Delta)>0$ such that for every $\varepsilon<\varepsilon(\Delta)$ it holds that
\begin{equation*}
    \left\|\sum_{n=0}^{N(\Delta)}(\mathcal K_V^n(\varepsilon)-\mathcal K_V^n(0))u\right\|<\frac{\Delta}3.
\end{equation*}
Therefore for every $\varepsilon<\varepsilon(\Delta)$ we have
\begin{equation*}
    \left\|\sum_{n=0}^{\infty}\mathcal K_V^n(\varepsilon)u-\sum_{n=0}^{\infty}\mathcal K_V^n(0)u\right\|
    <\frac{2\Delta}3+\left\|\sum_{n=0}^{N(\Delta)}(\mathcal K_V^n(\varepsilon)-\mathcal K_V^n(0))u\right\|<\Delta,
\end{equation*}
which proves that
\begin{equation*}
    (I-\mathcal K_V(\varepsilon))^{-1}u\rightarrow(I-\mathcal K_V(0))^{-1}u\text{ as }\varepsilon\rightarrow0^+.
\end{equation*}
Since $u\in L_{\infty}((t_V,+\infty),\mathbb C^2)$ was arbitrary, this proves the strong convergence and thus completes the proof of the lemma.
\end{proof}

Now we are able to prove the remaining part of Lemma \ref{lem V result for u V2}.

\begin{proof}[Proof of Lemma \ref{lem V result for u V2}, item \textit{2.}]
Rewriting the equation \eqref{V integral eq with operator} as
\begin{equation*}
    u_{V,2}(\varepsilon,g)=(I-\mathcal K_V(\varepsilon))^{-1}g
\end{equation*}
we use Lemma \ref{lem V convergence of operators} to conclude that $u_{V,2}(\varepsilon,g)\rightarrow u_{V,2}(0,g)$ as $\varepsilon\rightarrow0^+$ in the norm of $L_{\infty}((t_V,+\infty),\mathbb C^2)$, where
\begin{equation}\label{V eq 1}
    u_{V,2}(0,g):=(I-\mathcal K_V(0))^{-1}g,
\end{equation}
which means that for every $g\in\mathbb C^2$ and $t\in[t_V,+\infty)$,  $u_{V,2}(t,\varepsilon,g)$ converges as $\varepsilon\rightarrow0^+$ to $u_{V,2}(t,0,g)$ uniformly with respect to $t$ in this interval. Applying the operator $I-\mathcal K_V(0)$ to both sides of the equality \eqref{V eq 1} we get:
\begin{equation*}
    u_{V,2}(t,0,g)=g-\int_t^{+\infty}\text{diag}\,S_V(\tau)u_{V,2}(\tau,0,g)d\tau.
\end{equation*}
Solution of this equation is
\begin{equation*}
    u_{V,2}(t,0,g)=\exp\left(-\int_t^{+\infty}\text{diag}\,S_V(\tau)d\tau\right)g,
\end{equation*}
which coincides with the expression in \eqref{V asymptotics u V2}  and thus completes the proof.
\end{proof}

The proof of Lemma \ref{lem V answer} follows.

\begin{proof}[Proof of Lemma \ref{lem V answer}]
Convergence in \eqref{V asymptotics} follows from Lemma \ref{lem V result for u V2} and substitution of the limit \eqref{V asymptotics u V2} to the relations \eqref{u V2} and \eqref{u V1}. To prove convergence in \eqref{V convergence u V t->infty} first note that $u_V(t,\varepsilon,g)$ is bounded as $t\rightarrow+\infty$, because
\begin{equation*}
    \|u_V(t,\varepsilon,g)\|\le\|T_V(t)\|\|u_{V,2}(t,\varepsilon,g)\|
\end{equation*}
and $u_{V,2}(t,\varepsilon,g)$ is bounded by Lemma \ref{lem V result for u V2}. Then, since
\begin{equation*}
    T_V(t)\rightarrow
    \left(
      \begin{array}{cc}
        1 & 1 \\
        i & -i \\
      \end{array}
    \right)
    \text{ as }t\rightarrow+\infty
\end{equation*}
and the matrix
$\frac1{\sqrt2}\left(
  \begin{array}{cc}
    1 & 1 \\
    i & -i \\
  \end{array}
\right)$
is unitary, by Lemma \ref{lem V result for u V2} we have the convergence of the norm: $\|u_V(t,\varepsilon,g)\|\rightarrow\sqrt2\|g\|\text{ as }\rightarrow+\infty.$ This completes the proof.
\end{proof}

\section{Intermediate region $II$: hyperbolic case}\label{section II}
Cosider the system \eqref{system u 2}
\begin{equation*}
    \varepsilon u_{II}'(t)=
    \left(
    \left(
    \begin{array}{cc}
    \frac{\beta}{t^{\gamma}} & -\frac12
    \\
    \frac12 & -\frac{\beta}{t^{\gamma}}
    \end{array}
    \right)
    +R_{II}(t,\varepsilon)
    \right)
    u_{II}(t).
\end{equation*}
Let us again, like in the region $I$,  diagonalise the main term of the coefficient matrix, this time with the transformation
\begin{equation}\label{u II1}
    u_{II}(t)=T_{II}(t)u_{II,1}(t),
\end{equation}
where
\begin{equation}\label{T II}
    T_{II}(t):=
    \left(
    \begin{array}{cc}
    1
    &
    1
    \\
    \frac{t^{\gamma}}{2\beta\left(1+\sqrt{1-\frac{t^{2\gamma}}{4\beta^2}}\right)}
    &
    \frac{t^{\gamma}}{2\beta\left(1-\sqrt{1-\frac{t^{2\gamma}}{4\beta^2}}\right)}
   \end{array}
    \right).
\end{equation}
Note that the first column is the same as of the matrix $T_{I}(t)$ in \eqref{T I}. The substitution gives:
\begin{equation}\label{system u II1}
    u_{II,1}'(t)=\left(\frac{\lambda_{II}(t)}{\varepsilon}
    \left(
    \begin{array}{cc}
    1 & 0 \\
    0 & -1 \\
    \end{array}
    \right)
    +S_{II}(t)+R_{II,1}(t,\varepsilon)\right)u_{II,1}(t),
\end{equation}
where
\begin{equation}\label{lambda II}
    \lambda_{II}(t):=\sqrt{\frac{\beta^2}{t^{2\gamma}}-\frac14}=\lambda_{I}(t),
\end{equation}
\begin{equation}\label{S II}
    S_{II}(t):=\frac{\gamma}{2t\left(1-\frac{t^{2\gamma}}{4\beta^2}\right)}
    \left(
    \begin{array}{cc}
    1-\sqrt{1-\frac{t^{2\gamma}}{4\beta^2}}
    &
    -1-\sqrt{1-\frac{t^{2\gamma}}{4\beta^2}}
    \\
    -1+\sqrt{1-\frac{t^{2\gamma}}{4\beta^2}}
    &
    1+\sqrt{1-\frac{t^{2\gamma}}{4\beta^2}}
    \end{array}
    \right),
\end{equation}
\begin{equation*}
    R_{II,1}(t,\varepsilon):=\frac{T_{II}^{-1}(t)R_{II}(t,\varepsilon)T_{II}(t)}{\varepsilon}.
\end{equation*}

In the region $I$ our analysis was based on the fact that the term $S_I$ is summable over the whole interval. Since $S_I$ is not summable up to the turning point $t_0=(2\beta)^{\frac1{\gamma}}$, the integral containing $S_I$ diverges as $t$ approaches $t_0$, so Lemma \ref{lem I result} does not work for the interval $(0,t_0]$. In the same way and for the same reason Lemma \ref{lem V answer} does not work for the interval $[t_0,+\infty)$. This can be seen as an effect of the interplay between the first and the second terms in the coefficient matrix of the system \eqref{system u II1} which should be taken into account. We do this by scaling the independent variable near the turning point. To choose the new scale observe that near the turning point the first term has the order $\frac{\sqrt{t-t_0}}{\varepsilon}$, while the second term has the order $\frac1{t-t_0}$. To match these orders we have to consider the values of $t$ such that $t-t_0\approx\varepsilon^{\frac23}$. Therefore let us take
\begin{equation}\label{z}
    z:=\frac{1-\frac{t^{2\gamma}}{4\beta^2}}{\varepsilon^{\frac23}}.
\end{equation}
Since
\begin{equation*}
    1-\frac{t^{2\gamma}}{4\beta^2}=\frac{2\gamma(t_0-t)}{(4\beta)^{\frac1{\gamma}}}+O((t-t_0)^2)\text{ as }t\rightarrow t_0,
\end{equation*}
this is almost the same as taking $z'=\frac{2\gamma}{(4\beta)^{\frac1{\gamma}}}\frac{t_0-t}{\varepsilon^{\frac23}}$. These two substitutions lead to different systems, but essentially they are equivalent, because in the limit as $\varepsilon\rightarrow0^+$, asymptotically, solutions of these systems behave similarly. The first choice leads to simpler formulae, so we use the variable $z$ defined by \eqref{z}. One could also choose the sign differently, $\widetilde z:=\frac{\frac{t^{2\gamma}}{4\beta^2}-1}{\varepsilon^{\frac23}}=-z$, but our choice of sign will be more natural for the region $III$ where we use the same variable $z$.

Let us take
\begin{equation}\label{u II2}
    u_{II,1}(t)=u_{II,2}(z(t,\varepsilon))
\end{equation}
and substitute this into the system \eqref{system u II1}. We come to the system
\begin{equation}\label{system u II2}
    u_{II,2}'(z)=(\Lambda_{II,2}(z,\varepsilon)+S_{II,2}(z,\varepsilon)+R_{II,2}(z,\varepsilon))u_{II,2}(z)
\end{equation}
with
\begin{equation}\label{Lambda II2}
    \Lambda_{II,2}(z,\varepsilon):=
    \lambda_{II,2}(z,\varepsilon)
    \left(
      \begin{array}{cc}
        1 & 0 \\
        0 & -1 \\
      \end{array}
    \right),
\end{equation}
\begin{equation}\label{lambda II2}
    \lambda_{II,2}(z,\varepsilon):=\frac{\lambda_{II}(t(z,\varepsilon))t'(z,\varepsilon)}{\varepsilon}=\frac{-c_0\sqrt z}{(1-\varepsilon^{\frac23}z)^\varkappa},
\end{equation}
\begin{equation}\label{c_0 kappa}
    c_0:=\frac{t_0}{4\gamma},\ \varkappa:=\frac32-\frac1{2\gamma},
\end{equation}
\begin{equation}\label{S II2}
    S_{II,2}(z,\varepsilon):=S_{II}(t(z,\varepsilon))t'(z,\varepsilon)=
    -\frac1{4z(1-\varepsilon^{\frac23}z)}
    \left(
      \begin{array}{cc}
        1-\varepsilon^{\frac13}\sqrt z & -1-\varepsilon^{\frac13}\sqrt z \\
        -1+\varepsilon^{\frac13}\sqrt z & 1+\varepsilon^{\frac13}\sqrt z \\
      \end{array}
    \right)
\end{equation}
and
\begin{equation}\label{R II2}
    R_{II,2}(z,\varepsilon):=
    -\frac{t_0}{2\gamma\varepsilon^{\frac13}}
    \frac{T_{II}^{-1}(t(z,\varepsilon))R_{II}(t(z,\varepsilon),\varepsilon)T_{II}(t(z,\varepsilon))}
    {(1-\varepsilon^{\frac23}z)^{1-\frac1{2\gamma}}}.
\end{equation}
We consider the system \eqref{system u II2} on the interval $[Z_0,Z_2(\varepsilon)]$ with
\begin{equation}\label{Z 2}
    Z_2(\varepsilon)=\frac1{\varepsilon^{\frac23}}\left(1-\frac{t_{I-II}^{2\gamma}}{4\beta^2}\right)
\end{equation}
(so that $t_{II-III}(\varepsilon)=t_0(1-\varepsilon^{\frac23}Z_0)^{\frac1{2\gamma}}$), where the constant $Z_0$ should be large enough and the constant $t_{I-II}$ such that the distance $(t_0-t_{I-II})$ is small enough. We will choose these constants later.

Consider for a moment the free system. We need the next lemma to be formulated here, although its statement will automatically follow from our analysis in the proof of Lemma \ref{lem II result 1st part}, see Remark \ref{rem proof of Lemma II free system}.

\begin{lem}\label{lem II free system}
The system
\begin{equation}\label{system v II2}
    v_{II,2}'(z)=
    \left(
    -c_0\sqrt z
    \left(
      \begin{array}{cc}
        1 & 0 \\
        0 & -1 \\
      \end{array}
    \right)
    \right.
    \\
    \left.
    -\frac1{4z}
    \left(
      \begin{array}{cc}
        1 & -1 \\
        -1 & 1 \\
      \end{array}
    \right)
    \right)
    v_{II,2}(z)
\end{equation}
has for $z\in[1,+\infty)$ solutions $v_{II,2}^{\pm}(z)$ such that
\begin{equation}\label{II asymptotics v II2}
    v_{II,2}^{\pm}(z)=\frac{\exp\left({\mp}\frac{2c_0}3z^{\frac32}\right)}{z^{\frac14}}(e_{\pm}+o(1))
\end{equation}
as $z\rightarrow+\infty$.
\end{lem}

The main result for the region $II$ is the following lemma.

\begin{lem}\label{lem II result}
Let $c_0,\varkappa>0$ and let $R_{II,2}(z,\varepsilon)$ be given by \eqref{R II2} with the use of the expression \eqref{T II} and the definition \eqref{z} of $z$. There exist $Z_0>0$ and $t_{I-II}\in(0,t_0)$ such that, with $Z_2(\varepsilon)$ given by \eqref{Z 2}, if
\begin{equation}\label{condition R II integral}
    \int_{Z_0}^{Z_2(\varepsilon)}\frac{\|R_{II}(t(s,\varepsilon),\varepsilon)\|}{\sqrt{s}}ds=o\left(\varepsilon^{\frac23}\right)
   \text{ as }\varepsilon\rightarrow0^+,
\end{equation}
then for every sufficiently small $\varepsilon>0$ the system \eqref{system u II2} on the interval $[Z_0,Z_2(\varepsilon)]$ has two solutions $u_{II,2}^{\pm}(z,\varepsilon)$ such that, as $\varepsilon\rightarrow0^+$,
\begin{equation}\label{II answer}
    u_{II,2}^{\pm}(Z_2(\varepsilon),\varepsilon)=a_{II}^{\pm}\exp\left(\int_{Z_0}^{Z_2(\varepsilon)}
    \left(\mp\frac{c_0\sqrt s}{(1-\varepsilon^{\frac23}s)^{\varkappa}}-\frac1{4s(1\pm\varepsilon^{\frac13}\sqrt s)}\right)
    ds\right)(e_{\pm}+o(1)),
\end{equation}
where $a_{II}^{\pm}$ are positive constants and the vectors $e_{\pm}$ are given by \eqref{e +-}. Moreover, $u_{II,2}^{\pm}(z,\varepsilon)\rightarrow v_{II,2}^{\pm}(z)$ as $\varepsilon\rightarrow0^+$ for every fixed $z\ge Z_0$, where $v_{II,2}^{\pm}$ are defined in Lemma \ref{lem II free system}.
\end{lem}

Let us rewrite the condition \eqref{condition R II integral} in terms of the remainder $R_{II,2}$.

\begin{lem}\label{lem II remainder}
Under the conditions of Lemma \ref{lem II result}, for every $z_0>0$ and $\nu\in(0,1)$ with $z_2(\varepsilon)=\frac{\nu}{\varepsilon^{\frac23}}$, if
\begin{equation*}
 \int_{z_0}^{z_2(\varepsilon)}\frac{\|R_{II}(t(s,\varepsilon),\varepsilon)\|}{\sqrt{s}}ds=o\left(\varepsilon^{\frac23}\right),
   \text{ then }
   \int_{z_0}^{z_2(\varepsilon)}\|R_{II,2}(s,\varepsilon)\|ds\rightarrow0
\end{equation*}
as $\varepsilon\rightarrow0^+$.
\end{lem}

\begin{proof}
Let
\begin{equation}\label{t nu}
    t_{\nu}:=t_0(1-\nu)^{\frac1{2\gamma}}.
\end{equation}
From the expression \eqref{T II} for $T_{II}$ it is clear that there exists $c_{19}>0$ such that $\|T_{II}(t)\|<c_{19}$ for every $t\in[t_{\nu},t_0]$ and that
\begin{equation*}
    \text{det}\,T_{II}(t)\sim2\sqrt{1-\frac{t^{2\gamma}}{4\beta^2}}=2\varepsilon^{\frac13}\sqrt{z(t,\varepsilon)}\text{ as }t\rightarrow t_0,
\end{equation*}
therefore there exists $c_{20}>0$ such that
\begin{equation*}
    \|T_{II}^{-1}(t(z,\varepsilon))\|<\frac{c_{20}}{\varepsilon^{\frac13}\sqrt z}
    \text{ for every }\varepsilon\in U\text{ and }z\in[z_0,z_2(\varepsilon)].
\end{equation*}
Hence from the definition \eqref{R II2} of $R_{II,2}$  one has that with some $c_{21}>0$
\begin{equation*}
    \int_{z_0}^{z_2(\varepsilon)}\|R_{II,2}(s,\varepsilon)\|ds
    <\frac{c_{21}}{\varepsilon^{\frac23}}\int_{z_0}^{z_2(\varepsilon)}\frac{\|R_2^+(t(s,\varepsilon),\varepsilon)\|ds}{\sqrt s}
\end{equation*}
which goes to zero as $\varepsilon\rightarrow0^+$ by the hypothesis.
\end{proof}

Now let us see that the condition \eqref{condition R II integral} is satisfied, if $R_{II}=R_2^+$.

\begin{lem}\label{lem II R 2+ estimate}
Let $R_2^+(t,\varepsilon)$ be given by \eqref{R 2 pm} and $t(z,\varepsilon)$ be defined by \eqref{z}.  Let the conditions \eqref{model problem conditions} and \eqref{model problem r} hold. Then for every $z_0>0$ and $\nu\in(0,1)$ with $z_2(\varepsilon)=\frac{\nu}{\varepsilon^{\frac23}}$, the following estimate holds:
\begin{equation*}
   \int_{z_0}^{z_2(\varepsilon)}\frac{\|R_2^+(t(s,\varepsilon),\varepsilon)\|}{\sqrt{s}}ds=O\left(\varepsilon^{\frac23}\varepsilon_0^{\frac{\alpha_r}{\gamma}}\right)
    \text{ as }\varepsilon\rightarrow0^+.
\end{equation*}
\end{lem}

\begin{proof}
Using the equalities \eqref{R 2 pm} and the estimate of the norm of $R$ from the conditions \eqref{model problem initial condition} we have
\begin{equation*}
    \int_{z_0}^{z_2(\varepsilon)}\frac{\|R_2^+(t(s,\varepsilon),\varepsilon)\|ds}{\sqrt s}
    <
    \frac{1}{\varepsilon_0}\int_{z_0}^{z_2(\varepsilon)}
    \left(\varepsilon_0^{-\frac1{\gamma}}t(s,\varepsilon)\right)\frac{ds}{\sqrt s}.
\end{equation*}
Since the derivative of $z$ is $\frac{dz(t,\varepsilon)}{dt}=-\frac{\gamma t^{2\gamma-1}}{2\beta^2\varepsilon^{\frac23}}$, one has
\begin{equation*}
    \frac{1}{\varepsilon_0}\int_{z_0}^{z_2(\varepsilon)}
    r\left(\varepsilon_0^{-\frac1{\gamma}}t(s,\varepsilon)\right)\frac{ds}{\sqrt s}
    <
    \frac{\gamma}{2\beta^2\varepsilon_0\varepsilon^{\frac13}}\int_{t_{\nu}}^{t_0}
    r\left(\varepsilon_0^{-\frac1{\gamma}}t\right)\frac{t^{2\gamma-1}dt}{\sqrt{1-\frac{t^{2\gamma}}{4\beta^2}}},
\end{equation*}
where $t_{\nu}$ is given by \eqref{t nu}. It follows that there exists $c_{22}>0$ such that
\begin{equation*}
    \frac{\gamma}{2\beta^2\varepsilon_0\varepsilon^{\frac13}}\int_{t_{\nu}}^{t_0}
    r\left(\varepsilon_0^{-\frac1{\gamma}}t\right)\frac{t^{2\gamma-1}dt}{\sqrt{1-\frac{t^{2\gamma}}{4\beta^2}}}
    <\frac{c_{22}\varepsilon^{\frac23}}{\varepsilon_0^{\frac1{\gamma}}}\int_{t_{\nu}}^{t_0}
    r\left(\varepsilon_0^{-\frac1{\gamma}}t\right)\frac{dt}{\sqrt{t_0-t}}.
\end{equation*}
Using the expression \eqref{model problem r} for $r$ we arrive at the estimate
\begin{equation*}
    \int_{z_0}^{z_2(\varepsilon)}\frac{\|R_2^+(t(s,\varepsilon),\varepsilon)\|ds}{\sqrt s}
    <c_rc_{22}\varepsilon^{\frac23}\varepsilon_0^{\frac{\alpha_r}{\gamma}}\int_{t_{\nu}}^{t_0}
    \frac{dt}{t^{1+\alpha_r}\sqrt{t_0-t}}=O\left(\varepsilon^{\frac23}\varepsilon_0^{\frac{\alpha_r}{\gamma}}\right)
\end{equation*}
as $\varepsilon\rightarrow0^+$. This completes the proof.
\end{proof}

To prove Lemma \ref{lem II result} we need to further divide the interval $[Z_0,Z_2(\varepsilon)]$ into two parts by the point $Z_1(\varepsilon)$ such that  $Z_1(\varepsilon)\rightarrow+\infty$ and $Z_1(\varepsilon)=o(\varepsilon^{-\frac23})$ as $\varepsilon\rightarrow0^+$ which we will choose later, in \eqref{Z 1}. On the first subinterval $[Z_0,Z_1(\varepsilon)]$ solutions of the system \eqref{system u II2} behave like solutions of the free system \eqref{system v II2} which does not contain $\varepsilon$. On the second subinterval $[Z_1(\varepsilon),Z_2(\varepsilon)]$ one cannot neglect the terms $\varepsilon^{\frac13}\sqrt z$ and $\varepsilon^{\frac23}z$, so the free system should depend on $\varepsilon$. Nevertheless, the answer in the second subinterval is even simpler, because $Z_1(\varepsilon)\rightarrow+\infty$ and solutions are already in their asymptotic regime as $z\rightarrow+\infty$.

\begin{rem}
Since we are interested only in the behaviour as $\varepsilon\rightarrow0^+$ we do not need to ensure that the  inequality $Z_0<Z_1(\varepsilon)<Z_2(\varepsilon)$ holds for every $\varepsilon\in U$. It is enough if it holds for sufficiently small values of $\varepsilon$.
\end{rem}

For both subintervals one has to prove that the remainder $R_{II,2}$ does not affect the asymptotics. Some difficulty in showing this is that the second term of the coefficient matrix has off-diagonal entries. To get rid of them we use the Harris--Lutz transformation, although slightly different for subintervals $[Z_0,Z_1(\varepsilon)]$ and $[Z_1(\varepsilon),Z_2(\varepsilon)]$. Let us see how it works generally in the hyperbolic case \cite{Harris-Lutz-1975}.

\subsection{Formulae for the Harris--Lutz transformation}\label{subsection Harris-Lutz}
Suppose that we start with the linear differential system in $\mathbb C^2$
\begin{equation*}\label{Harris--Lutz given}
    u'=(\Lambda+S+R)u,
\end{equation*}
where $\Lambda=\text{diag}\{\lambda_1,\lambda_2\}$ and $\Re\lambda_1>\Re\lambda_2$. The aim is to transform it to the system
\begin{equation*}\label{Harris--Lutz result}
    u_1'=(\Lambda+\text{diag}\,S+R_1)u_1
\end{equation*}
by the substitution $u=(I+\widehat T)u_1$.  We suppose that $(I+\widehat T)$ is invertible. This gives
\begin{equation*}
    (I+\widehat T)u'=(\Lambda+\Lambda \widehat T-\widehat T'+S+(S\widehat T+R(I+\widehat T)))u_1
\end{equation*}
and
\begin{equation*}
    u_1'=(\Lambda+\underbrace{[\Lambda,\widehat T]-\widehat T'+S}_{=\text{diag}\,S}+R_1)u_1,
\end{equation*}
where $[\cdot,\cdot]$ is the commutator of matrices and
\begin{equation*}
    R_1=(I+\widehat T)^{-1}((R(I+\widehat T)+S\widehat T)
    -\widehat T(\Lambda \widehat T-\widehat T'+S)
    +\widehat T^2\Lambda).
\end{equation*}
Let $S_d:=\text{diag}\,S$, $S_{ad}:=S-S_d$. We need  the following equality:
\begin{equation*}
    [\Lambda,\widehat T]-\widehat T'=-S_{ad}.
\end{equation*}
This, in particular, implies
\begin{equation}\label{Harris--Lutz R 1}
    R_1=(I+\widehat T)^{-1}R(I+\widehat T)+(I+\widehat T)^{-1}(S\widehat T-\widehat T S_d).
\end{equation}
If one represents $\widehat T$ as
\begin{equation*}
    \widehat T(x)=\exp\left(\int_0^x\Lambda\right)\widetilde T(x)\exp\left(-\int_0^x\Lambda\right),
\end{equation*}
differentiation gives
\begin{equation*}
    \widehat T'=[\Lambda,\widehat T]+\exp\left(\int_0^x\Lambda\right)\widetilde T'\exp\left(-\int_0^x\Lambda\right),
\end{equation*}
and so one needs to solve the equation
\begin{equation*}
    \exp\left(\int_0^x\Lambda\right)\widetilde T'\exp\left(-\int_0^x\Lambda\right)=S_{ad}.
\end{equation*}
Its solution is
\begin{equation*}
    \widetilde T(x)=
    \left(
    \begin{array}{cc}
      0
      &
      -\int_x^{\infty}S_{12}(x')\exp(\int_0^{x'}(\lambda_2-\lambda_1))dx'
      \\
      \int_0^xS_{21}(x')\exp(\int_0^{x'}(\lambda_1-\lambda_2))dx'
      &
      0
    \end{array}
    \right),
\end{equation*}
where $S_{12}$ and $S_{21}$ are the off-diagonal entries of the matrix $S$ and the choice of the domain of integration is determined by the sign of the real part of the exponent. This eventually gives
\begin{equation}\label{Harris--Lutz T tilde}
    \widehat T(x)=
    \left(
      \begin{array}{cc}
        0
        &
        -\int_x^{\infty}S_{12}(x')\exp(\int_x^{x'}(\lambda_2-\lambda_1))dx'
        \\
        \int_0^xS_{21}(x')\exp(\int_{x'}^{x}(\lambda_2-\lambda_1))dx'
        &
        0
      \end{array}
    \right),
\end{equation}
where both exponential terms are less than one in modulus. If $R$ is summable and $\Lambda,V$ are such that $\widehat T$ goes to zero at infinity, then $R_1$ can be also summable, which means that the transformation effectively eliminates the off-diagonal terms of $V$.

According to the argument above, for the system \eqref{system u II2} we take for $z\in[Z_0,Z_2(\varepsilon)]$
\begin{equation}\label{T II hat}
    \widehat T_{II}(z,\varepsilon)=
    \left(
      \begin{array}{cc}
        0 & \widehat T_{II_{12}}(z,\varepsilon) \\
        \widehat T_{II_{21}}(z,\varepsilon) & 0 \\
      \end{array}
    \right)
\end{equation}
with
\begin{equation*}
        \widehat T_{II_{12}}(z,\varepsilon):=
        \int_{Z_0}^{z}\frac{ds}{4s(1-\varepsilon^{\frac13}\sqrt s)}
        \exp
        \left(-\int_{s}^{z}\frac{2c_0\sqrt{\sigma}d\sigma}{(1-\varepsilon^{\frac23}\sigma)^{\varkappa}}\right),
\end{equation*}
\begin{equation*}
        \widehat T_{II_{21}}(z,\varepsilon):=
        -\int_{z}^{Z_2(\varepsilon)}\frac{ds}{4s(1+\varepsilon^{\frac13}\sqrt s)}
        \exp
        \left(-\int_{z}^{s}\frac{2c_0\sqrt{\sigma}d\sigma}{(1-\varepsilon^{\frac23}\sigma)^{\varkappa}}\right).
\end{equation*}
Note that we have interchanged the domains of integration according to the signs of the entries of $\Lambda_{II}$. Note also that the argument works with $Z_2(\varepsilon)$ in the upper limit instead of $+\infty$ for the lower-left entry.

\begin{lem}\label{lem II estimate T hat}
There exist $t_{I-II}\in(0,t_0)$ and $c_{II}>0$ such that for every $\varepsilon\in U\cup\{0\}$ and $z\in[1,Z_2(\varepsilon)]$ with $Z_2(\varepsilon)$ given by \eqref{Z 2} the matrix $\widehat T_{II}(z,\varepsilon)$ given by \eqref{T II hat} satisfies the following estimate:
\begin{equation}\label{II estimate of T hat}
\| \widehat T_{II}(z,\varepsilon)\|<\frac{c_{II}}{z^{\frac32}}.
\end{equation}
\end{lem}

\begin{proof}
Let us choose the point $t_{I-II}$ so close to $t_0$ that for every $z\in[0,Z_2(\varepsilon)]$ it holds that $\varepsilon^{\frac13}\sqrt z<\frac12$, which is equivalent to
\begin{equation*}
    t_0\left(\frac34\right)^{\frac1{2\gamma}}<t_{I-II}<t_0.
\end{equation*}
So let us take
\begin{equation}\label{t I-II}
    t_{I-II}:=t_0\left(\frac45\right)^{\frac1{2\gamma}}.
\end{equation}
Then
\begin{equation*}
    |\widehat T_{II_{12}}(z,\varepsilon)|
    <\int_{Z_0}^z\frac{\exp\left(-\int_s^z2c_0\sqrt\sigma d\sigma\right)ds}{2s}
    \\
    =\exp\left(-\frac{4c_0}{3}z^{\frac32}\right)\int_{Z_0}^z\exp\left(\frac{4c_0}{3}s^{\frac32}\right)\frac{ds}{2s}.
\end{equation*}
The last expression does not depend on $\varepsilon$ and can be estimated for large $z$ using integration by parts. Indeed, for $z>2Z_0$ one has:
\begin{multline*}
    \exp\left(-\frac{4c_0}{3}z^{\frac32}\right)\int_{Z_0}^z\exp\left(\frac{4c_0}{3}s^{\frac32}\right)\frac{ds}{2s}
    \\
    =\frac{1}{4c_0z^{\frac32}}
    -\frac{\exp\left(\frac{4c_0}{3}\left({Z_0}^{\frac32}-z^{\frac32}\right)\right)}{4c_0Z_0^{\frac32}}
    +\frac{3}{8c_0}\exp\left(-\frac{4c_0}{3}z^{\frac32}\right)\int_{Z_0}^z\frac{\exp\left(\frac{4c_0}{3}s^{\frac32}\right)ds}{s^{\frac52}}
    \\
    <\frac1{4c_0z^{\frac32}}+
    \frac{3}{8c_0}\exp\left(-\frac{4c_0}{3}z^{\frac32}\right)
    \left(
    \int_{Z_0}^{\frac z2}+\int_{\frac z2}^z
    \right)
    \frac{\exp\left(\frac{4c_0}{3}s^{\frac32}\right)ds}{s^{\frac52}}
    \\
    =\frac1{4c_0z^{\frac32}}
    +O\left(\exp\left(-\frac{4c_0}{3}z^{\frac32}\left(1-\frac1{2^{\frac32}}\right)\right)+\frac1{z^{\frac32}}\right)
    =O\left(\frac1{z^{\frac32}}\right)
\end{multline*}
as $z\rightarrow+\infty$, and hence we have a uniform estimate for the upper-right entry. For the lower-left entry one has:
\begin{multline*}
    |\widehat T_{II_{21}}(z,\varepsilon)|
    <\int_{z}^{Z_2(\varepsilon)}\frac{ds}{4s}\exp\left(-\int_z^s2c_0\sqrt\sigma d\sigma\right)
    \\
    <\int_{z}^{+\infty}\frac{ds}{4s}\exp\left(-\int_z^s2c_0\sqrt\sigma d\sigma\right)
    =\exp\left(\frac{4c_0}{3}z^{\frac32}\right)\int_z^{+\infty}\exp\left(-\frac{4c_0}{3}s^{\frac32}\right)\frac{ds}{4s}
    \\
    =\frac1{8c_0z^{\frac32}}-\frac3{16c_0}\int_z^{+\infty}\frac{\exp\left(-\frac{4c_0}{3}(s^{\frac32}-z^{\frac32})\right)ds}{s^{\frac52}}<\frac1{8c_0z^{\frac32}}
\end{multline*}
for every $z\in[Z_0,Z_2(\varepsilon)]$ and $\varepsilon\in U\cup\{0\}$. Therefore we have the estimate \eqref{II estimate of T hat}.
\end{proof}

On the first part $[Z_0,Z_1(\varepsilon)]$ of the region $II$ we treat the system \eqref{system u II2} as a perturbation of the free system \eqref{system v II2}. Let us see how to choose $Z_1(\varepsilon)$ in order to make this possible. Rewrite \eqref{system u II2} as
\begin{equation*}
    u_{II,2}'(z)=\left(
    -c_0\sqrt z
    \left(
      \begin{array}{cc}
        1 & 0 \\
        0 & -1 \\
      \end{array}
    \right)
    \right.
    \\
    \left.
    -\frac1{4z}
    \left(
      \begin{array}{cc}
        1 & -1 \\
        -1 & 1 \\
      \end{array}
    \right)
    +\widetilde R_{II,2}(z,\varepsilon)
    \right)u_{II,2}(z),
\end{equation*}
where
\begin{multline*}
    \widetilde R_{II,2}(z,\varepsilon)=R_{II,2}(z,\varepsilon)
    \\
    -c_0\sqrt{z}
    \left(
      \begin{array}{cc}
        1 & 0 \\
        0 & -1 \\
      \end{array}
    \right)
    \left(\frac1{(1-\varepsilon^{\frac23}z)^{\varkappa}}-1\right)
    +\frac{\varepsilon^{\frac13}}{4\sqrt{z}}
    \left(
      \begin{array}{cc}
        \frac1{1+\varepsilon^{\frac13}\sqrt z} & -\frac1{1-\varepsilon^{\frac13}\sqrt z} \\
        -\frac1{1+\varepsilon^{\frac13}\sqrt z} & \frac1{1-\varepsilon^{\frac13}\sqrt z} \\
      \end{array}
    \right).
\end{multline*}
From this we see that with some $c_{23}>0$ for every $z\in[Z_0,Z_2(\varepsilon)]$
\begin{equation*}
    \|\widetilde R_{II,2}(z,\varepsilon)\|\le\|R_{II,2}(z,\varepsilon)\|
    +c_{23}\left(\varepsilon^{\frac23}z^{\frac32}+\frac{\varepsilon^{\frac13}}{\sqrt z}\right).
\end{equation*}
Since by Lemma \ref{lem II remainder} one has
\begin{equation*}
    \int_{Z_0}^{Z_2(\varepsilon)}\|R_{II,2}(s,\varepsilon)\|ds\rightarrow0\text{ as }\varepsilon\rightarrow0^+,
\end{equation*}
the condition on $Z_1(\varepsilon)$ is the following:
\begin{equation*}
    \int_{Z_0}^{Z_1(\varepsilon)}\left(\varepsilon^{\frac23}s^{\frac32}+\frac{\varepsilon^{\frac13}}{\sqrt s}\right)ds\rightarrow0
    \text{ as }\varepsilon\rightarrow0^+.
\end{equation*}
This is equivalent to the condition
\begin{equation*}
    \varepsilon^{\frac23}(Z_1(\varepsilon))^{\frac52}+\varepsilon^{\frac13}(Z_1(\varepsilon))^{\frac12}\rightarrow0,
\end{equation*}
which in turn is equivalent to $Z_1(\varepsilon)=o\left(\varepsilon^{-\frac4{15}}\right)$. Therefore let us take
\begin{equation}\label{Z 1}
    Z_1(\varepsilon):=\frac{1}{\varepsilon^{\frac15}}.
\end{equation}
This ensures that
\begin{equation}\label{II remainder estimate 1st part}
    \int_{Z_0}^{Z_1(\varepsilon)}\|\widetilde R_{II,2}(s,\varepsilon)\|ds\rightarrow0\text{ as }\varepsilon\rightarrow0^+.
\end{equation}
Now we can obtain results for the subintervals $[Z_0,Z_1(\varepsilon)]$ and $[Z_1(\varepsilon),Z_2(\varepsilon)]$.

\begin{lem}\label{lem II result 1st part}
Let the conditions of Lemma \ref{lem II result} hold and let $Z_1(\varepsilon)$ be given by \eqref{Z 1}. There exists $Z_0>0$ such that if
\begin{equation*}
\int_{Z_0}^{Z_1(\varepsilon)}\frac{\|R_{II}(t(s,\varepsilon),\varepsilon)\|}{\sqrt{s}}ds=o\left(\varepsilon^{\frac23}\right)
   \text{ as }\varepsilon\rightarrow0^+,
\end{equation*}
then for every sufficiently small $\varepsilon>0$ the system \eqref{system u II2} on the interval $[Z_0,Z_1(\varepsilon)]$ has two solutions $\widetilde u_{II,2}^{\pm}(z,\varepsilon)$ such that, as $\varepsilon\rightarrow0^+$,
\begin{equation}\label{II answer 1st part}
    \widetilde u_{II,2}^{\pm}(Z_1(\varepsilon),\varepsilon)
    =\frac{\exp\left(\mp\frac{2c_0}3(Z_1(\varepsilon))^{\frac32}\right)}{(Z_1(\varepsilon))^{\frac14}}
    (e_{\pm}+o(1)),
\end{equation}
where the vectors $e_{\pm}$ are given by \eqref{e +-}. Moreover, $\widetilde u_{II,2}^{\pm}(z,\varepsilon)\rightarrow v_{II,2}^{\pm}(z)$ as $\varepsilon\rightarrow0^+$ for every fixed $z\ge Z_0$, where $v_{II,2}^{\pm}$ are defined in Lemma \ref{lem II free system}.
\end{lem}

\begin{proof}
Let us choose
\begin{equation}\label{Z 0}
    Z_0:=\max\{(2c_{II})^{\frac23},(2c_{IV})^{\frac23}\}
\end{equation}
where $c_{II}$ is defined in Lemma \ref{lem II estimate T hat} and $c_{IV}$ will be defined in Lemma \ref{lem IV estimate T hat} independently. We now only need to know that $Z_0\ge(2c_{II})^{\frac23}$: this ensures that $\|\widehat T_{II}(z,\varepsilon)\|<\frac12$ for every $z\in[Z_0,Z_2(\varepsilon)]$ by Lemma \ref{lem II estimate T hat} and hence $(I+\widehat T_{II}(z,0))$ is invertible. Take
\begin{equation}\label{u II3}
    u_{II,2}(z)=(I+\widehat T_{II}(z,0))u_{II,3}(z).
\end{equation}
According to the argument for the Harris--Lutz transformation and due to the formula \eqref{Harris--Lutz R 1} this leads to the system
\begin{equation}\label{system u II3}
    u_{II,3}'(z)=\left(
    -c_0\sqrt z
    \left(
      \begin{array}{cc}
        1 & 0 \\
        0 & -1 \\
      \end{array}
    \right)
    \right.
    \\
    \left.
    -\frac1{4z}I+Q_{II,3}(z)+R_{II,3}(z,\varepsilon)
    \right)u_{II,3}(z)
\end{equation}
with
\begin{equation*}
    Q_{II,3}(z):=(I+\widehat T_{II}(z,0))^{-1}(S_{II,2}(z,0)\widehat T_{II}(z,0)-\widehat T_{II}(z,0) \text{diag}\,S_{II,2}(z,0))
\end{equation*}
and
\begin{equation}\label{R II3}
    R_{II,3}(z,\varepsilon):=(I+\widehat T_{II}(z,0))^{-1}R_{II,2}(z,\varepsilon)(I+\widehat T_{II}(z,0)).
\end{equation}
From the expression \eqref{S II2} for $S_{II,2}(z,0)$ and the estimate \eqref{II estimate of T hat} for $T_{II}(z,0)$ we have
\begin{equation}\label{II estimate Q II3}
    Q_{II,3}(z)=O\left(\frac1{z^{\frac52}}\right)\text{ as }z\rightarrow+\infty.
\end{equation}
Consider the free system
\begin{equation}\label{system v II3}
    v_{II,3}'(z)=\left(
    -c_0\sqrt z
    \left(
      \begin{array}{cc}
        1 & 0 \\
        0 & -1 \\
      \end{array}
    \right)
    \right.
    \\
    \left.
    -\frac1{4z}I+Q_{II,3}(z)
    \right)v_{II,3}(z).
\end{equation}
Since $\int_1^{\infty}\|Q_{II,3}(s)\|ds<\infty$, the asymptotic Levinson theorem is applicable and yields the existence of two solutions $v_{II,3}^{\pm}$ of the system \eqref{system v II3} with the asymptotics
\begin{equation}\label{v II3 pm}
    v_{II,3}^{\pm}(z)=\frac{\exp\left(\mp\frac{2c_0}3z^{\frac32}\right)}{z^{\frac14}}(e_{\pm}+o(1))\text{ as }z\rightarrow+\infty.
\end{equation}

\begin{rem}\label{rem proof of Lemma II free system}
Now we can easily obtain the proof of Lemma \ref{lem II free system}.
\end{rem}

\begin{proof}[Proof of Lemma \ref{lem II free system}]
The following definition of solutions $v_{II,2}^{\pm}$ of the system \eqref{system v II2} immediately gives the asymptotics \eqref{II asymptotics v II2} as stated in the lemma:
\begin{equation}\label{v II2 pm def}
    v_{II,2}^{\pm}(z):=(I+\widehat T_{II}(z,0)) v_{II,3}^{\pm}(z)=\frac{\exp\left(\mp\frac{2c_0}3z^{\frac32}\right)}{z^{\frac14}}(e_{\pm}+o(1))\text{ as }z\rightarrow+\infty.
\end{equation}
\end{proof}

From the integral estimate of the remainder \eqref{II remainder estimate 1st part} using Lemma \ref{lem II estimate T hat} we have the following estimate of $R_{II,3}$ defined in \eqref{R II3}
\begin{equation}\label{II estimate R II3}
    \int_{Z_0}^{Z_1(\varepsilon)}\|R_{II,3}(s,\varepsilon)\|ds\rightarrow0\text{ as }\varepsilon\rightarrow0^+.
\end{equation}
Let us prove that solutions $u_{II,3}$ of the system \eqref{system u II3} behave  on the interval $[Z_0,Z_1(\varepsilon)]$ similarly to solutions $v_{II,3}$ of the free system \eqref{system v II3}. This will imply that solutions $u_{II,2}$ of the system \eqref{system u II2} behave similarly to solutions $v_{II,2}$ of the free system \eqref{system v II2}. To do this make variation of parameters
\begin{equation}\label{u II4}
    u_{II,3}(z)=\frac1{z^{\frac14}}
    \left(
      \begin{array}{cc}
        \exp\left(-\frac{2c_0}{3}z^{\frac32}\right) & 0 \\
        0 & \exp\left(\frac{2c_0}{3}z^{\frac32}\right) \\
      \end{array}
    \right)
    u_{II,4}(z),
\end{equation}
which turns the system \eqref{system u II3} into the system
\begin{multline}\label{system u II4}
    u_{II,4}'(z)=
    \left(
      \begin{array}{cc}
        \exp\left(\frac{2c_0}{3}z^{\frac32}\right) & 0 \\
        0 & \exp\left(-\frac{2c_0}{3}z^{\frac32}\right) \\
      \end{array}
    \right)
    \\
    \times
    (Q_{II,3}(z)+R_{II,3}(z,\varepsilon))
    \left(
      \begin{array}{cc}
        \exp\left(-\frac{2c_0}{3}z^{\frac32}\right) & 0 \\
        0 & \exp\left(\frac{2c_0}{3}z^{\frac32}\right) \\
      \end{array}
    \right)
    u_{II,4}(z).
\end{multline}
At this point we need to consider separately the ``small'' and the ``large'' solutions.
\\
\emph{``Small'' solution.} Consider the following particular solution of the system \eqref{system u II4}:
\begin{multline*}
    u_{II,4}^+(z,\varepsilon)=e_+-\int_z^{Z_1(\varepsilon)}
    \left(
      \begin{array}{cc}
        \exp\left(\frac{2c_0}{3}s^{\frac32}\right) & 0 \\
        0 & \exp\left(-\frac{2c_0}{3}s^{\frac32}\right) \\
      \end{array}
    \right)
    \\
    \times
    (Q_{II,3}(s)+R_{II,3}(s,\varepsilon))
    \left(
      \begin{array}{cc}
        \exp\left(-\frac{2c_0}{3}s^{\frac32}\right) & 0 \\
        0 & \exp\left(\frac{2c_0}{3}s^{\frac32}\right) \\
      \end{array}
    \right)
    u_{II,4}^+(s,\varepsilon)ds.
\end{multline*}
Returning to $u_{II,3}$ we have:
\begin{multline*}
    u_{II,3}^+(z,\varepsilon)=\frac{\exp\left(-\frac{2c_0}{3}z^{\frac32}\right)}{z^{\frac14}}e_+
    \\
    -\int_z^{Z_1(\varepsilon)}\left(\frac sz\right)^{\frac14}
    \left(
      \begin{array}{cc}
        \exp\left(\frac{2c_0}{3}(s^{\frac32}-z^{\frac32})\right) & 0 \\
        0 & \exp\left(-\frac{2c_0}{3}(s^{\frac32}-z^{\frac32})\right) \\
      \end{array}
    \right)
    \\
    \times
    (Q_{II,3}(s)+R_{II,3}(s,\varepsilon))
    u_{II,3}^+(s,\varepsilon)ds.
\end{multline*}
Normalising the solution
\begin{equation}\label{u II5}
    u_{II,3}^+(z,\varepsilon)=\frac{\exp\left(-\frac{2c_0}{3}z^{\frac32}\right)}{z^{\frac14}}u_{II,5}^+(z,\varepsilon)
\end{equation}
we come to the integral equation
\begin{multline}\label{system u II5}
    u_{II,5}^+(z,\varepsilon)=e_+-\int_z^{Z_1(\varepsilon)}
    \left(
      \begin{array}{cc}
        1 & 0 \\
        0 & \exp\left(-\frac{4c_0}{3}(s^{\frac32}-z^{\frac32})\right) \\
      \end{array}
    \right)
    \\
    \times
    (Q_{II,3}(s)+R_{II,3}(s,\varepsilon))
    u_{II,5}^+(s,\varepsilon)ds.
\end{multline}
Repeating the same manipulations with the system \eqref{system v II3} we get
\begin{equation}\label{system v II5}
    v_{II,5}^+(z)=e_+-\int_z^{+\infty}
    \left(
      \begin{array}{cc}
        1 & 0 \\
        0 & \exp\left(-\frac{4c_0}{3}(s^{\frac32}-z^{\frac32})\right) \\
      \end{array}
    \right)
    Q_{II,3}(s)v_{II,5}^+(s)ds
\end{equation}
with
\begin{equation*}
    \widetilde v_{II,3}^{\,+}(z):=\frac{\exp\left(-\frac{2c_0}{3}z^{\frac32}\right)}{z^{\frac14}}v_{II,5}^+(z)
\end{equation*}
in place of \eqref{u II5}. Note that we need to formally distinguish this $\widetilde v_{II,3}^{\,+}$ from $v_{II,3}^{\,+}$ of \eqref{v II3 pm}, because we have not yet shown that they are the same. Subtracting \eqref{system v II5} from \eqref{system u II5} we obtain the equality
\begin{multline}\label{II eq 1}
    u_{II,5}^+(z,\varepsilon)-v_{II,5}^+(z)=
    -\int_z^{Z_1(\varepsilon)}
    \left(
      \begin{array}{cc}
        1 & 0 \\
        0 & \exp\left(-\frac{4c_0}{3}(s^{\frac32}-z^{\frac32})\right) \\
      \end{array}
    \right)
    \\
    \times
    R_{II,3}(s,\varepsilon)u_{II,5}^+(s,\varepsilon)ds
    +\int_{Z_1(\varepsilon)}^{+\infty}
    \left(
      \begin{array}{cc}
        1 & 0 \\
        0 & \exp\left(-\frac{4c_0}{3}(s^{\frac32}-z^{\frac32})\right) \\
      \end{array}
    \right)
    Q_{II,3}(s)
    v_{II,5}^+(s)ds
    \\
    -\int_z^{Z_1(\varepsilon)}
    \left(
      \begin{array}{cc}
        1 & 0 \\
        0 & \exp\left(-\frac{4c_0}{3}(s^{\frac32}-z^{\frac32})\right) \\
      \end{array}
    \right)
    Q_{II,3}(s)(u_{II,5}^+(s,\varepsilon)-v_{II,5}^+(s))ds.
\end{multline}
The following variant of Gronwall lemma helps to obtain an estimate on $u_{II,5}^+-v_{II,5}^+$.

\begin{lem}\label{lem II tech 1}
Let $-\infty\le N_1<N_2\le+\infty,v\in L_{\infty}(N_1,N_2)$, let $K(z,s)$ be a measurable matrix-valued function defined for $N_1<z<s<N_2$  (or $N_1<s<z<N_2$) and such that for every $z,s$ it satisfies the estimate $\|K(z,s)\|<k(s)$ with some $k\in L_1(N_1,N_2)$. Then the equation
\begin{equation*}
    u(z)=v(z)-\int_z^{N_2}K(z,s)u(s)ds
\end{equation*}
(or the equation
\begin{equation*}
    u(z)=v(z)+\int_{N_1}^zK(z,s)u(s)ds,
\end{equation*}
respectively) has the unique solution in $L_{\infty}(N_1,N_2)$ which satisfies the estimate
\begin{equation*}
    \|u\|_{L_{\infty}(N_1,N_2)}\le\exp\left(\int_{N_1}^{N_2}k(s)ds\right)\|v\|_{L_{\infty}(N_1,N_2)}.
\end{equation*}
\end{lem}

\begin{proof}
One has to consider the operator $\mathcal K$ in $L_{\infty}(N_1,N_2)$ defined as
\begin{equation*}
    \mathcal K:u\mapsto-\int_z^{N_2}K(z,s)u(s)ds\text{ (or }u\mapsto\int_z^{N_2}K(z,s)u(s)ds\text{, respectively) }
\end{equation*}
and to check that $\|\mathcal K\|_{\mathcal B(L_{\infty}(N_1,N_2))}\le\|k\|_{L_1(N_1,N_2)}$, that
\begin{equation*}
    \|\mathcal K^n\|_{\mathcal B(L_{\infty}(N_1,N_2))}\le\frac{\|\mathcal K\|_{\mathcal B(L_{\infty}(N_1,N_2))}^n}{n!}
\end{equation*}
and that hence
\begin{equation*}
    \|(I-\mathcal K^n)^{-1}\|_{\mathcal B(L_{\infty}(N_1,N_2))}\le\exp(\|k\|_{L_1(N_1,N_2)}),
\end{equation*}
which completes the proof.
\end{proof}

The lemma immediately yields for the equation \eqref{system u II5}:
\begin{equation}\label{II estimate u II5}
    \sup_{z\in[Z_0,Z_1(\varepsilon)]}\|u_{II,5}^+(z,\varepsilon)\|\le\exp\left(\int_{Z_0}^{Z_1(\varepsilon)}\|Q_{II,3}(s)+R_{II,3}(s,\varepsilon)\|ds\right),
\end{equation}
for the equation \eqref{system v II5}:
\begin{equation}\label{II estimate v II5}
    \sup_{z\in[Z_1(\varepsilon),+\infty)}\|v_{II,5}^+(z)\|\le\exp\left(\int_{Z_1(\varepsilon)}^{+\infty}\|Q_{II,3}(s)\|ds\right),
\end{equation}
and finally for the equality \eqref{II eq 1}:
\begin{multline}\label{II eq 2}
    \sup_{z\in[Z_0,Z_1(\varepsilon)]}\|u_{II,5}^+(z,\varepsilon)-v_{II,5}^+(z)\|
    \le\exp\left(\int_{Z_0}^{Z_1(\varepsilon)}\|Q_{II,3}(s)\|ds\right)
    \\
    \times
    \left(
    \sup_{z\in[Z_0,Z_1(\varepsilon)]}\|u_{II,5}^+(z,\varepsilon)\|\int_{Z_0}^{Z_1(\varepsilon)}\|R_{II,3}(s,\varepsilon)\|ds
    \right.
    \\
    \left.
    +\sup_{z\in[Z_1(\varepsilon),+\infty)}\|v_{II,5}^+(z)\|\int_{Z_1(\varepsilon)}^{+\infty}\|Q_{II,3}(s)\|ds
    \right)
    \\
    \le\exp\left(2\int_{Z_0}^{+\infty}\|Q_{II,3}(s)\|ds\right)
    \left(
    \exp\left(\int_{Z_0}^{Z_1(\varepsilon)}\|R_{II,3}(s,\varepsilon)\|ds\right)
    \right.
    \\
    \left.
    \times
    \int_{Z_0}^{Z_1(\varepsilon)}\|R_{II,3}(s,\varepsilon)\|ds
    +\int_{Z_1(\varepsilon)}^{+\infty}\|Q_{II,3}(s)\|ds
    \right)\rightarrow0
\end{multline}
as $\varepsilon\rightarrow0^+$ due to the estimate \eqref{II estimate Q II3} for $Q_{II,3}$ and the integral estimate \eqref{II estimate R II3} for $R_{II,3}$. Using again \eqref{II estimate Q II3} with \eqref{II estimate v II5} to estimate the integral in the equation \eqref{system v II5} we conclude that $v_{II,5}^+(z)\rightarrow e_+$ as $z\rightarrow+\infty$, and $\widetilde v_{II,3}^{\,+}(z)$ has the same exponentially vanishing asymptotics as $v_{II,3}^+(z)$. Therefore $\widetilde v_{II,3}^{\,+}(z)=v_{II,3}^+(z)$ and
\begin{equation}\label{v II2 via vII5}
    v_{II,2}^{+}(z):=(I+\widehat T_{II}(z,0))\frac{\exp\left(-\frac{2c_0}3z^{\frac32}\right)}{z^{\frac14}}v_{II,5}^{+}(z).
\end{equation}

Now, using the formulae \eqref{u II5} and \eqref{u II3} which establish the relation between $u_{II,2},u_{II,3}$ and $u_{II,5}$  we define:
\begin{equation*}
    \widetilde u_{II,2}^+(z,\varepsilon)=(I+\widehat T_{II}(z,0))\frac{\exp\left(-\frac{2c_0}{3}z^{\frac32}\right)}{z^{\frac14}}u_{II,5}^+(z,\varepsilon),
\end{equation*}
and this is a solution of the system \eqref{system u II2}. From the convergence in \eqref{II eq 2} and  the equality \eqref{v II2 via vII5} we conclude that
\begin{equation*}
    \widetilde u_{II,2}^+(z,\varepsilon)=v_{II,2}^+(z)+o(1)
\end{equation*}
as $\varepsilon\rightarrow0^+$ for every fixed $z\ge Z_0$. For $z=Z_1(\varepsilon)$ we use the estimate for $\widehat T_{II}(z,0)$ by Lemma \ref{lem II estimate T hat} and the fact that $u_{II,5}^+(Z_1(\varepsilon),\varepsilon)=e_+$ which follows from \eqref{system u II5}. Thus we have
\begin{equation*}
    \widetilde u_{II,2}^+(Z_1(\varepsilon),\varepsilon)
    =\frac{\exp\left(-\frac{2c_0}{3}(Z_1(\varepsilon))^{\frac32}\right)}{(Z_1(\varepsilon))^{\frac14}}(e_++o(1))
\end{equation*}
as $\varepsilon\rightarrow0^+$. This proves the part of Lemma \ref{lem II result 1st part} concerning the ``small'' solution.
\\
\emph{``Large'' solution.}  Take
\begin{equation}\label{v+}
    v^-:=Z_0^{\frac14}
    \left(
      \begin{array}{cc}
        \exp\left(\frac{2c_0}{3}(Z_0)^{\frac32}\right) & 0 \\
        0 & \exp\left(-\frac{2c_0}{3}(Z_0)^{\frac32}\right) \\
      \end{array}
    \right)
    v_{II,3}^-(Z_0),
\end{equation}
where $v_{II,3}^-$ is defined in \eqref{v II3 pm}. Consider the following solution of the system \eqref{system u II4}:
\begin{multline*}
    u_{II,4}^-(z,\varepsilon)=v^-+\int_{Z_0}^z
    \left(
      \begin{array}{cc}
        \exp\left(\frac{2c_0}{3}s^{\frac32}\right) & 0 \\
        0 & \exp\left(-\frac{2c_0}{3}s^{\frac32}\right) \\
      \end{array}
    \right)
    \\
    \times
    (Q_{II,3}(s)+R_{II,3}(s,\varepsilon))
    \left(
      \begin{array}{cc}
        \exp\left(-\frac{2c_0}{3}s^{\frac32}\right) & 0 \\
        0 & \exp\left(\frac{2c_0}{3}s^{\frac32}\right) \\
      \end{array}
    \right)
    u_{II,4}^-(s,\varepsilon)ds.
\end{multline*}
Returning to $u_{II,3}$ by the relation \eqref{system u II4} we have
\begin{multline*}
    u_{II,3}^-(z,\varepsilon)=\frac1{z^{\frac14}}
    \left(
      \begin{array}{cc}
        \exp\left(-\frac{2c_0}{3}z^{\frac32}\right) & 0 \\
        0 & \exp\left(\frac{2c_0}{3}z^{\frac32}\right) \\
      \end{array}
    \right)
    v^-
    \\
    +\int_{Z_0}^z\left(\frac sz\right)^{\frac14}
    \left(
      \begin{array}{cc}
        \exp\left(-\frac{2c_0}{3}(z^{\frac32}-s^{\frac32})\right) & 0 \\
        0 & \exp\left(\frac{2c_0}{3}(z^{\frac32}-s^{\frac32})\right) \\
      \end{array}
    \right)
    \\
    \times
    (Q_{II,3}(s)+R_{II,3}(s,\varepsilon))
    u_{II,3}^-(s,\varepsilon)ds.
\end{multline*}
Normalising the solution
\begin{equation}\label{u II6}
    u_{II,3}^-(z,\varepsilon)=\frac{\exp\left(\frac{2c_0}{3}z^{\frac32}\right)}{z^{\frac14}}u_{II,6}^-(z,\varepsilon)
\end{equation}
we come to the integral equation
\begin{multline}\label{system u II6}
    u_{II,6}^-(z,\varepsilon)=
    \left(
      \begin{array}{cc}
        \exp\left(-\frac{4c_0}{3}z^{\frac32}\right) & 0 \\
        0 &  1\\
      \end{array}
    \right)
    v^-
    \\
    +\int_{Z_0}^z
    \left(
      \begin{array}{cc}
        \exp\left(-\frac{4c_0}{3}(z^{\frac32}-s^{\frac32})\right) & 0 \\
        0 &  1\\
      \end{array}
    \right)
    (Q_{II,3}(s)+R_{II,3}(s,\varepsilon))
    u_{II,6}^-(s,\varepsilon)ds.
\end{multline}
Doing the same transformations with the free system \eqref{system v II3} we get the equation
\begin{multline}\label{system v II6}
    v_{II,6}^-(z)=
    \left(
      \begin{array}{cc}
        \exp\left(-\frac{4c_0}{3}z^{\frac32}\right) & 0 \\
        0 &  1\\
      \end{array}
    \right)
    v^-
    \\
    +\int_{Z_0}^z
    \left(
      \begin{array}{cc}
        \exp\left(-\frac{4c_0}{3}(z^{\frac32}-s^{\frac32})\right) & 0 \\
        0 &  1\\
      \end{array}
    \right)
    Q_{II,3}(s)v_{II,6}^-(s)ds.
\end{multline}
and define $\widetilde v_{II,3}^{\,-}(z):=\frac{\exp\left(\frac{2c_0}{3}z^{\frac32}\right)}{z^{\frac14}}v_{II,6}^-(z)$. Since
\begin{equation*}
    \widetilde v_{II,3}^{\,-}(Z_0)=\frac1{(Z_0)^{\frac14}}
    \left(
      \begin{array}{cc}
        \exp\left(-\frac{2c_0}{3}(Z_0)^{\frac32}\right) & 0 \\
        0 &  \exp\left(\frac{2c_0}{3}(Z_0)^{\frac32}\right)\\
      \end{array}
    \right)
    v^-=v_{II,3}^{\,-}(Z_0)
\end{equation*}
and both $\widetilde v_{II,3}^{\,-}$ and $v_{II,3}^-$ satisfy the same system \eqref{system v II3}, they coincide. From the asymptotics \eqref{v II3 pm} of $v_{II,3}^-(z)$ we conclude that $v_{II,6}^-(z)\rightarrow e_-$ as $z\rightarrow+\infty$. Subtracting \eqref{system v II6} from \eqref{system u II6} we get the equality
\begin{multline}\label{II eq 3}
    u_{II,6}^-(z,\varepsilon)-v_{II,6}^-(z)=
    \int_{Z_0}^z
    \left(
      \begin{array}{cc}
        \exp\left(-\frac{4c_0}{3}(z^{\frac32}-s^{\frac32})\right) & 0 \\
        0 &  1\\
      \end{array}
    \right)
    R_{II,3}(s,\varepsilon)u_{II,6}^-(s,\varepsilon)ds
    \\
    +\int_{Z_0}^z
    \left(
      \begin{array}{cc}
        \exp\left(-\frac{4c_0}{3}(z^{\frac32}-s^{\frac32})\right) & 0 \\
        0 &  1\\
      \end{array}
    \right)
    Q_{II,3}(s)(u_{II,6}^-(s,\varepsilon)-v_{II,6}^-(s))ds.
\end{multline}
Applying Lemma \ref{lem II tech 1} to the equation \eqref{system u II6} and the equality \eqref{II eq 3} we have:
\begin{equation*}
    \sup_{z\in[Z_0,Z_1(\varepsilon)]}\|u_{II,6}^-(z,\varepsilon)\|\le\|v^-\|\exp\left(\int_{Z_0}^{Z_1(\varepsilon)}\|Q_{II,3}(s)+R_{II,3}(s,\varepsilon)\|ds\right),
\end{equation*}
and
\begin{multline}\label{II eq 5}
    \sup_{z\in[Z_0,Z_1(\varepsilon)]}\|u_{II,6}^-(z,\varepsilon)-v_{II,6}^-(z)\|
    \le\exp\left(\int_{Z_0}^{Z_1(\varepsilon)}\|Q_{II,3}(s)\|ds\right)
    \\
    \times
    \sup_{z\in[Z_0,Z_1(\varepsilon)]}\|u_{II,6}^-(z,\varepsilon)\|\int_{Z_0}^{Z_1(\varepsilon)}\|R_{II,3}(s,\varepsilon)\|ds
    \le\|v^-\|\int_{Z_0}^{Z_1(\varepsilon)}\|R_{II,3}(s,\varepsilon)\|ds
    \\
    \times\exp\left(2\int_{Z_0}^{+\infty}\|Q_{II,3}(s)\|ds+\int_{Z_0}^{Z_1(\varepsilon)}\|R_{II,3}(s,\varepsilon)\|ds\right)\rightarrow0
\end{multline}
as $\varepsilon\rightarrow0^+$ due to the estimate \eqref{II estimate Q II3} of $Q_{II,3}$ and the integral estimate \eqref{II estimate R II3} of $R_{II,3}$.
Define $\widetilde u_{II,2}^-$ using the formulae \eqref{u II3} and \eqref{u II6} which relate $u_{II,2},u_{II,3}$ and $u_{II,6}$ as
\begin{equation*}
    \widetilde u_{II,2}^-(z,\varepsilon):=(I+\widehat T_{II}(z,0))\frac{\exp\left(\frac{2c_0}{3}z^{\frac32}\right)}{z^{\frac14}}u_{II,6}^-(z,\varepsilon),
\end{equation*}
and this is a solution of the system \eqref{system u II2}. For $z=Z_1(\varepsilon)$ from the convergence \eqref{II eq 5}, asymptotics $v_{II,6}^-(z)\rightarrow e_-$ as $z\rightarrow+\infty$ and the estimate for $\widehat T_{II}(z,0)$ by Lemma \ref{lem II estimate T hat} we have:
\begin{equation*}
    \widetilde u_{II,2}^-(Z_1(\varepsilon),\varepsilon)
    =\frac{\exp\left(\frac{2c_0}{3}(Z_1(\varepsilon))^{\frac32}\right)}{(Z_1(\varepsilon))^{\frac14}}(e_-+o(1))
\end{equation*}
as $\varepsilon\rightarrow0^+$. For every fixed $z\ge Z_0$ convergence in \eqref{II eq 5} implies that
\begin{equation*}
    \widetilde u_{II,2}^-(z,\varepsilon)
    \rightarrow
    (I+\widehat T_{II}(z,0))\frac{\exp\left(\frac{2c_0}{3}z^{\frac32}\right)}{z^{\frac14}}v_{II,6}^-(z)=v_{II,2}^-(z)
\end{equation*}
as $\varepsilon\rightarrow0^+$. This completes the proof of Lemma \ref{lem II result 1st part}.
\end{proof}

The result for the interval $[Z_1(\varepsilon),Z_2(\varepsilon)]$ is given by the following lemma.

\begin{lem}\label{lem II result 2nd part}
Let the conditions of Lemma \ref{lem II result} hold, $Z_1(\varepsilon)$ be given by \eqref{Z 1} and $Z_2(\varepsilon)$ by \eqref{Z 2} with $t_{I-II}$ given by \eqref{t I-II}. If
\begin{equation*}
   \int_{Z_1(\varepsilon)}^{Z_2(\varepsilon)}\frac{\|R_{II}(t(s,\varepsilon),\varepsilon)\|}{\sqrt{s}}ds=o\left(\varepsilon^{\frac23}\right)
   \text{ as }\varepsilon\rightarrow0^+,
\end{equation*}
then for every sufficiently small $\varepsilon>0$ the system \eqref{system u II2} on the interval $[Z_1(\varepsilon),Z_2(\varepsilon)]$ has two solutions $\widehat u_{II,2}^{\pm}(z,\varepsilon)$ such that, as $\varepsilon\rightarrow0^+$,
\begin{equation}\label{II answer 2nd part}
    \widehat u_{II,2}^{\pm}(z,\varepsilon)=\exp
    \left(
    \int_{Z_1(\varepsilon)}^z
    \left(
    \mp\frac{c_0\sqrt s}{(1-\varepsilon^{\frac23}s)^{\varkappa}}-\frac1{4s(1\pm\varepsilon^{\frac13}\sqrt s)}
    \right)ds
    \right)
    (e_{\pm}+o(1)),
\end{equation}
where the vectors $e_{\pm}$ are given by \eqref{e +-} and the remainder $o(1)$ converges uniformly with respect to $z\in[Z_1(\varepsilon),Z_2(\varepsilon)]$.
\end{lem}

\begin{proof}
First let us eliminate the off-diagonal entries of the matrix $V_{II}(z,\varepsilon)$ with the Harris--Lutz transformation
\begin{equation}\label{u II7}
    u_{II,2}(z)=(I+\widehat T_{II}(z,\varepsilon))u_{II,7}(z)
\end{equation}
where $\widehat T_{II}$ is given by \eqref{T II hat}. On the one hand, in contrast with the interval $[Z_0,Z_1(\varepsilon)]$, the transformation depends on $\varepsilon$, because one cannot ignore the difference between $V_{II}(z,\varepsilon)$ and $V_{II}(z,0)$ anymore. On the other hand, we do not need to prove convergence to solutions of some other system independent of $\varepsilon$ (it is no longer true that there is convergence to solutions of \eqref{system v II2}). With the substitution \eqref{u II7} we come to the system
\begin{equation}\label{system u II7}
    u_{II,7}'(z)=(\Lambda_{II,7}(z,\varepsilon)+R_{II,7}(z,\varepsilon))u_{II,7}(z),
\end{equation}
where, according to \eqref{Lambda II2}, \eqref{S II2} and \eqref{Harris--Lutz R 1},
\begin{equation}\label{Lambda II7}
    \Lambda_{II,7}(z,\varepsilon):=\Lambda_{II,2}(z,\varepsilon)+\text{diag}\, S_{II,2}(z,\varepsilon)
    =
\left(
  \begin{array}{cc}
    \lambda^+_{II,7}(z,\varepsilon) & 0 \\
    0 & \lambda^-_{II,7}(z,\varepsilon) \\
  \end{array}
\right),
\end{equation}
\begin{equation}\label{lambda II7}
    \lambda^{\pm}_{II,7}(z,\varepsilon):=\mp\frac{c_0\sqrt{z}}{(1-\varepsilon^{\frac23}z)^{\varkappa}}-\frac1{4z(1\pm\varepsilon^{\frac13}\sqrt{z})},
\end{equation}
\begin{multline}\label{R II7}
    R_{II,7}(z,\varepsilon):=(I+\widehat T_{II}(z,\varepsilon))^{-1}R_{II,2}(z,\varepsilon)(I+\widehat T_{II}(z,\varepsilon))
    \\
    +(I+\widehat T_{II}(z,\varepsilon))^{-1}(S_{II,2}(z,\varepsilon)\widehat T_{II}(z,\varepsilon)-\widehat T_{II}(z,\varepsilon)\, \text{diag}\, S_{II,2}(z,\varepsilon)).
\end{multline}
Making variation of parameters
\begin{equation}\label{u II8}
    u_{II,7}(z)=\exp\left(\int_{Z_1(\varepsilon)}^{z}\Lambda_{II,7}(\sigma,\varepsilon)d\sigma\right)u_{II,8}(z),
\end{equation}
and substituting to the system \eqref{system u II7} we have:
\begin{multline}\label{system u II8}
    u_{II,8}'(z)=\exp\left(-\int_{Z_1(\varepsilon)}^{z}\Lambda_{II,7}(\sigma,\varepsilon)d\sigma\right)
    \\
    \times
    R_{II,7}(z,\varepsilon)
    \exp\left(\int_{Z_1(\varepsilon)}^{z}\Lambda_{II,7}(\sigma,\varepsilon)d\sigma\right)u_{II,8}(z).
\end{multline}
Let us now introduce two solutions $u_{II,8}^{\pm}$ of this system which satisfy the following equations:
\begin{multline*}
    u_{II,8}^+(z,\varepsilon)=e_+-\int_z^{Z_2(\varepsilon)}\exp\left(-\int_{Z_1(\varepsilon)}^{s}\Lambda_{II,7}(\sigma,\varepsilon)d\sigma\right)
    \\
    \times
    R_{II,7}(s,\varepsilon)
    \exp\left(\int_{Z_1(\varepsilon)}^{s}\Lambda_{II,7}(\sigma,\varepsilon)d\sigma\right)u_{II,8}^+(s,\varepsilon)ds
\end{multline*}
and
\begin{multline*}
    u_{II,8}^-(z,\varepsilon)=e_-+\int_{Z_1(\varepsilon)}^{z}\exp\left(-\int_{Z_1(\varepsilon)}^{s}\Lambda_{II,7}(\sigma,\varepsilon)d\sigma\right)
    \\
    \times
    R_{II,7}(s,\varepsilon)
    \exp\left(\int_{Z_1(\varepsilon)}^{s}\Lambda_{II,7}(\sigma,\varepsilon)d\sigma\right)u_{II,8}^-(s,\varepsilon)ds.
\end{multline*}
For solutions $u_{II,7}^{\pm}(z,\varepsilon):=\exp\left(\int_{Z_1(\varepsilon)}^{z}\Lambda_{II,7}(\sigma,\varepsilon)d\sigma\right)u_{II,8}^{\pm}(z,\varepsilon)$ these equations read as follows:
\begin{multline*}
    u_{II,7}^+(z,\varepsilon)=\exp\left(\int_{Z_1(\varepsilon)}^{z}\lambda^+_{II,7}(\sigma,\varepsilon)d\sigma\right)e_+
    \\
    -\int_z^{Z_2(\varepsilon)}\exp\left(-\int_{z}^{s}\Lambda_{II,7}(\sigma,\varepsilon)d\sigma\right)
    R_{II,7}(s,\varepsilon)u_{II,7}^+(s,\varepsilon)ds
\end{multline*}
and
\begin{multline*}
    u_{II,7}^-(z,\varepsilon)=\exp\left(\int_{Z_1(\varepsilon)}^{z}\lambda^-_{II,7}(\sigma,\varepsilon)d\sigma\right)e_-
    \\
    +\int_{Z_1(\varepsilon)}^{z}\exp\left(\int_{s}^{z}\Lambda_{II,7}(\sigma,\varepsilon)d\sigma\right)
    R_{II,7}(s,\varepsilon)u_{II,7}^-(s,\varepsilon)ds.
\end{multline*}
Normalising these solutions by the substitution
\begin{equation}\label{u II9}
    u_{II,7}^{\pm}(z,\varepsilon)=\exp\left(\int_{Z_1(\varepsilon)}^z\lambda^{\pm}_{II,7}(\sigma,\varepsilon)d\sigma\right)u_{II,9}^{\pm}(z,\varepsilon)
\end{equation}
we come to the following equations:
\begin{multline*}
    u_{II,9}^+(z,\varepsilon)=e_+-\int_z^{Z_2(\varepsilon)}
    \left(
      \begin{array}{cc}
        1 & 0 \\
        0 & \exp\left(\int_z^s(\lambda^+_{II,7}(\sigma,\varepsilon)-\lambda^-_{II,7}(\sigma,\varepsilon))d\sigma\right) \\
      \end{array}
    \right)
    \\
    \times
    R_{II,7}(s,\varepsilon)
    u_{II,9}^+(s,\varepsilon)ds
\end{multline*}
and
\begin{multline*}
    u_{II,9}^-(z,\varepsilon)=e_-+\int_{Z_1(\varepsilon)}^{z}
    \left(
      \begin{array}{cc}
        \exp\left(\int_s^z(\lambda^+_{II,7}(\sigma,\varepsilon)-\lambda^-_{II,7}(\sigma,\varepsilon))d\sigma\right) & 0 \\
        0 & 1 \\
      \end{array}
    \right)
    \\
    \times
    R_{II,7}(s,\varepsilon)
    u_{II,9}^-(s,\varepsilon)ds.
\end{multline*}
These can be rewritten as
\begin{multline}\label{II eq 8}
    u_{II,9}^+(z,\varepsilon)-e_+
    \\
    =
    -\int_z^{Z_2(\varepsilon)}
    \left(
      \begin{array}{cc}
        1 & 0 \\
        0 & \exp\left(\int_z^s(\lambda^+_{II,7}(\sigma,\varepsilon)-\lambda^-_{II,7}(\sigma,\varepsilon))d\sigma\right) \\
      \end{array}
    \right)
    R_{II,7}(s,\varepsilon)
    e_+ds
    \\
    -\int_z^{Z_2(\varepsilon)}
    \left(
      \begin{array}{cc}
        1 & 0 \\
        0 & \exp\left(\int_z^s(\lambda^+_{II,7}(\sigma,\varepsilon)-\lambda^-_{II,7}(\sigma,\varepsilon))d\sigma\right) \\
      \end{array}
    \right)
    R_{II,7}(s,\varepsilon)
    \\
    \times
    (u_{II,9}^+(s,\varepsilon)-e_+)ds
\end{multline}
and
\begin{multline}\label{II eq 7}
    u_{II,9}^-(z,\varepsilon)-e_-
    \\
    =
    \int_{Z_1(\varepsilon)}^{z}
    \left(
      \begin{array}{cc}
        \exp\left(\int_s^z(\lambda^+_{II,7}(\sigma,\varepsilon)-\lambda^-_{II,7}(\sigma,\varepsilon))d\sigma\right) & 0 \\
        0 & 1 \\
      \end{array}
    \right)
    R_{II,7}(s,\varepsilon)
    e_-ds
    \\
    +
    \int_{Z_1(\varepsilon)}^{z}
    \left(
      \begin{array}{cc}
        \exp\left(\int_s^z(\lambda^+_{II,7}(\sigma,\varepsilon)-\lambda^-_{II,7}(\sigma,\varepsilon))d\sigma\right) & 0 \\
        0 & 1 \\
      \end{array}
    \right)
    R_{II,7}(s,\varepsilon)
    \\
    \times
    (u_{II,9}^-(s,\varepsilon)-e_-)ds.
\end{multline}
From the expression \eqref{S II2} for $S_{II,2}$, the estimate for $T_{II}$ by Lemmas \ref{lem II estimate T hat} and \ref{lem II remainder} we get:
\begin{equation}\label{II estimate R II7}
    \int_{Z_1(\varepsilon)}^{Z_2(\varepsilon)}\|R_{II,7}(s,\varepsilon)\|ds\rightarrow0\text{ as }\varepsilon\rightarrow0^+.
\end{equation}
We also have
\begin{equation*}
    \lambda^+_{II,7}(z,\varepsilon)-\lambda^-_{II,7}(z,\varepsilon)=-\frac{2\sqrt z}{1-\varepsilon^{\frac23}z}\left(c_0(1-\varepsilon^{\frac23}z)^{1-\varkappa}-\frac{\varepsilon^{\frac13}}{4z}\right).
\end{equation*}
Since $1-\varkappa=\frac{1-\gamma}{2\gamma}$ and the values  $1-\varepsilon^{\frac23}z$ and $z$ are separated from zero for $z\in[Z_1(\varepsilon),Z_2(\varepsilon)]$, the above expression is strictly negative for all $z$ from the interval considered, if $\varepsilon$ is small enough. Hence Lemma \ref{lem II tech 1} yields for both \eqref{II eq 7} and \eqref{II eq 8}:
\begin{multline}\label{II eq 9}
    \sup_{z\in[Z_1(\varepsilon),Z_2(\varepsilon)]}\left\|u_{II,9}^{\pm}(z,\varepsilon)-e_{\pm}\right\|
    \\
    \le
    \exp\left(\int_{Z_1(\varepsilon)}^{Z_2(\varepsilon)}\|R_{II,7}(s,\varepsilon)\|ds\right)
    \int_{Z_1(\varepsilon)}^{Z_2(\varepsilon)}\|R_{II,7}(s,\varepsilon)\|ds\rightarrow0
\end{multline}
as $\varepsilon\rightarrow0^+$. Now using the relations \eqref{u II7} and \eqref{u II9} we define functions
\begin{equation*}
    \widehat u_{II,2}^{\pm}(z,\varepsilon):=\exp\left(\int_{Z_1(\varepsilon)}^z\lambda^{\pm}_{II,7}(\sigma,\varepsilon)d\sigma\right)(I+\widehat T_{II}(z,\varepsilon))u_{II,9}^{\pm}(z,\varepsilon)
\end{equation*}
and they are solutions of the system \eqref{u II2}. Lemma \ref{lem II estimate T hat}, expressions \eqref{lambda II7} and convergence in \eqref{II eq 9} imply that these solutions have asymptotics \eqref{II answer 2nd part}. This completes the proof.
\end{proof}

The following trivial lemma helps to match the results in different intervals.

\begin{lem}\label{lem II tech 2}
Let $g(\varepsilon),f_+(\varepsilon),f_-(\varepsilon)$ be functions of $\varepsilon$ defined in some neighbourhood of the point $\varepsilon=0$ with values in $\mathbb C^2$ and such that $g(\varepsilon)\rightarrow g,f_{\pm}(\varepsilon)\rightarrow f_+$ as $\varepsilon\rightarrow0$, where the vectors $f_+$ and $f_-$ are linearly independent and $g=c_+f_++c_-f_-$. Then in the decomposition $g(\varepsilon)=c_+(\varepsilon)f_+(\varepsilon)+c_-(\varepsilon)f_-(\varepsilon)$ the coefficients converge: $c_{\pm}(\varepsilon)\rightarrow c_{\pm}$ as $\varepsilon\rightarrow0$.
\end{lem}

\begin{proof}
Consider scalar products with the vectors orthogonal to $f_{\pm}$ to immediately see the result.
\end{proof}

Combining the results of Lemmas \ref{lem II result 1st part} and \ref{lem II result 2nd part} we can now prove Lemma \ref{lem II result}.

\begin{proof}[Proof of Lemma \ref{lem II result}]
Let us first rewrite the formula \eqref{II answer 1st part} for the asymptotics from Lemma \ref{lem II result 1st part} using that
\begin{equation*}
    \frac{\exp\left(\mp\frac{2c_0}3(Z_1(\varepsilon))^{\frac32}\right)}{(Z_1(\varepsilon))^{\frac14}}
    =a^{\pm}_{II}\exp\left(\int_{Z_0}^{Z_1(\varepsilon)}
    \lambda_{II,7}^{\pm}(s,\varepsilon)ds\right)(1+o(1))
\end{equation*}
as $\varepsilon\rightarrow0^+$, where $a^{\pm}_{II}:=\frac{\exp\left(\mp\frac{2c_0}3Z_0^{\frac32}\right)}{Z_0^{\frac14}}$, which we can do because
\begin{equation*}
    \left|\lambda_{II,7}^{\pm}(s,\varepsilon)-
    \left(
    \mp c_0\sqrt s
    -\frac1{4s}
    \right)
    \right|
    <
    c_{24}
    \left(
    \varepsilon^{\frac23}s^{\frac32}+\frac{\varepsilon^{\frac13}}{\sqrt s},
    \right)
\end{equation*}
with some $c_{24}>0$ and
\begin{equation*}
    \int_{Z_0}^{Z_1(\varepsilon)}\left(\varepsilon^{\frac23}s^{\frac32}+\frac{\varepsilon^{\frac13}}{\sqrt s},
    \right)ds\rightarrow0\text{ as }\varepsilon\rightarrow0^+
\end{equation*}
due to the choice of $Z_1(\varepsilon)$ in \eqref{Z 1}.
Let us define for $z\in[Z_1(\varepsilon),Z_2(\varepsilon)]$
\begin{equation}\label{u II2-}
    u_{II,2}^-(z,\varepsilon):=a_{II}^-\exp
    \left(
    \int_{Z_0}^{Z_1(\varepsilon)}
    \lambda_{II,7}^-(s,\varepsilon)ds
    \right)
    \widehat u_{II,2}^{\,-}(z,\varepsilon).
\end{equation}
On the interval $[Z_0,Z_1(\varepsilon)]$ the continuation of the solution $\widehat u_{II,2}^{\,-}(z,\varepsilon)$ has a decomposition with some coefficients in terms of the basis of solutions $\widetilde u_{II,2}^{\pm}(z,\varepsilon)$, and at the point $Z_1(\varepsilon)$ one has:
\begin{equation*}
    \begin{array}{rl}
    \widehat u_{II,2}^{\,+}(Z_1(\varepsilon),\varepsilon)=&e_++o(1),
    \\
    \frac{\widetilde u_{II,2}^{\pm}(Z_1(\varepsilon),\varepsilon))}{a_{II}^{\pm}}
    \exp
    \left(
    -\int_{Z_0}^{Z_1(\varepsilon)}\lambda_{II,7}^{\pm}(s,\varepsilon)ds
    \right)
    =&e_{\pm}+o(1).
    \end{array}
\end{equation*}
By Lemma \ref{lem II tech 2} we conclude that for $z\in[Z_0,Z_1(\varepsilon)]$
\begin{multline*}
    \widehat u_{II,2}^{\,+}(z,\varepsilon)=(1+o(1))
    \frac{\widetilde u_{II,2}^{+}(z,\varepsilon))}{a_{II}^+}
    \exp
    \left(
    -\int_{Z_0}^{Z_1(\varepsilon)}\lambda_{II,7}^{+}(s,\varepsilon)ds
    \right)
    \\
    +o
    \left(
    \widetilde u_{II,2}^{-}(z,\varepsilon))
    \exp
    \left(
    -\int_{Z_0}^{Z_1(\varepsilon)}\lambda_{II,7}^-(s,\varepsilon)ds
    \right)
    \right)
\end{multline*}
and therefore by \eqref{u II2-}
\begin{multline*}
    u_{II,2}^{\,+}(z,\varepsilon)=(1+o(1))
    \widetilde u_{II,2}^{+}(z,\varepsilon)
    \\
    +o
    \left(
    \widetilde u_{II,2}^{-}(z,\varepsilon))
    \exp
    \left(
    \int_{Z_0}^{Z_1(\varepsilon)}(\lambda_{II,7}^+(s,\varepsilon)-\lambda_{II,7}^-(s,\varepsilon))ds
    \right)
    \right).
\end{multline*}
For every fixed $z\ge Z_0$ this means that
\begin{equation*}
    u_{II,2}^{\,+}(z,\varepsilon)\rightarrow v_{II,2}^+(z)\text{ as }\varepsilon\rightarrow0^+.
\end{equation*}
Asymptotics of $u_{II,2}^{\,+}(z,\varepsilon)$ at $z=Z_2(\varepsilon)$ is due to Lemma \ref{lem II result 2nd part}.

For the second solution define for $z\in[Z_0,Z_1(\varepsilon)]$
\begin{equation}\label{u II2+}
    u_{II,2}^-(z,\varepsilon):=\widetilde u_{II,2}^-(z,\varepsilon)
\end{equation}
Analogously, we have
\begin{equation*}
    \begin{array}{rl}
    \frac{\widetilde u_{II,2}^{-}(Z_1(\varepsilon),\varepsilon))}{a_{II}^{-}}
    \exp
    \left(
    -\int_{Z_0}^{Z_1(\varepsilon)}\lambda_{II,7}^{-}(s,\varepsilon)ds
    \right)
    =&e_-+o(1),
    \\
    \widehat u_{II,2}^{\pm}(Z_1(\varepsilon),\varepsilon)=&e_{\pm}+o(1)
    \end{array}
\end{equation*}
as $\varepsilon\rightarrow0^+$. Therefore for the continuation of $u_{II,2}^-=\widetilde u_{II,2}^-$ to the interval $[Z_1(\varepsilon),Z_2(\varepsilon)]$ one has
\begin{multline*}
    u_{II,2}^-(z,\varepsilon)=\widetilde u_{II,2}^-(z,\varepsilon)
    \\
    =a_{II}^{-}\exp
    \left(
    \int_{Z_0}^{Z_1(\varepsilon)}\lambda_{II,7}^{-}(s,\varepsilon)ds
    \right)
    ((1+o(1))\widehat u_{II,2}^{\,-}(z,\varepsilon)+o(\widehat u_{II,2}^{\,+}(z,\varepsilon))).
\end{multline*}
For $z=Z_2(\varepsilon)$ this means that
\begin{multline*}
    u_{II,2}^-(Z_2(\varepsilon),\varepsilon)=a_{II}^{-}
    \exp
    \left(
    \int_{Z_0}^{Z_2(\varepsilon)}\lambda_{II,7}^{-}(s,\varepsilon)ds
    \right)
    \\
    \times
    \left(
    e_-+o(1)+o
    \left(
    \exp
    \left(
    \int_{Z_1(\varepsilon)}^{Z_2(\varepsilon)}(\lambda_{II,7}^{+}(s,\varepsilon)-\lambda_{II,7}^{-}(s,\varepsilon))ds
    \right)
    \right)
    \right)
    \\
    =a_{II}^{-}
    \exp
    \left(
    \int_{Z_0}^{Z_2(\varepsilon)}\lambda_{II,7}^{-}(s,\varepsilon)ds
    \right)
    (e_-+o(1))
\end{multline*}
as $\varepsilon\rightarrow0^+$, which is the desired asymptotics \eqref{II answer}. Convergence
\begin{equation*}
    u_{II,2}^{\,-}(z,\varepsilon)\rightarrow v_{II,2}^{\,-}(z)\text{ as }\varepsilon\rightarrow0^+
\end{equation*}
for every fixed $z\ge Z_0$ is due to \eqref{u II2+} and Lemma \ref{lem II result 1st part}.
\end{proof}

\section{Intermediate region $IV$: elliptic case}\label{section IV}
Once again we start with the system \eqref{system u 2} written in the form
\begin{equation*}
    \varepsilon u_{IV}'(t)=
    \left(
    \left(
    \begin{array}{cc}
    \frac{\beta}{t^{\gamma}} & -\frac12
    \\
    \frac12 & -\frac{\beta}{t^{\gamma}}
    \end{array}
    \right)
    +R_{IV}(t,\varepsilon)
    \right)
    u_{IV}(t)
\end{equation*}
We diagonalise the main term of the coefficient matrix with the transformation
\begin{equation}\label{u IV1}
    u_{IV}(t)=T_{IV}(t)u_{IV,1}(t),
\end{equation}
where
\begin{equation}\label{T IV}
    T_{IV}(t):=
    \left(
    \begin{array}{cc}
    1
    &
    1
    \\
    \frac{t^{\gamma}}{2\beta\left(1-i\sqrt{\frac{t^{2\gamma}}{4\beta^2}-1}\right)}
    &
    \frac{t^{\gamma}}{2\beta\left(1+i\sqrt{\frac{t^{2\gamma}}{4\beta^2}-1}\right)}
   \end{array}
    \right).
\end{equation}
Note that this matrix coincides with $T_{V}(t)$ given by \eqref{T V} and with $T_{II}(t)$ given by \eqref{T II} with one of the possible choices of the branch of the square root. The substitution gives:
\begin{equation}\label{system u IV1}
    u_{IV,1}'(t)=\left(\frac{\lambda_{IV}(t)}{\varepsilon}
    \left(
    \begin{array}{cc}
    1 & 0 \\
    0 & -1 \\
    \end{array}
    \right)
    +S_{IV}(t)+R_{IV,1}(t,\varepsilon)\right)u_{IV,1}(t),
\end{equation}
where
\begin{equation}\label{lambda IV}
    \lambda_{IV}(t):=-i\frac{\beta}{t^{\gamma}}\sqrt{\frac{t^{2\gamma}}{4\beta^2}-1}=\lambda_{V}(t),
\end{equation}
\begin{equation}\label{S IV}
    S_{IV}(t):=\frac{\gamma}{2t\left(\frac{t^{2\gamma}}{4\beta^2}-1\right)}
    \left(
    \begin{array}{cc}
    -1-i\sqrt{\frac{t^{2\gamma}}{4\beta^2}-1}
    &
    1-i\sqrt{\frac{t^{2\gamma}}{4\beta^2}-1}
    \\
    1+i\sqrt{\frac{t^{2\gamma}}{4\beta^2}-1}
    &
    -1+i\sqrt{\frac{t^{2\gamma}}{4\beta^2}-1}
    \end{array}
    \right)=S_{V}(t),
\end{equation}
\begin{equation*}
    R_{IV,1}(t,\varepsilon):=\frac{T_{IV}^{-1}(t)R_{IV}(t,\varepsilon)T_{IV}(t)}{\varepsilon}=R_{V,1}(t,\varepsilon).
\end{equation*}

In the region $IV$ analysis goes along the same lines as in the region $II$. However, there are slight changes and formulae should be written in a different way. To avoid confusion we repeat the argument highlighting differences and omitting details which are the same. Let us take
\begin{equation}\label{u IV2}
    u_{IV,1}(t)=u_{IV,2}(z(t,\varepsilon))
\end{equation}
with
\begin{equation*}
    z(t,\varepsilon)=\frac1{\varepsilon^{\frac23}}\left(1-\frac{t^{2\gamma}}{4\beta^2}\right),
\end{equation*}
as in \eqref{z}, and substitute this into the system \eqref{system u IV1}. This gives the system
\begin{equation}\label{system u IV2}
    u_{IV,2}'(z)=(\Lambda_{IV,2}(z,\varepsilon)+S_{IV,2}(z,\varepsilon)+R_{IV,2}(z,\varepsilon))u_{IV,2}(z)
\end{equation}
with
\begin{equation}\label{Lambda IV2}
    \Lambda_{IV,2}(z,\varepsilon):=\lambda_{IV,2}(z,\varepsilon)
    \left(
      \begin{array}{cc}
        1 & 0 \\
        0 & -1 \\
      \end{array}
    \right),
\end{equation}
\begin{equation}\label{lambda IV2}
    \lambda_{IV,2}(z,\varepsilon):=\frac{\lambda_{IV}(t(z,\varepsilon))t'(z,\varepsilon)}{\varepsilon}=
    \frac{ic_0\sqrt{-z}}{(1-\varepsilon^{\frac23}z)^\varkappa},
\end{equation}
\begin{multline}\label{S IV2}
    S_{IV,2}(z,\varepsilon):=S_{IV}(t(z,\varepsilon))t'(z,\varepsilon)
    \\
    =-\frac1{4z(1-\varepsilon^{\frac23}z)}
    \left(
      \begin{array}{cc}
        1+i\varepsilon^{\frac13}\sqrt{-z} & -1+i\varepsilon^{\frac13}\sqrt{-z} \\
        -1-i\varepsilon^{\frac13}\sqrt{-z} & 1-i\varepsilon^{\frac13}\sqrt{-z} \\
      \end{array}
    \right)
\end{multline}
and
\begin{equation}\label{R IV2}
    R_{IV,2}(z,\varepsilon):=
    -\frac{t_0}{2\gamma\varepsilon^{\frac13}}
    \frac{T_{IV}^{-1}(t(z,\varepsilon))R_{IV}(t(z,\varepsilon),\varepsilon)T_{IV}(t(z,\varepsilon))}
    {(1-\varepsilon^{\frac23}z)^{1-\frac1{2\gamma}}},
\end{equation}
where $c_0$ and $\varkappa$ are given by \eqref{c_0 kappa}. We consider the system \eqref{system u IV2} on the interval $[-Z_2(\varepsilon),-Z_0]$ where $Z_0$ is given by \eqref{Z 0} and $Z_2(\varepsilon)$ by \eqref{Z 2} with $t_{I-II}$ given by \eqref{t I-II}. The point $z=-Z_0$ corresponds to the point $t=t_{III-IV}(\varepsilon)$,
\begin{equation}\label{t III-IV}
    t_{III-IV}(\varepsilon):=(2\beta)^{\frac1{\gamma}}(1+\varepsilon^{\frac23}Z_0)^{\frac1{2\gamma}},
\end{equation}
and the point $z=-Z_2(\varepsilon)$ corresponds to the point $t=t_{IV-V}$,
\begin{equation}\label{t IV-V}
    t_{IV-V}:=(2\beta)^{\frac1{\gamma}}(1+\varepsilon^{\frac23}Z_2(\varepsilon))^{\frac1{2\gamma}}=(8\beta^2-t_{I-II}^{2\gamma})^{\frac1{2\gamma}}.
\end{equation}

Consider the free system.

\begin{lem}\label{lem IV free system}
The system
\begin{equation}\label{system v IV2}
    v_{IV,2}'(z)=
    \left(
    ic_0\sqrt{-z}
    \left(
      \begin{array}{cc}
        1 & 0 \\
        0 & -1 \\
      \end{array}
    \right)
    \right.
    \\
    \left.
    -\frac1{4z}
    \left(
      \begin{array}{cc}
        1 & -1 \\
        -1 & 1 \\
      \end{array}
    \right)
    \right)
    v_{IV,2}(z)
\end{equation}
has for $z\in(-\infty,1]$ solutions $v_{IV,2}^{\pm}(z)$ such that
\begin{equation}\label{IV asymptotics v IV2}
    v_{IV,2}^{\pm}(z)=\frac{\exp\left({\mp}\frac{2ic_0}3(-z)^{\frac32}\right)}{(-z)^{\frac14}}(e_{\pm}+o(1))
\end{equation}
as $z\rightarrow-\infty$.
\end{lem}

As in the previous section, the proof of this lemma will be given later, see Remark \ref{rem proof of Lemma IV free system}. The main result for the region $IV$ is the following.

\begin{lem}\label{lem IV result}
Let $c_0,\varkappa>0$ and $R_{IV,2}(z,\varepsilon)$ be given by \eqref{R IV2} with the use of the expression \eqref{T IV} and the definition \eqref{z}. Let $Z_0$ be given by \eqref{Z 0} and $Z_2(\varepsilon)$ by \eqref{Z 2} with $t_{I-II}$ given by \eqref{t I-II}. If
\begin{equation}\label{estimate R IV}
   \int_{-Z_2(\varepsilon)}^{-Z_0}\frac{\|R_{IV}(t(s,\varepsilon),\varepsilon)\|}{\sqrt{-s}}ds=o\left(\varepsilon^{\frac23}\right)
   \text{ as }\varepsilon\rightarrow0^+,
\end{equation}
then for every sufficiently small $\varepsilon>0$ the system \eqref{system u IV2} on the interval $[-Z_2(\varepsilon),-Z_0]$ has two solutions $u_{IV,2}^{\pm}(z,\varepsilon)$ such that, as $\varepsilon\rightarrow0^+$,
\begin{multline}\label{IV answer}
    u_{IV,2}^{\pm}(-Z_2(\varepsilon),\varepsilon)
    \\
    =a_{IV}^{\pm}
    \exp
    \left(
    -\int_{-Z_2(\varepsilon)}^{-Z_0}
    \left(
    \pm\frac{ic_0\sqrt{-s}}{(1-\varepsilon^{\frac23}s)^{\varkappa}}-\frac1{4s(1\mp i\varepsilon^{\frac13}\sqrt{-s})}
    \right)
    ds
    \right)(e_{\pm}+o(1)),
\end{multline}
where $a_{IV}^{\pm}$ are complex constants which are conjugate to each other and the vectors $e_{\pm}$ are given by \eqref{e +-}. Moreover, $u_{IV,2}^{\pm}(z,\varepsilon)\rightarrow v_{IV,2}^{\pm}(z)$ as $\varepsilon\rightarrow0^+$ for every fixed $z\le-Z_0$, where $v_{IV,2}^{\pm}$ are defined in Lemma \ref{lem IV free system}.
\end{lem}

The following lemma shows that the condition \eqref{estimate R IV} is satisfied, if $R_{IV}=R_2^+$.

\begin{lem}\label{lem IV R 2+ estimate}
Let $R_2^+(t,\varepsilon)$ be given by \eqref{R 2 pm} and $t(z,\varepsilon)$ be defined by \eqref{z}.  Let the conditions \eqref{model problem conditions} and \eqref{model problem r} hold. Then for every $z_0>0$ and $z_2(\varepsilon)=-\frac{\nu}{\varepsilon^{\frac23}}$ with $\nu\in(-\infty,0)$ the following estimate holds:
\begin{equation*}
   \int_{-z_2(\varepsilon)}^{-z_0}\frac{\|R_2^+(t(s,\varepsilon),\varepsilon)\|}{\sqrt{-s}}ds
   =O\left(\varepsilon^{\frac23}\varepsilon_0^{\frac{\alpha_r}{\gamma}}\right)
   \text{ as }\varepsilon\rightarrow0^+.
\end{equation*}
\end{lem}

\begin{proof}
Using the same argument as in the proof of Lemma \ref{lem IV R 2+ estimate} we have with some $c_{25}>0$ and $t_{\nu}$ given by \eqref{t nu}:
\begin{equation*}
    \int_{-z_2(\varepsilon)}^{-z_0}\frac{\|R_2^+(t(s,\varepsilon),\varepsilon)\|}{\sqrt{-s}}ds
    <
    c_rc_{25}\varepsilon^{\frac23}\varepsilon_0^{\frac{\alpha_r}{\gamma}}\int_{t_0}^{t_{\nu}}
    \frac{dt}{t^{1+\alpha_r}\sqrt{t-t_0}}.
\end{equation*}
\end{proof}

Now let us rewrite the estimate \eqref{estimate R IV} in terms of $R_{IV,2}$.

\begin{lem}\label{lem IV remainder}
Under the conditions of Lemma \ref{lem IV result}, for every $z_0>0$ and $z_2(\varepsilon)=-\frac{\nu}{\varepsilon^{\frac23}}$ with $\nu\in(-\infty,0)$, if
\begin{equation*}
   \int_{-z_2(\varepsilon)}^{-z_0}\frac{\|R_{IV}(t(s,\varepsilon),\varepsilon)\|}{\sqrt{-s}}ds=o\left(\varepsilon^{\frac23}\right),
   \text{ then }
   \int_{-z_2(\varepsilon)}^{-z_0}\|R_{IV,2}(s,\varepsilon)\|ds\rightarrow0
\end{equation*}
as $\varepsilon\rightarrow0^+$.
\end{lem}

\begin{proof}
Following the proof of Lemma \ref{lem II remainder} we have with some $c_{26}>0$:
\begin{equation*}
    \int_{-z_2(\varepsilon)}^{-z_0}\|R_{IV,2}(s,\varepsilon)\|ds
    <c_{26}\int_{-z_2(\varepsilon)}^{-z_0}\frac{\|R_{IV}(t(s,\varepsilon),\varepsilon)\|}{\varepsilon^{\frac23}\sqrt{-s}}ds\rightarrow0
    \text{ as }\varepsilon\rightarrow0^+.
\end{equation*}
\end{proof}

In the elliptic case the formula for the Harris--Lutz transformation can be simplified. It is not difficult to check that one can take
\begin{equation*}
    \widehat T(x)=
    \left(
      \begin{array}{cc}
        0
        &
        \int_{-\infty}^xS_{12}(x')\exp(\int_{x'}^{x}(\lambda_1-\lambda_2))dx'
        \\
        \int_{-\infty}^xS_{21}(x')\exp(\int_{x'}^{x}(\lambda_2-\lambda_1))dx'
        &
        0
      \end{array}
    \right)
\end{equation*}
instead of \eqref{Harris--Lutz T tilde}. Therefore let us take for $z\in[-Z_2(\varepsilon),-Z_0]$
\begin{equation}\label{T IV hat}
    \widehat T_{IV}(z,\varepsilon)=
    \left(
      \begin{array}{cc}
        0 & \widehat T_{IV_{12}}(z,\varepsilon) \\
        \widehat T_{IV_{21}}(z,\varepsilon) & 0 \\
      \end{array}
    \right)
\end{equation}
with
\begin{equation*}
        \widehat T_{IV_{12}}(z,\varepsilon):=
        \int_{-Z_2(\varepsilon)}^{z}\frac{ds}{4s(1+i\varepsilon^{\frac13}\sqrt{-s})}
        \exp
        \left(\int_{s}^{z}\frac{2ic_0\sqrt{-\sigma}d\sigma}{(1-\varepsilon^{\frac23}\sigma)^{\varkappa}}\right),
\end{equation*}
\begin{equation*}
        \widehat T_{IV_{21}}(z,\varepsilon):=
        \int_{-Z_2(\varepsilon)}^{z}\frac{ds}{4s(1-i\varepsilon^{\frac13}\sqrt{-s})}
        \exp
        \left(-\int_{s}^{z}\frac{2ic_0\sqrt{-\sigma}d\sigma}{(1-\varepsilon^{\frac23}\sigma)^{\varkappa}}\right)=\overline{\widehat T_{IV_{12}}(z,\varepsilon)}.
\end{equation*}

\begin{lem}\label{lem IV estimate T hat}
There exists $c_{IV}>0$ such that for every $\varepsilon\in U\cup\{0\}$ and $z\in[-Z_2(\varepsilon),-1]$, where $Z_2(\varepsilon)$ is given by \eqref{Z 2} with $t_{I-II}$ given by \eqref{t I-II}, the matrix $\widehat T_{IV}(z,\varepsilon)$ given by \eqref{T IV hat} satisfies the following estimate:
\begin{equation}\label{IV estimate of T hat}
    \| \widehat T_{IV}(z,\varepsilon)\|<\frac{c_{IV}}{|z|^{\frac32}}.
\end{equation}
\end{lem}

\begin{proof}
This is the place where the proof for the elliptic case differs from the corresponding proof in the hyperbolic case not only in notation. It is impossible to effectively estimate oscillating exponentials by something which is independent of $\varepsilon$, so one has to pay more attention. Denote $\nu_2:=\varepsilon^{\frac23}Z_2(\varepsilon)$. For $\varepsilon\in U$ we have:
\begin{multline*}
    \|\widehat T_{IV}(z,\varepsilon)\|=|\widehat T_{IV_{12}}(z,\varepsilon)|=|\widehat T_{IV_{21}}(z,\varepsilon)|
    =\left|
    \int_{-Z_2(\varepsilon)}^z
    \frac{\exp
    \left(
    \int_s^z\frac{2ic_0\sqrt{-\sigma}d\sigma}{(1-\varepsilon^{\frac23}\sigma)^{\varkappa}}
    \right)ds}{4s(1+i\varepsilon^{\frac13}\sqrt{-s})}
    \right|
    \\
    =\left|
    \int_{|z|}^{Z_2(\varepsilon)}
    \frac{\exp
    \left(
    \int_{|z|}^s\frac{2ic_0\sqrt{\sigma}d\sigma}{(1+\varepsilon^{\frac23}\sigma)^{\varkappa}}
    \right)ds}{4s(1+i\varepsilon^{\frac13}\sqrt{s})}
    \right|
    =
    \left|
    \int_{\varepsilon^{\frac23}|z|}^{\nu_2}
    \frac{\exp
    \left(
    \int_{\varepsilon^{\frac23}|z|}^{s_1}\frac{2ic_0\sqrt{\sigma_1}d\sigma_1}{\varepsilon(1+\sigma_1)^{\varkappa}}
    \right)ds_1}{4s_1(1+i\sqrt{s_1})}
    \right|
    \\
    =
    \varepsilon
    \left|
    \int_{\varepsilon^{\frac23}|z|}^{\nu_2}
    \frac{(1+s_1)^{\varkappa}d\left(\exp
    \left(
    \int_{\varepsilon^{\frac23}|z|}^{s_1}\frac{2ic_0\sqrt{\sigma_1}d\sigma_1}{\varepsilon(1+\sigma_1)^{\varkappa}}
    \right)\right)}{8c_0s_1^{\frac32}(1+i\sqrt{s_1})}
    \right|
    \\
    =
    \frac{\varepsilon}{8c_0}
    \left|
    \frac{(1+\nu_2)^{\varkappa}}{\nu_2^{\frac32}(1+i\sqrt{\nu_2})}
    \exp
    \left(
    \int_{\varepsilon^{\frac23}|z|}^{\nu_2}\frac{2ic_0\sqrt{\sigma_1}d\sigma_1}{\varepsilon(1+\sigma_1)^{\varkappa}}
    \right)
    -\frac{(1+\varepsilon^{\frac32}|z|)^{\varkappa}}{\varepsilon |z|^{\frac32}(1+i\varepsilon^{\frac13}\sqrt{|z|})}
    \right.
    \\
    \left.
    -\int_{\varepsilon^{\frac23}|z|}^{\nu_2}
    \exp
    \left(
    \int_{\varepsilon^{\frac23}|z|}^{s_1}\frac{2ic_0\sqrt{\sigma_1}d\sigma_1}{\varepsilon(1+\sigma_1)^{\varkappa}}
    \right)
    d\left(
    \frac{(1+s_1)^{\varkappa}}{s_1^{\frac32}(1+i\sqrt{s_1})}
    \right)
    \right|.
\end{multline*}
The first two terms can be estimated as follows:
\begin{equation*}
    \left|
    \frac{(1+\nu_2)^{\varkappa}}{\nu_2^{\frac32}(1+i\sqrt{\nu_2})}
    \exp
    \left(
    \int_{\varepsilon^{\frac23}|z|}^{\nu_2}\frac{2ic_0\sqrt{\sigma_1}d\sigma_1}{\varepsilon(1+\sigma_1)^{\varkappa}}
    \right)
    \right|\le\frac{(1+\nu_2)^{\varkappa}}{\nu_2^{\frac32}}\le\frac{(1+\nu_2)^{\varkappa}}{\varepsilon|z|^{\frac32}},
\end{equation*}
\begin{equation*}
    \left|
    \frac{(1+\varepsilon^{\frac32}|z|)^{\varkappa}}{\varepsilon |z|^{\frac32}(1+i\varepsilon^{\frac13}\sqrt{|z|})}
    \right|
    \le\frac{(1+\nu_2)^{\varkappa}}{\varepsilon|z|^{\frac32}}.
\end{equation*}
Furthermore,
\begin{multline*}
    \left|
    \left(
    \frac{(1+s_1)^{\varkappa}}{s_1^{\frac32}(1+i\sqrt{s_1})}
    \right)'
    \right|
    =
    \left|
    \left(
    \frac{(1+s_1)^{\varkappa}}{s_1^{\frac32}(1+i\sqrt{s_1})}
    \right)
    \left(
    \frac{\varkappa}{1+s_1}-\frac3{2s_1}-\frac i{2\sqrt{s_1}(1+i\sqrt{s_1})}
    \right)
    \right|
    \\
    \le
    \frac{(1+\nu_2)^{\varkappa}}{s_1^{\frac52}}\left(\varkappa+\frac32+\frac{\sqrt{\nu_2}}2\right),
\end{multline*}
and thus
\begin{multline*}
    \left|
    \int_{\varepsilon^{\frac23}|z|}^{\nu_2}
    \exp
    \left(
    \int_{\varepsilon^{\frac23}|z|}^{s_1}\frac{2ic_0\sqrt{\sigma_1}d\sigma_1}{\varepsilon(1+\sigma_1)^{\varkappa}}
    \right)
    d\left(
    \frac{(1+s_1)^{\varkappa}}{s_1^{\frac32}(1+i\sqrt{s_1})}
    \right)
    \right|
    \\
    \le(1+\nu_2)^{\varkappa}\left(\varkappa+\frac32+\frac{\sqrt{\nu_2}}2\right)
    \int_{\varepsilon^{\frac23}|z|}^{\nu_2}
    \frac{ds_1}{s_1^{\frac52}}
    \le\frac{2(1+\nu_2)^{\varkappa}\left(\varkappa+\frac32+\frac{\sqrt{\nu_2}}2\right)}{3\varepsilon|z|^{\frac32}}.
\end{multline*}
Combining everything we get:
\begin{equation*}
    \|\widehat T_{IV}(z,\varepsilon)\|
    \le
    \frac
    {(1+\nu_2)^{\varkappa}}{8c_0|z|^{\frac32}}
    \left(2+\frac23\left(\varkappa+\frac32+\frac{\sqrt{\nu_2}}2\right)\right).
\end{equation*}
The case $\varepsilon=0$ should be considered separately:
\begin{multline*}
    \|\widehat T_{IV}(z,0)\|
    =
    \left|\int_{-\infty}^z\frac{\exp\left(\int_s^z2ic_0\sqrt{-\sigma}d\sigma\right)ds}{4s}\right|
    =
    \left|\int_{|z|}^{+\infty}\frac{\exp\left(\frac{4ic_0}3s^{\frac32}\right)ds}{4s}\right|
    \\
    =
    \left|\frac i{8c_0|z|^{\frac32}}\exp\left(\frac{4ic_0}3|z|^{\frac32}\right)-\frac{3i}{16c_0}\int_{|z|}^{+\infty}
    \frac{\exp\left(\frac{4ic_0}3s^{\frac32}\right)ds}{s^{\frac52}}\right|
    \\
    \le
    \frac1{8c_0|z|^{\frac32}}+\frac3{16c_0}\int_{|z|}^{+\infty}\frac{ds}{s^{\frac52}}=\frac1{4c_0|z|^{\frac32}}.
\end{multline*}
Thus the estimate is proved for both cases.
\end{proof}

On the first part $[-Z_1(\varepsilon),-Z_0]$ of the region $IV$ we treat the system \eqref{system u IV2} as a perturbation of the free system \eqref{system v IV2}.
\begin{equation*}
    u_{IV,2}'(z)=\left(
    ic_0\sqrt{-z}
    \left(
      \begin{array}{cc}
        1 & 0 \\
        0 & -1 \\
      \end{array}
    \right)
    \right.
    \\
    \left.
    -\frac1{4z}
    \left(
      \begin{array}{cc}
        1 & -1 \\
        -1 & 1 \\
      \end{array}
    \right)
    +\widetilde R_{IV,2}(z,\varepsilon)
    \right)u_{IV,2}(z),
\end{equation*}
where
\begin{multline*}
    \widetilde R_{IV,2}(z,\varepsilon)=R_{IV,2}(z,\varepsilon)
    \\
    +ic_0\sqrt{-z}
    \left(
      \begin{array}{cc}
        1 & 0 \\
        0 & -1 \\
      \end{array}
    \right)
    \left(\frac1{(1-\varepsilon^{\frac23}z)^{\varkappa}}-1\right)
    +\frac{i\varepsilon^{\frac13}}{4\sqrt{-z}}
    \left(
      \begin{array}{cc}
        \frac1{1-i\varepsilon^{\frac13}\sqrt{-z}} & -\frac1{1+i\varepsilon^{\frac13}\sqrt{-z}} \\
        -\frac1{1-i\varepsilon^{\frac13}\sqrt{-z}} & \frac1{1+i\varepsilon^{\frac13}\sqrt{-z}} \\
      \end{array}
    \right).
\end{multline*}
From this we see that with some $c_{27}>0$ for every $z\in[-Z_2(\varepsilon),-Z_0]$
\begin{equation*}
    \|\widetilde R_{IV,2}(z,\varepsilon)\|\le\|R_{IV,2}(z,\varepsilon)\|
    +c_{27}\left(\varepsilon^{\frac23}|z|^{\frac32}+\frac{\varepsilon^{\frac13}}{\sqrt{|z|}}\right).
\end{equation*}
In the same way as in the region $II$ due to Lemma \ref{lem IV remainder} and the choice of $Z_1(\varepsilon)=\frac{1}{\varepsilon^{\frac15}}$ one has:
\begin{equation}\label{IV remainder estimate 1st part}
    \int_{-Z_1(\varepsilon)}^{-Z_0}\|\widetilde R_{IV,2}(s,\varepsilon)\|ds\rightarrow0\text{ as }\varepsilon\rightarrow0^+.
\end{equation}
Now we can obtain results for the subintervals $[-Z_1(\varepsilon),-Z_0]$ and $[-Z_2(\varepsilon),-Z_1(\varepsilon)]$.

\begin{lem}\label{lem IV result 1st part}
Let the conditions of Lemma \ref{lem IV result} hold and let $Z_1(\varepsilon)$ be given by \eqref{Z 1}. If
\begin{equation*}
   \int_{-Z_1(\varepsilon)}^{-Z_0}\frac{\|R_{IV}(t(s,\varepsilon),\varepsilon)\|}{\sqrt{-s}}ds=o\left(\varepsilon^{\frac23}\right)
   \text{ as }\varepsilon\rightarrow0^+,
\end{equation*}
then for every sufficiently small $\varepsilon>0$ the system \eqref{system u IV2} on the interval $[-Z_1(\varepsilon),-Z_0]$ has two solutions $\widetilde u_{IV,2}^{\pm}(z,\varepsilon)$ such that, as $\varepsilon\rightarrow0^+$,
\begin{equation}\label{IV answer 1st part}
    \widetilde u_{IV,2}^{\pm}(-Z_1(\varepsilon),\varepsilon)
    =\frac{\exp\left(\mp\frac{2ic_0}3(Z_1(\varepsilon))^{\frac32}\right)}{(Z_1(\varepsilon))^{\frac14}}
    (e_{\pm}+o(1)),
\end{equation}
where the vectors $e_{\pm}$ are given by \eqref{e +-}. Moreover, $\widetilde u_{IV,2}^{\pm}(z,\varepsilon)\rightarrow v_{IV,2}^{\pm}(z)$ as $\varepsilon\rightarrow0^+$ for every fixed $z\le-Z_0$, where $v_{IV,2}^{\pm}$ are defined in Lemma \ref{lem IV free system}.
\end{lem}

\begin{proof}
Take
\begin{equation}\label{u IV3}
    u_{IV,2}(z)=(I+\widehat T_{IV}(z,0))u_{IV,3}(z)
\end{equation}
with $\widehat T_{IV}$ given by \eqref{T IV hat}. According to the argument of the Subsection \ref{subsection Harris-Lutz} and due to the formula \eqref{Harris--Lutz R 1} this leads to the system
\begin{equation}\label{system u IV3}
    u_{IV,3}'(z)=\left(
    ic_0\sqrt{-z}
    \left(
      \begin{array}{cc}
        1 & 0 \\
        0 & -1 \\
      \end{array}
    \right)
    \right.
    \\
    \left.
    -\frac1{4z}I+Q_{IV,3}(z)+R_{IV,3}(z,\varepsilon)
    \right)u_{IV,3}(z)
\end{equation}
with
\begin{equation*}
    Q_{IV,3}(z):=(I+\widehat T_{IV}(z,0))^{-1}(S_{IV,2}(z,0)\widehat T_{IV}(z,0)-\widehat T_{IV}(z,0)\,\text{diag}\,S_{IV,2}(z,0))
\end{equation*}
and
\begin{equation*}
    R_{IV,3}(z,\varepsilon):=(I+\widehat T_{IV}(z,0))^{-1}R_{IV,2}(z,\varepsilon)(I+\widehat T_{IV}(z,0)).
\end{equation*}
From the expression \eqref{S IV2} for $S_{IV,2}$ and the estimate \eqref{IV estimate of T hat} for $T_{IV}$ we have
\begin{equation}\label{IV estimate Q IV3}
    Q_{IV,3}(z)=O\left(\frac1{|z|^{\frac52}}\right)\text{ as }z\rightarrow-\infty.
\end{equation}
Consider the free system
\begin{equation}\label{system v IV3}
    v_{IV,3}'(z)=\left(
    ic_0\sqrt{-z}
    \left(
      \begin{array}{cc}
        1 & 0 \\
        0 & -1 \\
      \end{array}
    \right)
    \right.
    \\
    \left.
    -\frac1{4z}I+Q_{IV,3}(z)
    \right)v_{IV,3}(z).
\end{equation}
Since $\int_{-\infty}^{-1}\|Q_{IV,3}(s)\|ds<\infty$, the asymptotic Levinson theorem \cite[Theorem 8.1]{Coddington-Levinson-1955} is applicable and yields the existence of two solutions $v_{IV,3}^{\pm}$ of the system \eqref{system v IV3} with asymptotics
\begin{equation}\label{v IV3 pm}
    v_{IV,3}^{\pm}(z)=\frac{\exp\left(\mp\frac{2ic_0}3(-z)^{\frac32}\right)}{(-z)^{\frac14}}(e_{\pm}+o(1))\text{ as }z\rightarrow-\infty.
\end{equation}

\begin{rem}\label{rem proof of Lemma IV free system}
Now the proof of Lemma \ref{lem IV free system} follows.
\end{rem}

\begin{proof}[Proof of Lemma \ref{lem IV free system}]
Define solutions $v_{IV,2}^{\pm}$ of the system \eqref{system v IV2} as
\begin{equation}\label{v IV2 pm def}
    v_{IV,2}^{\pm}(z):=(I+\widehat T_{IV}(z,0)) v_{IV,3}^{\pm}(z).
\end{equation}
Due to \eqref{v IV3 pm} and Lemma \ref{lem IV estimate T hat} they have asymptotics \eqref{IV asymptotics v IV2}
\begin{equation*}
    v_{IV,2}^{\pm}(z)=\frac{\exp\left(\mp\frac{2ic_0}3(-z)^{\frac32}\right)}{(-z)^{\frac14}}(e_{\pm}+o(1))\text{ as }z\rightarrow-\infty.
\end{equation*}
\end{proof}

From \eqref{IV estimate of T hat} and the integral estimate of the remainder \eqref{IV remainder estimate 1st part} we also have
\begin{equation}\label{IV estimate R IV3}
    \int_{-Z_1(\varepsilon)}^{-Z_0}\|R_{IV,3}(s,\varepsilon)\|ds\rightarrow0\text{ as }\varepsilon\rightarrow0^+.
\end{equation}
Let us make a variation of parameters transformation: denote
\begin{equation*}
    E_1(z):=\frac{2c_0}3(-z)^{\frac32}\left(
    \begin{array}{cc}
        1 & 0 \\
        0 & -1 \\
      \end{array}
    \right)
\end{equation*}
and take
\begin{equation}\label{u IV4}
    u_{IV,3}(z)=\frac{e^{-iE_1(z)}}{(-z)^{\frac14}}u_{IV,4}(z),
\end{equation}
which turns the system \eqref{system u IV3} into the system
\begin{equation}\label{system u IV4}
    u_{IV,4}'(z)=e^{iE_1(z)}(Q_{IV,3}(z)+R_{IV,3}(z,\varepsilon))e^{-iE_1(z)}u_{IV,4}(z).
\end{equation}
Consider the following particular solutions of the system \eqref{system u IV4}:
\begin{equation}\label{eq u IV4 pm}
    u_{IV,4}^{\pm}(z,\varepsilon)=e_{\pm}+\int_{-Z_1(\varepsilon)}^ze^{iE_1(s)}(Q_{IV,3}(s)+R_{IV,3}(s,\varepsilon))e^{-iE_1(s)}u_{IV,4}^{\pm}(s,\varepsilon)ds.
\end{equation}
Repeating the same manipulations with the system \eqref{system v IV3} we get
\begin{equation}\label{eq v IV4 pm}
    v_{IV,4}^{\pm}(z)=e_{\pm}+\int_{-\infty}^ze^{iE_1(s)}Q_{IV,3}(s)e^{-iE_1(s)}v_{IV,4}^{\pm}(s)ds.
\end{equation}
with
\begin{equation*}
    \widetilde v_{IV,3}^{\,\pm}(z):=\frac{e^{-iE_1(z)}}{(-z)^{\frac14}}v_{IV,4}^{\pm}(z),
\end{equation*}
instead of \eqref{u IV4}. As in the previous section we need to show that $\widetilde v_{IV,3}^{\,\pm}=v_{IV,3}^{\pm}$.
Subtracting \eqref{eq v IV4 pm} from \eqref{eq u IV4 pm} we obtain the equality
\begin{multline}\label{eq u-v IV4}
    u_{IV,4}^{\pm}(z,\varepsilon)-v_{IV,4}^{\pm}(z)=
    \int_{-Z_1(\varepsilon)}^ze^{iE_1(s)}R_{IV,3}(s,\varepsilon)e^{-iE_1(s)}u_{IV,4}^{\pm}(s,\varepsilon)ds
    \\
    -\int_{-\infty}^{-Z_1(\varepsilon)}e^{iE_1(s)}Q_{IV,3}(s)e^{-iE_1(s)}v_{IV,4}^{\pm}(s)ds
    \\
    +\int_{-Z_1(\varepsilon)}^ze^{iE_1(s)}Q_{IV,3}(s)e^{-iE_1(s)}(u_{IV,4}^{\pm}(s,\varepsilon)-v_{IV,4}^{\pm}(s))ds.
\end{multline}
Lemma \ref{lem II tech 1} yields for the equation \eqref{eq u IV4 pm}:
\begin{equation}\label{IV estimate u IV4}
    \sup_{z\in[-Z_1(\varepsilon),-Z_0]}\|u_{IV,4}^{\pm}(z,\varepsilon)\|\le\exp\left(\int_{-Z_1(\varepsilon)}^{-Z_0}\|Q_{IV,3}(s)+R_{IV,3}(s,\varepsilon)\|ds\right),
\end{equation}
for the equation \eqref{eq v IV4 pm}:
\begin{equation}\label{IV estimate v IV4}
    \sup_{z\in(-\infty,-Z_1(\varepsilon)]}\|v_{IV,4}^{\pm}(z)\|\le\exp\left(\int_{-\infty}^{-Z_1(\varepsilon)}\|Q_{IV,3}(s)\|ds\right),
\end{equation}
and finally for the equality \eqref{eq u-v IV4}:
\begin{multline}\label{IV estimate u-v IV4}
    \sup_{z\in[-Z_1(\varepsilon),-Z_0]}\|u_{IV,4}^{\pm}(z,\varepsilon)-v_{IV,4}^{\pm}(z)\|
    \le\exp\left(\int_{-Z_1(\varepsilon)}^{-Z_0}\|Q_{IV,3}(s)\|ds\right)
    \\
    \times
    \left(
    \sup_{z\in[-Z_1(\varepsilon),-Z_0]}\|u_{IV,4}^{\pm}(z,\varepsilon)\|\int_{-Z_1(\varepsilon)}^{-Z_0}\|R_{IV,3}(s,\varepsilon)\|ds
    \right.
    \\
    \left.
    +\sup_{z\in(-\infty,-Z_1(\varepsilon)]}\|v_{IV,4}^{\pm}(z)\|\int_{-\infty}^{-Z_1(\varepsilon)}\|Q_{IV,3}(s)\|ds
    \right)
    \\
    \times
    \le\exp\left(2\int_{-\infty}^{-Z_0}\|Q_{IV,3}(s)\|ds\right)
    \left(
    \exp\left(\int_{-Z_1(\varepsilon)}^{-Z_0}\|R_{IV,3}(s,\varepsilon)\|ds\right)
    \right.
    \\
    \left.
    \times
    \int_{-Z_1(\varepsilon)}^{-Z_0}\|R_{IV,3}(s,\varepsilon)\|ds
    +\int_{-\infty}^{-Z_1(\varepsilon)}\|Q_{IV,3}(s)\|ds
    \right)\rightarrow0
\end{multline}
as $\varepsilon\rightarrow0^+$ due to the estimate \eqref{IV estimate Q IV3} for $Q_{IV,3}$ and the integral estimate \eqref{IV estimate R IV3} for $R_{IV,3}$. Using \eqref{IV estimate Q IV3} with \eqref{IV estimate v IV4} to estimate the integral in the equation \eqref{eq v IV4 pm} we conclude that
\begin{equation*}
    v_{IV,4}^{\pm}(z)\rightarrow e_{\pm}\text{ as }z\rightarrow-\infty,
\end{equation*}
and $\widetilde v_{IV,3}^{\,\pm}(z)$ have the same
asymptotics as $v_{IV,3}^{\pm}(z)$. Therefore $\widetilde v_{IV,3}^{\,\pm}(z)=v_{IV,3}^{\pm}(z)$ and
\begin{equation}\label{v IV2 via vIV5}
    v_{IV,2}^{\pm}(z)=(I+\widehat T_{IV}(z,0))
    \frac{e^{-iE_1(z)}}{(-z)^{\frac14}}
    v_{IV,4}^{\pm}(z).
\end{equation}
Using \eqref{u IV4} and \eqref{u IV3} we define functions
\begin{equation*}
    \widetilde u_{IV,2}^{\pm}(z,\varepsilon):=
    (I+\widehat T_{IV}(z,0))\frac{e^{-iE_1(z)}}{(-z)^{\frac14}} u_{IV,4}^{\pm}(z,\varepsilon),
\end{equation*}
and they are solutions of the system \eqref{system u IV2}. From the convergence in \eqref{IV estimate u-v IV4} and  the equality \eqref{v IV2 via vIV5} we conclude that
\begin{equation*}
    \widetilde u_{IV,2}^{\pm}(z,\varepsilon)\rightarrow v_{IV,2}^{\pm}(z)
\end{equation*}
as $\varepsilon\rightarrow0^+$ for every fixed $z\le-Z_0$. For $z=-Z_1(\varepsilon)$ we use the estimate for $\widehat T_{IV}(z,0)$ by Lemma \ref{lem IV estimate T hat} and the fact that $u_{IV,4}^{\pm}(-Z_1(\varepsilon),\varepsilon)=e_{\pm}$ which follows from the equation \eqref{eq u IV4 pm} to get
\begin{equation*}
    \widetilde u_{IV,2}^{\pm}(-Z_1(\varepsilon),\varepsilon)=(I+o(1))\frac{e^{-iE_1(Z_1(\varepsilon))}}{(Z_1(\varepsilon))^{\frac14}}e_{\pm}
    =\frac{\exp\left(\mp\frac{2ic_0}{3}(Z_1(\varepsilon))^{\frac32}\right)}{(Z_1(\varepsilon))^{\frac14}}(e_{\pm}+o(1))
\end{equation*}
as $\varepsilon\rightarrow0^+$. This completes the proof of Lemma \ref{lem IV result 1st part}.
\end{proof}

The result for the interval $[-Z_2(\varepsilon),-Z_1(\varepsilon)]$ is given by the following lemma.

\begin{lem}\label{lem IV result 2nd part}
Let the conditions of Lemma \ref{lem IV result} hold and let $Z_1(\varepsilon)$ be  given by \eqref{Z 1}. If
\begin{equation*}
    \int_{-Z_2(\varepsilon)}^{-Z_1(\varepsilon)}\frac{\|R_{IV}(t(s,\varepsilon),\varepsilon)\|}{\sqrt{-s}}ds=o\left(\varepsilon^{\frac23}\right)
    \text{ as }\varepsilon\rightarrow0^+,
\end{equation*}
then for every sufficiently small $\varepsilon>0$ the system \eqref{system u IV2} on the interval $[-Z_2(\varepsilon),-Z_1(\varepsilon)]$ has two solutions $\widehat u_{IV,2}^{\pm}(z,\varepsilon)$ such that, as $\varepsilon\rightarrow0^+$,
\begin{equation}\label{IV answer 2nd part}
    \widehat u_{IV,2}^{\pm}(z,\varepsilon)
    =\exp
    \left(
    -\int_z^{-Z_1(\varepsilon)}
    \left(
    \pm\frac{ic_0\sqrt{-s}}{(1-\varepsilon^{\frac23}s)^{\varkappa}}-\frac1{4s(1\mp i\varepsilon^{\frac13}\sqrt{-s})}
    \right)
    ds
    \right)
    (e_{\pm}+o(1)),
\end{equation}
where the vectors $e_{\pm}$ are given by \eqref{e +-} and the remainder $o(1)$ converges uniformly with respect to $z\in[-Z_2(\varepsilon),-Z_1(\varepsilon)]$.
\end{lem}

\begin{proof}
Let us make the Harris--Lutz transformation
\begin{equation}\label{u IV7}
    u_{IV,2}(z)=(I+\widehat T_{IV}(z,\varepsilon))u_{IV,7}(z)
\end{equation}
where $\widehat T_{IV}$ is given by formula \eqref{T IV hat}. The substitution gives
\begin{equation}\label{system u IV7}
    u_{IV,7}'(z)=(\Lambda_{IV,7}(z,\varepsilon)+R_{IV,7}(z,\varepsilon))u_{IV,7}(z),
\end{equation}
where
\begin{equation}\label{Lambda IV7}
    \Lambda_{IV,7}(z,\varepsilon):=i\left(\frac{c_0\sqrt{-z}}{(1-\varepsilon^{\frac32}z)^{\varkappa}}+\frac{\varepsilon^{\frac13}}{4\sqrt{-z}(1-\varepsilon^{\frac23}z)}\right)
    \left(
      \begin{array}{cc}
        1 & 0 \\
        0 & -1 \\
      \end{array}
    \right)
    -
    \frac1{4z(1-\varepsilon^{\frac23}z)}I,
\end{equation}
\begin{equation*}
    \Lambda_{IV,7}(z,\varepsilon)=\Lambda_{IV,2}(z,\varepsilon)+\text{diag}\,S_{IV,2}(z,\varepsilon)=
    \left(
      \begin{array}{cc}
        \lambda_{IV,7}^+(z,\varepsilon) & 0 \\
        0 & \lambda_{IV,7}^-(z,\varepsilon) \\
      \end{array}
    \right),
\end{equation*}
\begin{equation*}
    \lambda_{IV,7}^{\pm}(z,\varepsilon):=
    \pm\frac{ic_0\sqrt{-z}}{(1-\varepsilon^{\frac23}z)^{\varkappa}}-\frac1{4z(1\mp i\varepsilon^{\frac13}\sqrt{-z})}
\end{equation*}
and, from \eqref{Harris--Lutz R 1},
\begin{multline}\label{R IV7}
    R_{IV,7}(z,\varepsilon):=(I+\widehat T_{IV}(z,\varepsilon))^{-1}R_{IV,2}(z,\varepsilon)(I+\widehat T_{IV}(z,\varepsilon))
    \\
    +(I+\widehat T_{IV}(z,\varepsilon))^{-1}
    (S_{IV,2}(z,\varepsilon)\widehat T_{IV}(z,\varepsilon)-\widehat T_{IV}(z,\varepsilon)\,\text{diag}\,S_{IV,2}(z,\varepsilon))
\end{multline}
with $S_{IV,2}$ given by \eqref{S IV2}. From Lemma \ref{lem IV remainder} and the estimate of $T_{IV}(z,\varepsilon)$ by Lemma \ref{lem IV estimate T hat} we get:
\begin{equation}\label{IV estimate R IV7}
    \int_{-Z_2(\varepsilon)}^{-Z_1(\varepsilon)}\|R_{IV,7}(s,\varepsilon)\|ds\rightarrow0\text{ as }\varepsilon\rightarrow0^+.
\end{equation}
Variation of parameters
\begin{equation}\label{u IV8}
    u_{IV,7}(z)=\exp\left(-\int_z^{-Z_1(\varepsilon)}\Lambda_{IV,7}(\sigma,\varepsilon)d\sigma\right)u_{IV,8}(z),
\end{equation}
leads to the system
\begin{multline}\label{system u IV8}
    u_{IV,8}'(z)=\exp\left(\int_z^{-Z_1(\varepsilon)}\Lambda_{IV,7}(\sigma,\varepsilon)d\sigma\right)
    \\
    \times
    R_{IV,7}(z,\varepsilon)
    \exp\left(-\int_z^{-Z_1(\varepsilon)}\Lambda_{IV,7}(\sigma,\varepsilon)d\sigma\right)u_{IV,8}(z).
\end{multline}
Let us denote
\begin{equation*}
    E_2(z,\varepsilon):=\int_z^{-Z_1(\varepsilon)}
    \left(\frac{c_0\sqrt{-\sigma}}{(1-\varepsilon^{\frac32}\sigma)^{\varkappa}}+\frac{\varepsilon^{\frac13}}{4\sqrt{-\sigma}(1-\varepsilon^{\frac23}\sigma)}\right)
    \left(
      \begin{array}{cc}
        1 & 0 \\
        0 & -1 \\
      \end{array}
    \right)d\sigma,
\end{equation*}
and then
\begin{equation*}
    \exp\left(-\int_z^{-Z_1(\varepsilon)}\Lambda_{IV,7}(\sigma,\varepsilon)d\sigma\right)
    =e^{-iE_2(z,\varepsilon)}\exp\left(\int_z^{-Z_1(\varepsilon)}\frac{d\sigma}{4\sigma(1-\varepsilon^{\frac23}\sigma)}\right),
\end{equation*}
so that the equation reads
\begin{equation*}
    u_{IV,8}'(z)=e^{iE_2(z,\varepsilon)}R_{IV,7}(z,\varepsilon)e^{-iE_2(z,\varepsilon)}u_{IV,8}(z).
\end{equation*}
Let us now introduce two solutions $u_{IV,8}^{\pm}$ of this system which satisfy the following equations:
\begin{equation*}
    u_{IV,8}^{\pm}(z,\varepsilon)=e_{\pm}+\int_{-Z_2(\varepsilon)}^{z}e^{iE_2(s,\varepsilon)}R_{IV,7}(s,\varepsilon)e^{-iE_2(s,\varepsilon)}u_{IV,8}^{\pm}(s,\varepsilon)ds.
\end{equation*}
They can be rewritten as
\begin{multline}\label{IV eq 7}
    u_{IV,8}^{\pm}(z,\varepsilon)-e_{\pm}=\int_{-Z_2(\varepsilon)}^{z}e^{iE_2(s,\varepsilon)}R_{IV,7}(s,\varepsilon)e^{-iE_2(s,\varepsilon)}e_{\pm}ds
    \\
    +\int_{-Z_2(\varepsilon)}^{z}e^{iE_2(s,\varepsilon)}R_{IV,7}(s,\varepsilon)e^{-iE_2(s,\varepsilon)}(u_{IV,8}^{\pm}(s,\varepsilon)-e_{\pm})ds.
\end{multline}
For these equalities Lemma \ref{lem II tech 1} yields:
\begin{multline}\label{IV convergence u IV8}
    \sup_{z\in[-Z_2(\varepsilon),-Z_1(\varepsilon)]}\left\|u_{IV,8}^{\pm}(z,\varepsilon)-e_{\pm}\right\|
    \\
    \le
    \exp\left(\int_{-Z_2(\varepsilon)}^{-Z_1(\varepsilon)}\|R_{IV,7}(s,\varepsilon)\|ds\right)
    \int_{-Z_2(\varepsilon)}^{-Z_1(\varepsilon)}\|R_{IV,7}(s,\varepsilon)\|ds\rightarrow0
\end{multline}
as $\varepsilon\rightarrow0^+$. Now using the relations \eqref{u IV7} and \eqref{u IV8} we define
\begin{equation*}
    \widehat u_{IV,2}^{\pm}(z,\varepsilon):=(I+\widehat T_{IV}(z,\varepsilon))
    \exp\left(-\int_z^{-Z_1(\varepsilon)}\Lambda^{\pm}_{IV,7}(\sigma,\varepsilon)d\sigma\right)u_{IV,8}^{\pm}(z,\varepsilon)
\end{equation*}
which are solutions of the system \eqref{u IV2}. Lemma \ref{lem IV estimate T hat}, the expression \eqref{Lambda IV7} and convergence in \eqref{IV convergence u IV8} imply that these solutions have asymptotics \eqref{IV answer 2nd part}. This completes the proof.
\end{proof}

Combining the results of Lemmas \ref{lem IV result 1st part} and \ref{lem IV result 2nd part} we now come to the proof of Lemma \ref{lem IV result}.

\begin{proof}[Proof of Lemma \ref{lem IV result}]
First let us rewrite the formula \eqref{IV answer 1st part} for the asymptotics from Lemma \ref{lem IV result 1st part} using that
\begin{equation*}
    \frac{\exp\left(\mp\frac{2ic_0}3(Z_1(\varepsilon))^{\frac32}\right)}{(Z_1(\varepsilon))^{\frac14}}
    =a^{\pm}_{IV}\exp\left(-\int_{-Z_1(\varepsilon)}^{-Z_0}
    \lambda_{IV,7}^{\pm}(s,\varepsilon)ds\right)(1+o(1))
\end{equation*}
as $\varepsilon\rightarrow0^+$, where $a^{\pm}_{IV}:=\frac{\exp\left(\mp\frac{2ic_0}3Z_0^{\frac32}\right)}{Z_0^{\frac14}}$, which is true because
\begin{equation*}
    \left|\lambda_{IV,7}^{\pm}(s,\varepsilon)-
    \left(
    \pm ic_0\sqrt{-s}
    -\frac1{4s}
    \right)
    \right|
    <
    c_{28}
    \left(
    \varepsilon^{\frac23}|s|^{\frac32}+\frac{\varepsilon^{\frac13}}{\sqrt{|s|}},
    \right)
\end{equation*}
with some $c_{28}>0$ and
\begin{equation*}
    \int_{Z_0}^{Z_1(\varepsilon)}\left(\varepsilon^{\frac23}s^{\frac32}+\frac{\varepsilon^{\frac13}}{\sqrt s},
    \right)ds\rightarrow0\text{ as }\varepsilon\rightarrow0^+
\end{equation*}
due to the choice of $Z_1(\varepsilon)$ in \eqref{Z 1}.
Let us define for $z\in[-Z_2(\varepsilon),-Z_1(\varepsilon)]$
\begin{equation}\label{u IV2-}
    u_{IV,2}^{\pm}(z,\varepsilon):=a_{IV}^{\pm}\exp
    \left(-
    \int_{-Z_1(\varepsilon)}^{-Z_0}
    \lambda_{IV,7}^{\pm}(s,\varepsilon)ds
    \right)
    \widehat u_{IV,2}^{\pm}(z,\varepsilon).
\end{equation}
Continuations of the solutions $\widehat u_{IV,2}^{\pm}(z,\varepsilon)$ to the interval $[-Z_1(\varepsilon),-Z_0]$ have decompositions with some coefficients in terms of the basis of solutions $\widetilde u_{IV,2}^{\pm}(z,\varepsilon)$, and at the point $-Z_1(\varepsilon)$ one has:
\begin{equation*}
    \begin{array}{rl}
    \widehat u_{IV,2}^{\pm}(-Z_1(\varepsilon),\varepsilon)=&e_{\pm}+o(1),
    \\
    \frac{\widetilde u_{IV,2}^{\pm}(-Z_1(\varepsilon),\varepsilon))}{a_{IV}^{\pm}}
    \exp
    \left(
    \int_{-Z_1(\varepsilon)}^{-Z_0}\lambda_{IV,7}^{\pm}(s,\varepsilon)ds
    \right)
    =&e_{\pm}+o(1).
    \end{array}
\end{equation*}
By Lemma \ref{lem II tech 2} we conclude that for $z\in[-Z_1(\varepsilon),-Z_0]$
\begin{multline*}
    \widehat u_{IV,2}^{\pm}(z,\varepsilon)=(1+o(1))
    \frac{\widetilde u_{IV,2}^{\pm}(z,\varepsilon))}{a_{IV}^{\pm}}
    \exp
    \left(
    \int_{-Z_1(\varepsilon)}^{-Z_0}\lambda_{IV,7}^{\pm}(s,\varepsilon)ds
    \right)
    \\
    +o
    \left(
    \widetilde u_{IV,2}^{\mp}(z,\varepsilon))
    \exp
    \left(
    \int_{-Z_1(\varepsilon)}^{-Z_0}\lambda_{IV,7}^{\mp}(s,\varepsilon)ds
    \right)
    \right)
\end{multline*}
and, since $\lambda_{IV,7}^+(s,\varepsilon)=\overline{\lambda_{IV,7}^+(s,\varepsilon)}$, by \eqref{u IV2-} we have
\begin{equation*}
    u_{IV,2}^{\pm}(z,\varepsilon)=(1+o(1))
    \widetilde u_{IV,2}^{\pm}(z,\varepsilon)
    +o
    \left(
    \widetilde u_{IV,2}^{\mp}(z,\varepsilon))
    \right).
\end{equation*}
For every fixed $z\le-Z_0$ this means that
\begin{equation*}
    u_{IV,2}^{\pm}(z,\varepsilon)\rightarrow v_{IV,2}^{\pm}(z)\text{ as }\varepsilon\rightarrow0^+.
\end{equation*}
Asymptotics of $u_{IV,2}^{\pm}(z,\varepsilon)$ at $z=-Z_2(\varepsilon)$ are due to Lemma \ref{lem IV result 2nd part}.
\end{proof}

\section{Neighbourhood of the turning point (region $III$)}\label{section III}
Consider the system
\begin{equation}\label{system u III}
    \varepsilon u_{III}'(t)=
    \left(
    \left(
    \begin{array}{cc}
    \frac{\beta}{t^{\gamma}} & -\frac12
    \\
    \frac12 & -\frac{\beta}{t^{\gamma}}
    \end{array}
    \right)
    +R_{III}(t,\varepsilon)
    \right)
    u_{III}(t).
\end{equation}
The main term of its coefficient matrix is analytic near the turning point $t_0=(2\beta)^{\frac1{\gamma}}$ and degenerates as a matrix at $t_0$. Analytic theory for the case $R_{III}\equiv0$ is well known (see, for example, \cite[Chapter VIII]{Wasow-1965}). It suggests the transformation
\begin{equation}\label{u III1}
    u_{III}(t)=T_{III}(t)u_{III,1}(t)
\end{equation}
with
\begin{equation}\label{T III}
T_{III}(t):=
\left(
  \begin{array}{cc}
    1 & \frac{\beta}{t^{\gamma}}+\frac12 \\
    1 & -\frac{\beta}{t^{\gamma}}-\frac12 \\
  \end{array}
\right)
\end{equation}
which makes the structure of the main term simpler:
\begin{equation}\label{system u III1}
    \varepsilon u_{III,1}'(t)
    =\left(
\left(
  \begin{array}{cc}
    0 & 1 \\
    \frac{\beta^2}{t^{2\gamma}}-\frac14 & 0 \\
  \end{array}
\right)
+R_{III,1}(t,\varepsilon)
\right)
u_{III,1}(t),
\end{equation}
where
\begin{equation}\label{R III1}
    R_{III,1}(t,\varepsilon):=-\frac{\varepsilon\beta\gamma}{t^{1+\gamma}\left(\frac{\beta}{t^{\gamma}}+\frac12\right)}
\left(
  \begin{array}{cc}
    0 & 0 \\
    0 & 1 \\
  \end{array}
\right)
+
T_{III}^{-1}(t)R_{III}(t,\varepsilon)T_{III}(t).
\end{equation}
The problem comes from the remainder $R_{III,1}$ which is by no means analytic (we know that it wildly oscillates), but is small in the integral sense. Analytic theory would proceed with making the change of the variable $\tau(t)\sim\text{const}\cdot(t-t_0)$ as $t\rightarrow t_0$ and considering
\begin{equation*}
    u_{a}(\tau)=P(\tau,\varepsilon)u_{III,1}(t(\tau))
\end{equation*}
with the matrix-valued function $P$ analytic in both variables such that the system \eqref{system u III1} is transformed into
\begin{equation*}
    \varepsilon u_{a}'(\tau)=
    \left(
       \begin{array}{cc}
         0 & 1 \\
         \tau & 0 \\
       \end{array}
     \right)
    u_{a}(\tau).
\end{equation*}
Using the variable
\begin{equation*}
    \zeta=\frac{\tau}{\varepsilon^{\frac23}}
\end{equation*}
and the function $u_{a,1}$ defined by the equality
\begin{equation*}
    u_a(\tau)=
    \left(
      \begin{array}{cc}
        1 & 0 \\
        0 & \varepsilon^{\frac13} \\
      \end{array}
    \right)
    u_{a,1}\left(\frac{\tau}{\varepsilon^{\frac23}}\right)
\end{equation*}
this system can be further transformed into
\begin{equation*}
    u_{a,1}'(\zeta)=
    \left(
       \begin{array}{cc}
         0 & 1 \\
         \zeta & 0 \\
       \end{array}
     \right)
    u_{a,1}(\zeta).
\end{equation*}
Solutions of the latter are expressed in terms of Airy functions: the system has the following matrix solution:
\begin{equation*}
    U_{a,1}(\zeta)=
    \left(
      \begin{array}{cc}
        \Ai(\zeta) & \Bi(\zeta) \\
        \Ai'(\zeta) & \Bi'(\zeta) \\
      \end{array}
    \right).
\end{equation*}
In our case, due to presence of the remainder $R_{III,1}$, such transformations are not possible. Instead, in this section we show that in the scale of the variable $z$ which was used in the regions $II$ and $IV$, on any fixed interval $[-z_0,z_0]$, the presence of the remainder $R_{III}$ does not affect the asymptotics of solutions of the system \eqref{system u III}. In terms of the variable $t$ the interval $[-Z_0,Z_0]$ of the region $III$ corresponds to the interval $[t_{II-III}(\varepsilon),t_{III-IV}(\varepsilon)]$ which shrinks to the turning point $t_0$. Analytic method of \cite{Wasow-1965} gives the result for a fixed neighbourhood of the turning point in the scale of the variable $t$, the result which we do not have here. However, we already know what happens in the regions $II$ and $IV$.

Let consider the variable $z=\frac{1-\frac{t^{2\gamma}}{4\beta^2}}{\varepsilon^{\frac23}}$ and take
\begin{equation}\label{u III2}
    u_{III,1}(t)=
\left(
  \begin{array}{cc}
    2 & 0 \\
    0 & -\varepsilon^{\frac13} \\
  \end{array}
\right)
u_{III,2}(z(t,\varepsilon)).
\end{equation}
Substituting this to \eqref{system u III1} and simplifying the result we have:
\begin{multline*}
    u_{III,2}'(z)=\frac{c_0}{(1-\varepsilon^{\frac23}z)^{1-\frac1{2\gamma}}}
    \Bigg(
    \left(
      \begin{array}{cc}
        0 & 1 \\
        \frac{z}{1-\varepsilon^{\frac23}z} & 0 \\
      \end{array}
    \right)
    \\
    -\frac1{\varepsilon^{\frac23}}
    \left(
      \begin{array}{cc}
        \varepsilon^{\frac13} & 0 \\
        0 & -2 \\
      \end{array}
    \right)
    R_{III,1}(t(z,\varepsilon),\varepsilon)
    \left(
      \begin{array}{cc}
        2 & 0 \\
        0 & -\varepsilon^{\frac13} \\
      \end{array}
    \right)
    \Bigg)
    u_{III,2}(z)
\end{multline*}
with $c_0$ given by \eqref{c_0 kappa}. This system can be written in the form
\begin{equation}\label{system u III2}
    u_{III,2}'(z)
    =\left(
    c_0
    \left(
      \begin{array}{cc}
        0 & 1 \\
        z & 0 \\
      \end{array}
    \right)
    +
    R_{III,2}(z,\varepsilon)
    \right)
    u_{III,2}(z),
\end{equation}
where
\begin{multline}\label{R III2}
    R_{III,2}(z,\varepsilon):=c_0\left(\frac1{(1-\varepsilon^{\frac23}z)^{1-\frac1{2\gamma}}}
    \left(
      \begin{array}{cc}
        0 & 1 \\
        \frac{z}{1-\varepsilon^{\frac23}z} & 0 \\
      \end{array}
    \right)
    -
    \left(
      \begin{array}{cc}
        0 & 1 \\
        z & 0 \\
      \end{array}
    \right)
    \right)
    \\
    -
    \frac{c_0}{\varepsilon^{\frac23}(1-\varepsilon^{\frac23}z)^{1-\frac1{2\gamma}}}
    \left(
      \begin{array}{cc}
        \varepsilon^{\frac13} & 0 \\
        0 & -2 \\
      \end{array}
    \right)
    R_{III,1}(t(z,\varepsilon),\varepsilon)
    \left(
      \begin{array}{cc}
        2 & 0 \\
        0 & -\varepsilon^{\frac13} \\
      \end{array}
    \right).
\end{multline}
The following lemma is the main result for the region $III$.

\begin{lem}\label{lem III main}
Let $c_0,z_0>0$ and $R_{III,2}(z,\varepsilon)$ is given by \eqref{R III2}, \eqref{R III1} and \eqref{T III}. If
\begin{equation*}
    \int_{-z_0}^{z_0}\|R_{III}(t(s,\varepsilon),\varepsilon)\|ds=o\left(\varepsilon^{\frac23}\right)
    \text{ as }\varepsilon\rightarrow0^+,
\end{equation*}
then the system \eqref{system u III2} has the matrix solution $U_{III,2}(z,\varepsilon)$ such that for every $z\in[-z_0,z_0]$
\begin{equation}\label{U III2 asymptotics}
    U_{III,2}(z,\varepsilon)\rightarrow
    \left(
      \begin{array}{cc}
        \Ai(c_0^{\frac23}z) & \Bi(c_0^{\frac23}z) \\
        c_0^{-\frac13}\Ai\,'(c_0^{\frac23}z) & c_0^{-\frac13}\Bi\,'(c_0^{\frac23}z) \\
      \end{array}
    \right)
    \text{ as }\varepsilon\rightarrow0^+.
\end{equation}
\end{lem}

First let us estimate the remainder in the system \eqref{system u III2}.

\begin{lem}\label{lem III remainder estimate}
Under the conditions of Lemma \ref{lem III main}, if
\begin{equation}\label{III eq integral convergence of the remainder}
    \int_{-z_0}^{z_0}\|R_{III}(t(s,\varepsilon),\varepsilon)\|ds=o\left(\varepsilon^{\frac23}\right),
    \text{ then }
    \int_{-z_0}^{z_0}\|R_{III,2}(s,\varepsilon)\|ds\rightarrow0
\end{equation}
as $\varepsilon\rightarrow0^+$.
\end{lem}

\begin{proof}
From the expression \eqref{R III2} one immediately has: there exists $c_{29}>0$ such that for every sufficiently small $\varepsilon$ and every $z\in[-z_0,z_0]$
\begin{equation*}
    \|R_{III,2}(z,\varepsilon)\|\le c_{29}\left(\varepsilon^{\frac23}+\varepsilon^{-\frac23}\|R_{III,1}(t(z,\varepsilon),\varepsilon)\|\right).
\end{equation*}
Furthermore, due to \eqref{R III1}, boundedness and bounded invertibility of $T_{III}$ in the neighbourhood of the point $t_0$, one has:
\begin{equation*}
    \|R_{III,2}(z,\varepsilon)\|\le c_{30}\left(\varepsilon^{\frac13}+\varepsilon^{-\frac23}\|R_{III}(t(z,\varepsilon),\varepsilon)\|\right)
\end{equation*}
with some $c_{30}>0$, which converges to zero as $\varepsilon\rightarrow0^+$.
\end{proof}

Now let us see that conditions of Lemma \ref{lem III main} are satisfied, if $R_{III}=R_2^+$.

\begin{lem}\label{lem III R 2+ estimate}
Let $z_0>0$, let $R_2^+(t,\varepsilon)$ be given by \eqref{R 2 pm}, $t(z,\varepsilon)$ be defined by \eqref{z} and conditions \eqref{model problem conditions} and \eqref{model problem r} hold. Then
\begin{equation*}
    \int_{-z_0}^{z_0}\|R_2^+(s,\varepsilon)\|ds=O\left(\varepsilon\varepsilon_0^{\frac{\alpha_r}{\gamma}}\right)
    \text{ as }\varepsilon\rightarrow0^+.
\end{equation*}
\end{lem}

\begin{proof}
Since $\frac{dz}{dt}(t)=-\frac{\gamma t^{2\gamma-1}}{2\beta^2\varepsilon^{\frac23}}$, we have with some $c_{31}>0$:
\begin{equation*}
    \int_{-z_0}^{z_0}\|R_2^+(s,\varepsilon)\|ds
    <\frac{c_{31}}{\varepsilon^{\frac23}}
    \int_{t_0(1-\varepsilon^{\frac23}z_0)^{\frac1{2\gamma}}}^{t_0(1+\varepsilon^{\frac23}z_0)^{\frac1{2\gamma}}}
    \|R_2^+(t,\varepsilon)\|dt.
\end{equation*}
By equalities \eqref{R 2 pm} and the estimate of the norm of $R$ from conditions \eqref{model problem conditions} we have:
\begin{equation*}
    \frac1{\varepsilon^{\frac23}}
    \int_{t_0(1-\varepsilon^{\frac23}z_0)^{\frac1{2\gamma}}}^{t_0(1+\varepsilon^{\frac23}z_0)^{\frac1{2\gamma}}}
    \|R_2^+(t,\varepsilon)\|dt=
    \frac1{\varepsilon_0\varepsilon^{\frac23}}
    \int_{t_0(1-\varepsilon^{\frac23}z_0)^{\frac1{2\gamma}}}^{t_0(1+\varepsilon^{\frac23}z_0)^{\frac1{2\gamma}}}
    r\left(\varepsilon_0^{-\frac1{\gamma}}t\right)dt.
\end{equation*}
Using the expression \eqref{model problem r} for $r$ we have:
\begin{equation*}
    \int_{t_0(1-\varepsilon^{\frac23}z_0)^{\frac1{2\gamma}}}^{t_0(1+\varepsilon^{\frac23}z_0)^{\frac1{2\gamma}}}
    r\left(\varepsilon_0^{-\frac1{\gamma}}t\right)dt
    <c_r\varepsilon_0^{\frac{1+\alpha_r}{\gamma}}
    \int_{t_0(1-\varepsilon^{\frac23}z_0)^{\frac1{2\gamma}}}^{t_0(1+\varepsilon^{\frac23}z_0)^{\frac1{2\gamma}}}
    \frac{dt}{t^{1+\alpha_r}}
    =O\left(\varepsilon_0^{\frac{1+\alpha_r}{\gamma}}\varepsilon^{\frac23}\right)
\end{equation*}
as $\varepsilon\rightarrow0^+$. Putting everything together we come to the following:
\begin{equation*}
    \int_{-z_0}^{z_0}\|R_2^+(t(s,\varepsilon),\varepsilon)\|ds
    =O\left(\varepsilon_0^{\frac{1+\alpha_r}{\gamma}-1}\right)
    =O\left(\varepsilon\varepsilon_0^{\frac{\alpha_r}{\gamma}}\right)
\end{equation*}
as $\varepsilon\rightarrow0^+$. This completes the proof.
\end{proof}

We need the following technical lemma, which is simple and standard.

\begin{lem}\label{lem III tech}
Let $A(x,\varepsilon)$ be a $n\times n$ matrix-valued function defined for $x\in[a,b]$ and $\varepsilon$ from some neighbourhood of zero. Let $A(\cdot,\varepsilon)\in L_1((a,b),M^{n\times n}(\mathbb C))$ for every $\varepsilon$ and $A(x,\varepsilon)\rightarrow A(x,0)$ in the norm of $L_1((a,b),M^{n\times n}(\mathbb C))$. If $U(x)$ is a non-degenerate matrix solution of the system
\begin{equation}\label{III eq lemma tech}
    u'(x)=A(x,\varepsilon)u(x),
\end{equation}
for $\varepsilon=0$, then there exist non-degenerate solutions $U(x,\varepsilon)$ of this system for all $\varepsilon\neq0$ such that $U(x,\varepsilon)\rightarrow U(x)$ as $\varepsilon\rightarrow0$ uniformly in $x\in[a,b]$.
\end{lem}

\begin{proof}
Let us look for $U(x,\varepsilon)=U(x)Y(x,\varepsilon)$. The system \eqref{III eq lemma tech} is equivalent to
\begin{equation*}
    Y'(x)=U^{-1}(x)(A(x,\varepsilon)-A(x,0))U(x)Y(x).
\end{equation*}
Function $U$ is bounded on the interval $[a,b]$ and its determinant,
\begin{equation*}
    \text{det}\,U(x)=\text{det}\,U(a)\exp\left(\int_a^x\text{tr}\,A(t,0)dt\right),
\end{equation*}
is separated from zero, since $A(\cdot,0)\in L_1((a,b),M^{n\times n}(\mathbb C))$. Therefore the function $U^{-1}$ is also bounded on $[a,b]$. Take $Y$ as the solution of the following Volterra equation:
\begin{equation*}
    Y(x,\varepsilon)=I+\int_a^xU^{-1}(t)(A(t,\varepsilon)-A(t,0))U(t)Y(t,\varepsilon)dt.
\end{equation*}
Repeating the standard argument on inverting $I-\mathcal K_{III}(\varepsilon)$ where
\begin{equation*}
    \mathcal K_{III}(\varepsilon):Y(x)\mapsto\int_a^xU^{-1}(t)(A(t,\varepsilon)-A(t,0))U(t)Y(t,\varepsilon)dt
\end{equation*}
is a Volterra operator in the space $L_{\infty}((a,b),M^{n\times n}(\mathbb C))$ and estimating the norm of the inverse, we finally come to the estimate
\begin{equation*}
    \sup_{x\in[a,b]}\|Y(x,\varepsilon)-I\|
    \le\|\mathcal K_{III}(\varepsilon)\|_{\mathcal B(L_{\infty}(a,b))}\exp(\|\mathcal K_{III}(\varepsilon)\|_{\mathcal B(L_{\infty}(a,b))})
\end{equation*}
and
\begin{equation*}
    \|\mathcal K_{III}(\varepsilon)\|_{\mathcal B(L_{\infty}(a,b))}\le\int_a^b\|U^{-1}(t)(A(t,\varepsilon)-A(t,0))U(t)\|dt\rightarrow0
\end{equation*}
as $\varepsilon\rightarrow0^+$. From this the statement of the lemma follows.
\end{proof}

Now we are ready to prove Lemma \ref{lem III main}.

\begin{proof}[Proof of Lemma \ref{lem III main}]
Due to Lemmas \ref{lem III tech} and \ref{lem III remainder estimate}, we only need to check that the system
\begin{equation*}
    v_{III,2}'(z)
    =
    c_0
    \left(
      \begin{array}{cc}
        0 & 1 \\
        z & 0 \\
      \end{array}
    \right)
    v_{III,2}(z)
\end{equation*}
has the solution
\begin{equation}\label{V III2}
    V_{III,2}(z)=
    \left(
      \begin{array}{cc}
        \Ai(c_0^{\frac23}z) & \Bi(c_0^{\frac23}z) \\
        c_0^{-\frac13}\Ai\,'(c_0^{\frac23}z) & c_0^{-\frac13}\Bi\,'(c_0^{\frac23}z) \\
      \end{array}
    \right).
\end{equation}
This can be verified by the entry-wise direct substitution using the Airy equation $u''(x)=xu(x)$.
\end{proof}

\section{Matching of the results in regions}\label{section matching}
In this section we prove Theorem \ref{thm model problem} by putting together results of the previous five sections considering $R_{II}=R_{III}=R_{IV}=R_{V}=R_2^+$.

For the regions $II$, $III$ and $IV$ let us, according to Lemmas \ref{lem II result}, \ref{lem IV result} and \ref{lem III main}, define solutions $u_{II}^{\pm}$, $u_{IV}^{\pm}$ and $U_{III}$ (matrix-valued) of the system \eqref{system u 2} in the following way :
\begin{equation}\label{u II pm}
    u_{II}^{\pm}(t,\varepsilon):=T_{II}(t)u_{II,2}^{\pm}(z(t,\varepsilon),\varepsilon),
\end{equation}
\begin{equation}\label{u IV pm}
    u_{IV}^{\pm}(t,\varepsilon):=T_{IV}(t)u_{IV,2}^{\pm}(z(t,\varepsilon),\varepsilon)
\end{equation}
and
\begin{equation}\label{U III}
    U_{III}(t,\varepsilon):=T_{III}(t)
    \left(
      \begin{array}{cc}
        2 & 0 \\
        0 & -\varepsilon^{\frac13} \\
      \end{array}
    \right)
    U_{III,2}(z(t,\varepsilon),\varepsilon)
\end{equation}
with $T_{II}$, $T_{III}$ and $T_{IV}$ given by \eqref{T II}, \eqref{T III} and \eqref{T IV}.  These all are solutions of the same system and hence are linearly dependent with coefficients which depend on $\varepsilon$. By matching we mean finding asymptotic behaviour of these coefficients as $\varepsilon\rightarrow0^+$. The following lemma matches solutions in the regions $II$ and $IV$ using the results for the region $III$.

\begin{lem}\label{lem matching II-IV}
Let solutions $u_{II}^{\pm}(t,\varepsilon)$ and $u_{IV}^{\pm}(t,\varepsilon)$ of the system \eqref{system u 2} be defined by \eqref{u II pm} and \eqref{u IV pm}. One has:
\begin{equation}\label{eq u II + via u IV pm}
    u_{II}^+(t,\varepsilon)=\left(\frac i{\sqrt 2}+\delta_1(\varepsilon)\right)u_{IV}^+(t,\varepsilon)
    +\left(-\frac i{\sqrt 2}+\delta_2(\varepsilon)\right)u_{IV}^-(t,\varepsilon),
\end{equation}
\begin{equation}\label{eq u II - via u IV pm}
    u_{II}^-(t,\varepsilon)=\left(\frac{1+i\alpha_m}{2\sqrt 2}+\delta_3(\varepsilon)\right)u_{IV}^+(t,\varepsilon)
    +\left(\frac{1-i\alpha_m}{2\sqrt 2}+\delta_4(\varepsilon)\right)u_{IV}^-(t,\varepsilon),
\end{equation}
with some $\alpha_m\in\mathbb C$ and $\delta_1(\varepsilon),\delta_2(\varepsilon),\delta_3(\varepsilon),\delta_4(\varepsilon)\rightarrow0$ as $\varepsilon\rightarrow0^+$.
\end{lem}

\begin{proof}
Linear dependence of solutions can be written in the following form:
\begin{equation*}
    u_{II}^{\pm}(t,\varepsilon)=U_{III}(t,\varepsilon)\xi_{II}^{\pm}(\varepsilon),
\end{equation*}
\begin{equation*}
    u_{IV}^{\pm}(t,\varepsilon)=U_{III}(t,\varepsilon)\xi_{IV}^{\pm}(\varepsilon)
\end{equation*}
with some vector coefficients $\xi_{II}^{\pm}(\varepsilon)$ and $\xi_{IV}^{\pm}(\varepsilon)$, and also in the following form:
\begin{equation*}
    u_{II}^+(t,\varepsilon)=d_+^+(\varepsilon)u_{IV}^+(t,\varepsilon)
    +d_-^+(\varepsilon)u_{IV}^-(t,\varepsilon),
\end{equation*}
\begin{equation*}
    u_{II}^-(t,\varepsilon)=d_+^-(\varepsilon)u_{IV}^+(t,\varepsilon)
    +d_-^-(\varepsilon)u_{IV}^-(t,\varepsilon)
\end{equation*}
with some coefficients $d_+^+(\varepsilon),d_-^+(\varepsilon),d_+^-(\varepsilon),d_-^-(\varepsilon)$ which we need to determine. These coefficients should be related as
\begin{equation*}
    \xi_{II}^+(\varepsilon)=d_+^+(\varepsilon)\xi_{IV}^+(\varepsilon)+d_-^+(\varepsilon)\xi_{IV}^-(\varepsilon),
\end{equation*}
\begin{equation*}
    \xi_{II}^-(\varepsilon)=d_+^-(\varepsilon)\xi_{IV}^+(\varepsilon)+d_-^-(\varepsilon)\xi_{IV}^-(\varepsilon).
\end{equation*}
This can be rewritten in matrix notation as $C_{II}(\varepsilon)=C_{IV}(\varepsilon)D(\varepsilon)$, where
\begin{equation*}
    C_{II}:=\left(\xi_{II}^+|\xi_{II}^-\right), C_{IV}:=\left(\xi_{IV}^+|\xi_{IV}^-\right),
    D:=\left(
      \begin{array}{cc}
        d_+^+ & d_+^- \\
        d_-^+ & d_-^- \\
      \end{array}
    \right).
\end{equation*}
Therefore if we know $\xi_{II}^{\pm}$ and $\xi_{IV}^{\pm}$, we can calculate $d_+^+,d_-^+,d_+^-,d_-^-$ by the formula
\begin{equation}\label{D}
    D(\varepsilon)=C_{IV}(\varepsilon)^{-1}C_{II}(\varepsilon).
\end{equation}
From the expressions \eqref{u II pm}, \eqref{u IV pm} and \eqref{U III} we can write for any solutions $u_{II,2}$, $u_{IV,2}$ and $u_{III,2}$ of the systems \eqref{system u II2}, \eqref{system u IV2} and \eqref{system u III2}:
\begin{equation}\label{eq matching II-III-IV}
    u_{III,2}(z)=P_{II}(z,\varepsilon)u_{II,2}(z)=P_{IV}(z,\varepsilon)u_{IV,2}(z)
\end{equation}
with
\begin{equation}\label{P II}
    P_{II}(z,\varepsilon):=
    \left(
      \begin{array}{cc}
        2 & 0 \\
        0 & -\varepsilon^{\frac13} \\
      \end{array}
    \right)^{-1}
    T_{III}^{-1}(t(z,\varepsilon))T_{II}(t(z,\varepsilon)),
\end{equation}
\begin{equation}\label{P IV}
    P_{IV}(z,\varepsilon):=
    \left(
      \begin{array}{cc}
        2 & 0 \\
        0 & -\varepsilon^{\frac13} \\
      \end{array}
    \right)^{-1}
    T_{III}^{-1}(t(z,\varepsilon))T_{IV}(t(z,\varepsilon)).
\end{equation}
Using the expressions \eqref{T II} and \eqref{T III} for $T_{II}$ and $T_{III}$ we have:
\begin{multline*}
    P_{II}(z,\varepsilon)=\frac1{2\varepsilon^{\frac13}}
    \left(
      \begin{array}{cc}
        \varepsilon^{\frac13} & 0 \\
        0 & -2 \\
      \end{array}
    \right)
    \\
    \times
    \left(
      \begin{array}{cc}
        1 & \frac{\beta}{t^{\gamma}}+\frac12 \\
        1 & -\left(\frac{\beta}{t^{\gamma}}+\frac12\right) \\
      \end{array}
    \right)^{-1}
    \left(
      \begin{array}{cc}
        1 & 1 \\
        \frac{2\beta}{t^{\gamma}}-\sqrt{\frac{4\beta^2}{t^{2\gamma}}-1} & \frac{2\beta}{t^{\gamma}}+\sqrt{\frac{4\beta^2}{t^{2\gamma}}-1} \\
      \end{array}
    \right)
\end{multline*}
with $t=t_0(1-\varepsilon^{\frac23}z)^{\frac{1}{2\gamma}}$. Since
\begin{equation*}
    \frac{2\beta}{t^{\gamma}}=1+o(1)\text{ and }\sqrt{\frac{4\beta^2}{t^{2\gamma}}-1}=\varepsilon^{\frac13}(\sqrt z+o(1))
    \text{ as }\varepsilon\rightarrow0^+,
\end{equation*}
we have
\begin{equation}\label{P II}
    P_{II}(z,\varepsilon)\rightarrow\frac12
    \left(
      \begin{array}{cc}
        1 & 1 \\
        -\sqrt z & \sqrt z \\
      \end{array}
    \right)
    =:P_{II}(z)\text{ as }\varepsilon\rightarrow0^+.
\end{equation}
Analogously,
\begin{multline*}
    P_{IV}(z,\varepsilon)=\frac1{2\varepsilon^{\frac13}}
    \left(
      \begin{array}{cc}
        \varepsilon^{\frac13} & 0 \\
        0 & -2 \\
      \end{array}
    \right)
    \left(
      \begin{array}{cc}
        1 & \frac{\beta}{t^{\gamma}}+\frac12 \\
        1 & -\left(\frac{\beta}{t^{\gamma}}+\frac12\right) \\
      \end{array}
    \right)^{-1}
    \\
    \times
    \left(
      \begin{array}{cc}
        1 & 1 \\
        \frac{2\beta}{t^{\gamma}}+i\sqrt{1-\frac{4\beta^2}{t^{2\gamma}}} & \frac{2\beta}{t^{\gamma}}-i\sqrt{1-\frac{4\beta^2}{t^{2\gamma}}} \\
      \end{array}
    \right)
    \rightarrow
    \frac12
    \left(
      \begin{array}{cc}
        1 & 1 \\
        i\sqrt{-z} & -i\sqrt{-z} \\
      \end{array}
    \right)
    =:P_{IV}(z)
\end{multline*}
as $\varepsilon\rightarrow0^+$. For the solutions $u_{II,2}^{\pm},u_{IV,2}^{\pm}$ and $U_{III,2}$ this means:
\begin{equation*}
    P_{II}(z,\varepsilon)u_{II,2}^{\pm}(z,\varepsilon)=U_{III,2}(z,\varepsilon)\xi_{II}^{\pm}(\varepsilon)
\end{equation*}
and, due to Lemmas \ref{lem II result}, \ref{lem IV result} and \ref{lem III main},
\begin{equation*}
    \xi_{II}^{\pm}(\varepsilon)=U_{III,2}^{-1}(z,\varepsilon)P_{II}(z,\varepsilon)u_{II,2}^{\pm}(z,\varepsilon)\rightarrow V_{III,2}^{-1}(z)P_{II}(z)v_{II,2}^{\pm}(z)=:\xi_{II}^{\pm}
\end{equation*}
and
\begin{equation*}
    \xi_{IV}^{\pm}(\varepsilon)=U_{III,2}^{-1}(z,\varepsilon)P_{IV}(z,\varepsilon)u_{IV,2}^{\pm}(z,\varepsilon)\rightarrow V_{III,2}^{-1}(z)P_{IV}(z)v_{IV,2}^{\pm}(z)=:\xi_{IV}^{\pm},
\end{equation*}
as $\varepsilon\rightarrow0^+$, with $V_{III,2}$ given by \eqref{V III2}, $v_{II,2}^{\pm}$ defined in Lemma \ref{lem II free system} and $v_{IV,2}^{\pm}$ defined in Lemma \ref{lem IV free system}.

To find $\xi_{II}^{\pm}$ consider the equation $P_{II}(z)v_{II,2}^{\pm}(z)=V_{III,2}(z)\xi_{II}^{\pm}$ which, due to the expressions \eqref{V III2} for $V_{III,2}$, \eqref{P II} for $P_{II}$ and Lemma \ref{lem II free system}, reads as
\begin{equation*}
    \frac12
    \left(
      \begin{array}{cc}
        1 & 1 \\
        -\sqrt z & \sqrt z \\
      \end{array}
    \right)
    \frac{\exp\left(\mp\frac{2c_0}3z^{\frac32}\right)}{z^{\frac14}}(1+o(1))
    =
    \left(
      \begin{array}{cc}
        \Ai(c_0^{\frac23}z) & \Bi(c_0^{\frac23}z) \\
        c_0^{-\frac13}\Ai\,'(c_0^{\frac23}z) & c_0^{-\frac13}\Bi\,'(c_0^{\frac23}z) \\
      \end{array}
    \right)
    \xi_{II}^{\pm}.
\end{equation*}
Using the asymptotics of the Airy functions as $z\rightarrow+\infty$, see \cite{Abramowitz-Stegun-1964},
\begin{equation*}
    \begin{array}{rl}
        \Ai(c_0^{\frac23}z)=\frac{\exp\left(-\frac{2c_o}3z^{\frac32}\right)}{2\sqrt\pi c_0^{\frac16}z^{\frac14}}(1+o(1)),
        &
        \Ai\,'(c_0^{\frac23}z)=-\frac{c_0^{\frac16}z^{\frac14}\exp\left(-\frac{2c_o}3z^{\frac32}\right)}{2\sqrt\pi}(1+o(1)),
        \\
        \Bi(c_0^{\frac23}z)=\frac{\exp\left(\frac{2c_o}3z^{\frac32}\right)}{\sqrt\pi c_0^{\frac16}z^{\frac14}}(1+o(1)),
        &
        \Bi\,'(c_0^{\frac23}z)=\frac{c_0^{\frac16}z^{\frac14}\exp\left(-\frac{2c_o}3z^{\frac32}\right)}{\sqrt\pi}(1+o(1)), \\
    \end{array}
\end{equation*}
we conclude that
\begin{equation*}
    \xi_{II}^+=c_0^{\frac16}\sqrt\pi e_+,\ \xi_{II}^-=\frac{c_0^{\frac16}\sqrt\pi}2(e_-+\alpha_me_+)
\end{equation*}
and
\begin{equation}\label{C II}
    C_{II}=\frac{c_0^{\frac16}\sqrt\pi}2
    \left(
      \begin{array}{cc}
        2 & \alpha_m \\
        0 & 1 \\
      \end{array}
    \right).
\end{equation}
with some $\alpha_m$.

To find $\xi_{IV}^{\pm}=V_{III,2}^{-1}(z)P_{IV}(z)v_{IV,2}^{\pm}(z)$ recall that \begin{equation*}
    P_{IV}(z)=\frac12
    \left(
      \begin{array}{cc}
        1 & 1 \\
        i\sqrt{-z} & -i\sqrt{-z} \\
      \end{array}
    \right)
\end{equation*}
and, since $W\{\Ai,\Bi\}=-\frac1{\pi}$ (\cite{Abramowitz-Stegun-1964}),
\begin{equation*}
    V_{III,2}^{-1}(z)=\pi
    \left(
      \begin{array}{cc}
            \Bi\,'(c_0^{\frac23}z) & -c_0^{\frac13}\Bi(c_0^{\frac23}z) \\
            -\Ai\,'(c_0^{\frac23}z) & c_0^{\frac13}\Ai(c_0^{\frac23}z) \\
      \end{array}
    \right),
\end{equation*}
therefore
\begin{multline*}
    \xi_{IV}^{\pm}=\frac{\pi}2
    \left(
      \begin{array}{cc}
            \Bi\,'(c_0^{\frac23}z)-i\sqrt{-c_0^{\frac23}z}\,\Bi(c_0^{\frac23}z) & \Bi\,'(c_0^{\frac23}z)+i\sqrt{-c_0^{\frac23}z}\,\Bi(c_0^{\frac23}z) \\
            -(\Ai\,'(c_0^{\frac23}z)-i\sqrt{-c_0^{\frac23}z}\,\Ai(c_0^{\frac23}z)) & -(\Ai\,'(c_0^{\frac23}z)+i\sqrt{-c_0^{\frac23}z}\,\Ai(c_0^{\frac23}z)) \\
      \end{array}
    \right)
    \\
    \times
    v_{IV,2}^{\pm}(z)
\end{multline*}
Due to the asymptotics as $z\rightarrow-\infty$ (see \cite{Abramowitz-Stegun-1964})
\begin{equation*}
    \Ai(c_0^{\frac23}z)=\frac1{\sqrt\pi c_0^{\frac16}(-z)^{\frac14}}
    \left(\sin\left(\frac{2c_0}3(-z)^{\frac32}+\frac{\pi}4\right)
    -\cos\left(\frac{2c_0}3(-z)^{\frac32}+\frac{\pi}4\right)+o(1)\right),
\end{equation*}
\begin{equation*}
    \Ai\,'(c_0^{\frac23}z)=-\frac{c_0^{\frac16}(-z)^{\frac14}}{\sqrt\pi}
    \left(\cos\left(\frac{2c_0}3(-z)^{\frac32}+\frac{\pi}4\right)
    +\sin\left(\frac{2c_0}3(-z)^{\frac32}+\frac{\pi}4\right)+o(1)\right),
\end{equation*}
\begin{equation*}
    \Bi(c_0^{\frac23}z)=\frac1{\sqrt\pi c_0^{\frac16}(-z)^{\frac14}}
    \left(\cos\left(\frac{2c_0}3(-z)^{\frac32}+\frac{\pi}4\right)
    +\sin\left(\frac{2c_0}3(-z)^{\frac32}+\frac{\pi}4\right)+o(1)\right),
\end{equation*}
\begin{equation*}
    \Bi\,'(c_0^{\frac23}z)=-\frac{c_0^{\frac16}(-z)^{\frac14}}{\sqrt\pi}
    \left(-\sin\left(\frac{2c_0}3(-z)^{\frac32}+\frac{\pi}4\right)
    +\cos\left(\frac{2c_0}3(-z)^{\frac32}+\frac{\pi}4\right)+o(1)\right)
\end{equation*}
and by Lemma \ref{lem IV free system} we have:
\begin{multline*}
    \xi_{IV}^{\pm}=c_0^{\frac16}\sqrt{\frac{\pi}2}(-z)^{\frac14}
    \left(
      \begin{array}{cc}
        -i\exp\left(\frac{2ic_0}3(-z)^{\frac32}\right) & i\exp\left(-\frac{2ic_0}3(-z)^{\frac32}\right) \\
        \exp\left(\frac{2ic_0}3(-z)^{\frac32}\right) & \exp\left(-\frac{2ic_0}3(-z)^{\frac32}\right) \\
      \end{array}
    \right)
    \\
    \times
    \frac{\exp\left(\mp\frac{2ic_0}3(-z)^{\frac32}\right)}{(-z)^{\frac14}}e_{\pm}
    =c_0^{\frac16}\sqrt{\frac{\pi}2}
    \left(
      \begin{array}{c}
        \mp i \\
        1 \\
      \end{array}
    \right).
\end{multline*}
This means that
\begin{equation*}
    C_{IV}=c_0^{\frac16}\sqrt{\frac{\pi}2}
    \left(
      \begin{array}{cc}
        -i & i \\
        1 & 1 \\
      \end{array}
    \right).
\end{equation*}
Using the formula \eqref{C II} for $C_{II}$ and the relation \eqref{D} we come to the following convergence as $\varepsilon\rightarrow0^+$:
\begin{equation*}
    D(\varepsilon)\rightarrow C_{IV}^{-1}C_{II}=\frac{i}{2\sqrt2}
    \left(
      \begin{array}{cc}
        2 & \alpha_m-i \\
        -2 & -\alpha_m-i \\
      \end{array}
    \right),
\end{equation*}
which gives \eqref{eq u II + via u IV pm} and \eqref{eq u II - via u IV pm}. This completes the proof.
\end{proof}

Now we have everything to prove Theorem \ref{thm model problem}.

\begin{proof}[Proof of Theorem \ref{thm model problem}]
Since
\begin{equation*}
    \Lambda_{II,2}(z,\varepsilon)+\text{diag}\,S_{II,2}(z,\varepsilon)
    =\left(\frac{\lambda_{II}(t(z,\varepsilon))}{\varepsilon}
    \left(
    \begin{array}{cc}
    1 & 0 \\
    0 & -1 \\
    \end{array}
    \right)
    +\text{diag}\,S_{II}(t(z,\varepsilon))\right)t'(z,\varepsilon),
\end{equation*}
and
\begin{equation*}
    \Lambda_{IV,2}(z,\varepsilon)+\text{diag}\,S_{IV,2}(z,\varepsilon)
    =\left(\frac{\lambda_{IV}(t(z,\varepsilon))}{\varepsilon}
    \left(
    \begin{array}{cc}
    1 & 0 \\
    0 & -1 \\
    \end{array}
    \right)+\text{diag}\,S_{IV}(t(z,\varepsilon))\right)t'(z,\varepsilon),
\end{equation*}
see the formulae \eqref{Lambda II2}, \eqref{S II2} and \eqref{Lambda IV2}, \eqref{S IV2}, we can, using Lemmas \ref{lem II result} and \ref{lem IV result}, rewrite the asymptotics \eqref{II answer} and \eqref{IV answer} at the points $t_{I-II}$ and $t_{IV-V}$ as follows:
\begin{multline}\label{asympt u II pm}
    u_{II}^{\pm}(t_{I-II},\varepsilon)
    \\
    =a_{II}^{\pm}T_{II}(t_{I-II})
    \exp\left(
    \int_{t_{II-III}(\varepsilon)}^{t_{I-II}}
    \left(\pm\frac{\lambda_{II}(\tau)}{\varepsilon}+S_{II,\pm}(\tau)\right)d\tau\right)
    (e_{\pm}+o(1))
\end{multline}
and
\begin{multline}\label{asympt u IV pm}
    u_{IV}^{\pm}(t_{IV-V},\varepsilon)
    \\
    =a_{IV}^{\pm}T_{IV}(t_{IV-V})
    \exp\left(
    \int_{t_{III-IV}(\varepsilon)}^{t_{IV-V}}
    \left(\pm\frac{\lambda_{IV}(\tau)}{\varepsilon}+S_{IV,\pm}(\tau)\right)d\tau\right)
    (e_{\pm}+o(1))
\end{multline}
as $\varepsilon\rightarrow0^+$, where $S_{II,+}$ and $S_{IV,+}$ are upper-left and $S_{II,-}$ and $S_{IV,-}$ lower-right entries of the matrices $S_{II}$ and $S_{IV}$. Rewrite the last formula as
\begin{equation*}
    \exp\left(
    \int_{t_{IV-V}}^{t_{III-IV}(\varepsilon)}
    \left(\pm\frac{\lambda_{IV}(\tau)}{\varepsilon}+S_{IV,\pm}(\tau)\right)d\tau\right)u_{IV}^{\pm}(t_{IV-V},\varepsilon)
    \rightarrow a_{IV}^{\pm}T_{IV}(t_{IV-V})e_{\pm}.
\end{equation*}
By Lemma \ref{lem V answer} we also have:
\begin{equation*}
    \exp
     \left(-
    \int_{t_0}^{t_{IV-V}}\pm
    \frac{\lambda_V(\tau)}{\varepsilon}
    d\tau
    +\int_{t_{IV-V}}^{+\infty}
    S_{V,\pm}(\tau)
    d\tau
    \right)
    u_V(t_{IV-V},\varepsilon,e_{\pm})
    \rightarrow
     T_V(t_{IV-V})e_{\pm}.
\end{equation*}
Since $T_{IV}=T_{V}$, Lemma \ref{lem II tech 2} yields:
\begin{multline}\label{relation u IV+V}
    u_{IV}^{+}(t,\varepsilon)=\exp\left(
    \int_{t_{III-IV}(\varepsilon)}^{t_{IV-V}}
    \left(\frac{\lambda_{IV}(\tau)}{\varepsilon}+S_{IV,+}(\tau)\right)d\tau\right)
    \\
    \times
    \Biggl((a_{IV}^++\delta_5(\varepsilon))
    \exp
    \left(-
    \int_{t_0}^{t_{IV-V}}
    \frac{\lambda_V(\tau)}{\varepsilon}
    d\tau
    +\int_{t_{IV-V}}^{+\infty}
    S_{V,+}(\tau)
    d\tau
    \right)
    u_V(t,\varepsilon,e_+)
    \\
    +\delta_6(\varepsilon)
    \exp
    \left(
    \int_{t_0}^{t_{IV-V}}
    \frac{\lambda_V(\tau)}{\varepsilon}
    d\tau
    +\int_{t_{IV-V}}^{+\infty}
    S_{V,-}(\tau)
    d\tau
    \right)
    u_V(t,\varepsilon,e_-)
    \Biggr)
\end{multline}
and
\begin{multline}\label{relation u IV-V}
    u_{IV}^-(t,\varepsilon)=\exp\left(
    \int_{t_{III-IV}(\varepsilon)}^{t_{IV-V}}
    \left(-\frac{\lambda_{IV}(\tau)}{\varepsilon}+S_{IV,-}(\tau)\right)d\tau\right)
    \\
    \times
    \Biggl((a_{IV}^-+\delta_7(\varepsilon))
    \exp
    \left(
    \int_{t_0}^{t_{IV-V}}
    \frac{\lambda_V(\tau)}{\varepsilon}
    d\tau
    +\int_{t_{IV-V}}^{+\infty}
    S_{V,-}(\tau)
    d\tau
    \right)
    u_V(t,\varepsilon,e_-)
    \\
    +\delta_8(\varepsilon)
    \exp
    \left(-
    \int_{t_0}^{t_{IV-V}}
    \frac{\lambda_V(\tau)}{\varepsilon}
    d\tau
    +\int_{t_{IV-V}}^{+\infty}
    S_{V,+}(\tau)
    d\tau
    \right)
    u_V(t,\varepsilon,e_+)
    \Biggr)
\end{multline}
with some $\delta_5(\varepsilon),\delta_6(\varepsilon),\delta_7(\varepsilon),\delta_8(\varepsilon)\rightarrow0$ as $\varepsilon\rightarrow0^+$.

Let us now consider the solution $u_2^+(t,\varepsilon,f)$ for the case $f\nparallel f_-$ and prove the asymptotics \eqref{answer u}. By Lemma \eqref{lem I result},
\begin{equation*}
    u_2^+(t_{I-II},\varepsilon,f)
    =
     T_I(t_{I-II})\exp\left(
    \int_0^{t_{I-II}}\left(
    \frac{\lambda_I(\tau)}{\varepsilon}+
    S_{I,+}(\tau)
    \right)d\tau\right)
    \left(
    \Phi(f)e_+
    +o(1)
    \right).
\end{equation*}
Rewrite this as
\begin{equation*}
    \exp\left(
    -\int_0^{t_{I-II}}\left(
    \frac{\lambda_I(\tau)}{\varepsilon}+
    S_{I,+}(\tau)
    \right)d\tau\right)u_2^+(t_{I-II},\varepsilon,f)
    \rightarrow
     T_I(t_{I-II})\Phi(f)e_+.
\end{equation*}
Rewrite also the asymptotics \eqref{asympt u II pm} as
\begin{equation*}
    \exp\left(
    \int_{t_{I-II}}^{t_{II-III}(\varepsilon)}
    \left(\pm\frac{\lambda_{II}(\tau)}{\varepsilon}+S_{II,\pm}(\tau)\right)d\tau\right)
    u_{II}^{\pm}(t_{I-II},\varepsilon)
    \rightarrow
    a_{II}^{\pm}T_{II}(t_{I-II})e_{\pm}.
\end{equation*}
Since $T_Ie_+=T_{II}e_+$, by Lemma \ref{lem II tech 2} we get that
\begin{multline*}
    u_2^+(t,\varepsilon,f)=\exp\left(
    \int_0^{t_{I-II}}\left(
    \frac{\lambda_I(\tau)}{\varepsilon}+
    S_{I,+}(\tau)
    \right)d\tau\right)
    \\
    \times
    \Biggl(
    \left(\frac{\Phi(f)}{a_{II}^+}+\delta_9(\varepsilon)\right)
    \exp\left(
    \int_{t_{I-II}}^{t_{II-III}(\varepsilon)}
    \left(\frac{\lambda_{II}(\tau)}{\varepsilon}+S_{II,+}(\tau)\right)d\tau\right)
    u_{II}^{+}(t,\varepsilon)
    \\
    +\delta_{10}(\varepsilon)
    \exp\left(
    \int_{t_{I-II}}^{t_{II-III}(\varepsilon)}
    \left(-\frac{\lambda_{II}(\tau)}{\varepsilon}+S_{II,-}(\tau)\right)d\tau\right)
    u_{II}^{-}(t,\varepsilon)
    \Biggr)
\end{multline*}
with some $\delta_9(\varepsilon),\delta_{10}(\varepsilon)\rightarrow0$ as $\varepsilon\rightarrow0^+$. Using Lemma \ref{lem matching II-IV} and the fact that
\begin{multline*}
    \exp\left(
    \int_{t_{I-II}}^{t_{II-III}(\varepsilon)}
    \left(-\frac{\lambda_{II}(\tau)}{\varepsilon}+S_{II,-}(\tau)\right)d\tau\right)
    \\
    =O\left(
    \exp\left(
    \int_{t_{I-II}}^{t_{II-III}(\varepsilon)}
    \left(\frac{\lambda_{II}(\tau)}{\varepsilon}+S_{II,+}(\tau)\right)d\tau\right)
    \right)\text{ as }\varepsilon\rightarrow0^+,
\end{multline*}
as well as the identities $\lambda_I=\lambda_{II}$ and $S_{I,+}=S_{II,+}$, we have:
\begin{multline*}
    u_2^+(t,\varepsilon,f)=\frac{i\Phi(f)}{\sqrt2a_{II}^+}\exp\left(
    \int_0^{t_{II-III}(\varepsilon)}\left(
    \frac{\lambda_I(\tau)}{\varepsilon}+
    S_{I,+}(\tau)
    \right)d\tau\right)
    \\
    \times
    ((1+\delta_{11}(\varepsilon))u_{IV}^{+}(t,\varepsilon)-(1+\delta_{12}(\varepsilon))u_{IV}^{-}(t,\varepsilon))
\end{multline*}
with some $\delta_{11}(\varepsilon),\delta_{12}(\varepsilon)\rightarrow0$ as $\varepsilon\rightarrow0^+$. Since
\begin{equation*}
    \left|\exp\left(
    \int_{t_{III-IV}(\varepsilon)}^{t_{IV-V}}\pm
    \frac{\lambda_{IV}(\tau)}{\varepsilon}d\tau\right)\right|=1,
\end{equation*}
the relations \eqref{relation u IV+V} and \eqref{relation u IV-V} together with the identities $\lambda_{IV}=\lambda_{V}$ and $S_{IV}=S_{V}$ imply that
\begin{multline*}
    u_2^+(t,\varepsilon,f)=\frac{i\Phi(f)}{\sqrt2a_{II}^+}\exp\left(
    \int_0^{t_{II-III}(\varepsilon)}\left(
    \frac{\lambda_I(\tau)}{\varepsilon}+
    S_{I,+}(\tau)
    \right)d\tau\right)
    \\
    \times
    \Biggl((a_{IV}^++\delta_{13}(\varepsilon))
    \exp
    \left(-
    \int_{t_0}^{t_{III-IV}(\varepsilon)}
    \frac{\lambda_V(\tau)}{\varepsilon}
    d\tau
    +\int_{t_{III-IV}(\varepsilon)}^{+\infty}
    S_{V,+}(\tau)
    d\tau
    \right)
    u_V(t,\varepsilon,e_+)
    \\
    -(a_{IV}^-+\delta_{14}(\varepsilon))
    \exp
    \left(
    \int_{t_0}^{t_{III-IV}(\varepsilon)}
    \frac{\lambda_V(\tau)}{\varepsilon}
    d\tau
    +\int_{t_{III-IV}(\varepsilon)}^{+\infty}
    S_{V,-}(\tau)
    d\tau
    \right)
    u_V(t,\varepsilon,e_-)
    \Biggr)
\end{multline*}
with some $\delta_{13}(\varepsilon),\delta_{14}(\varepsilon)\rightarrow0$ as $\varepsilon\rightarrow0^+$. By Lemma \ref{lem V result for u V2} this means that
\begin{multline}\label{eq matching 1}
    \exp\left(-\int_{t_0}^t\frac{\lambda_V(\tau)}{\varepsilon}
    \left(
      \begin{array}{cc}
        1 & 0 \\
        0 & -1 \\
      \end{array}
    \right)
    d\tau
    \right)T_V^{-1}(t)u_2^+(t,\varepsilon,f)
    \\
    \rightarrow
    \frac{i\Phi(f)}{\sqrt2a_{II}^+}\exp\left(
    \int_0^{t_{II-III}(\varepsilon)}\left(
    \frac{\lambda_I(\tau)}{\varepsilon}+
    S_{I,+}(\tau)
    \right)d\tau\right)
    \\
    \times
    \Biggl((a_{IV}^++\delta_{13}(\varepsilon))
    \exp
    \left(-
    \int_{t_0}^{t_{III-IV}(\varepsilon)}
    \frac{\lambda_V(\tau)}{\varepsilon}
    d\tau
    +\int_{t_{III-IV}(\varepsilon)}^{+\infty}
    S_{V,+}(\tau)
    d\tau
    \right)
    e_+
    \\
    -(a_{IV}^-+\delta_{14}(\varepsilon))
    \exp
    \left(
    \int_{t_0}^{t_{III-IV}(\varepsilon)}
    \frac{\lambda_V(\tau)}{\varepsilon}
    d\tau
    +\int_{t_{III-IV}(\varepsilon)}^{+\infty}
    S_{V,-}(\tau)
    d\tau
    \right)
    e_-
    \Biggr)
\end{multline}
as $t\rightarrow+\infty$. Let us now simplify the right-hand side. As
\begin{equation*}
    a_{II}^+=\frac{\exp\left(-\frac{2c_0}3Z_0^{\frac32}\right)}{Z_0^{\frac14}},
    \
    a_{IV}^{\pm}=\frac{\exp\left(\mp\frac{2ic_0}3Z_0^{\frac32}\right)}{Z_0^{\frac14}},
\end{equation*}
we have
\begin{multline*}
    Z_0^{\frac14}a_{II}^+\exp\left(-\int_{t_{II-III}(\varepsilon)}^{t_0}\frac{\lambda_I(\tau)}{\varepsilon}d\tau\right)
    =
    Z_0^{\frac14}a_{II}^+\exp\left(\int_0^{Z_0}\lambda_{II,2}(s,\varepsilon)ds\right)
    \\
    =Z_0^{\frac14}a_{II}^+\exp\left(\int_0^{Z_0}(c_0\sqrt s +o(1))ds\right)=1+o(1)\text{ as }\varepsilon\rightarrow0^+
\end{multline*}
and
\begin{multline*}
    Z_0^{\frac14}a_{IV}^{\pm}\exp\left(\mp\int_{t_0}^{t_{III-IV}(\varepsilon)}\frac{\lambda_V(\tau)}{\varepsilon}d\tau\right)
    =
    Z_0^{\frac14}a_{IV}^{\pm}\exp\left(\pm\int_{-Z_0}^0\lambda_{IV,2}(s,\varepsilon)ds\right)
    \\
    =Z_0^{\frac14}a_{IV}^{\pm}\exp\left(\int_{-Z_0}^0(\pm ic_0\sqrt{-s} +o(1))ds\right)=1+o(1)\text{ as }\varepsilon\rightarrow0^+.
\end{multline*}
Furthermore,
\begin{equation*}
    S_{I,+}(t)=
    \frac{\gamma t^{\gamma}}{8\beta^2t^{1-\gamma}\left(1-\frac{t^{2\gamma}}{4\beta^2}\right)\left(1+\sqrt{1-\frac{t^{2\gamma}}{4\beta^2}}\right)}
    =\frac{\gamma\left(1-\sqrt{1-\frac{t^{2\gamma}}{4\beta^2}}\right)}{2t\left(1-\frac{t^{2\gamma}}{4\beta^2}\right)},
\end{equation*}
\begin{equation*}
    S_{V,\pm}(t)=
    -\frac{\beta\gamma\left(\frac{2\beta}{t^{\gamma}}\pm i\sqrt{1-\frac{4\beta^2}{t^{2\gamma}}}\right)}
    {\left(1-\frac{4\beta^2}{t^{2\gamma}}\right)t^{\gamma+1}}
    =\frac{\gamma\left(1\pm i\sqrt{\frac{t^{2\gamma}}{4\beta^2}-1}\right)}{2t\left(1-\frac{t^{2\gamma}}{4\beta^2}\right)}.
\end{equation*}
Therefore there exists the limit
\begin{multline}\label{c 32}
    c_{vp}(\beta,\gamma):=\lim_{\varepsilon\rightarrow0^+}
    \left(
    \int_0^{t_{II-III}(\varepsilon)}
    S_{I,+}(\tau)d\tau
    +\int_{t_{III-IV}(\varepsilon)}^{+\infty}\Re S_{V,\pm}(\tau)d\tau\right)
    \\
    =
    \lim_{\Delta\rightarrow0^+}
    \left(
    \int_0^{t_0-\Delta}
    \frac{\gamma\left(1-\sqrt{1-\frac{\tau^{2\gamma}}{4\beta^2}}\right)}{2\tau\left(1-\frac{\tau^{2\gamma}}{4\beta^2}\right)}d\tau
    +\int_{t_0+\Delta}^{+\infty}
    \frac{\gamma d\tau}{2\tau\left(1-\frac{\tau^{2\gamma}}{4\beta^2}\right)}\right)
    \\
    =
    \int_0^{\frac{t_0}2}
    \frac{\gamma\left(1-\sqrt{1-\frac{\tau^{2\gamma}}{4\beta^2}}\right)}{2\tau\left(1-\frac{\tau^{2\gamma}}{4\beta^2}\right)}d\tau
    -
    \int_{\frac{t_0}2}^{t_0}
    \frac{\gamma d\tau}{2\tau\sqrt{1-\frac{\tau^{2\gamma}}{4\beta^2}}}
    +\text{v.p.}\int_{\frac{t_0}2}^{+\infty}
    \frac{\gamma d\tau}{2\tau\left(1-\frac{\tau^{2\gamma}}{4\beta^2}\right)}.
\end{multline}
Then the right-hand side of \eqref{eq matching 1} can be rewritten as
\begin{multline*}
    \frac{ie^{c_{vp}(\beta,\gamma)}\Phi(f)}{\sqrt2}
    \exp\left(\int_0^{t_0}\frac{\lambda_I(\tau)}{\varepsilon}d\tau\right)
    \\
    \times
    \Biggl((1+\delta_{15}(\varepsilon))
    \exp
    \left(i\int_{t_{III-IV}(\varepsilon)}^{+\infty}
    \Im S_{V,+}(\tau)
    d\tau
    \right)
    e_+
    \\
    -(1+\delta_{16}(\varepsilon))
    \exp
    \left(i\int_{t_{III-IV}(\varepsilon)}^{+\infty}
    \Im S_{V,-}(\tau)
    d\tau
    \right)
    e_-
    \Biggr)
\end{multline*}
with some $\delta_{15}(\varepsilon),\delta_{16}(\varepsilon)\rightarrow0$ as $\varepsilon\rightarrow0^+$. It follows that
\begin{equation*}
    \|T_V^{-1}(t)u_2^+(t,\varepsilon,f)\|\rightarrow
    \frac{e^{c_{vp}(\beta,\gamma)}|\Phi(f)|}{\sqrt2}\exp\left(\int_0^{t_0}\frac{\lambda_I(\tau)d\tau}{\varepsilon}\right)
    \left\|
    \left(
      \begin{array}{c}
        1+\delta_{15}(\varepsilon) \\
        1+\delta_{16}(\varepsilon) \\
      \end{array}
    \right)
    \right\|
\end{equation*}
as $t\rightarrow+\infty$. To get rid of $T_V^{-1}$ we recall that
\begin{equation*}
    T_V(t)\rightarrow
    \left(
      \begin{array}{cc}
        1 & 1 \\
        i & -i \\
      \end{array}
    \right)
    \text{ as }t\rightarrow+\infty
\end{equation*}
and that the matrix
$\frac1{\sqrt2}
\left(
  \begin{array}{cc}
    1 & 1 \\
    i & -i \\
  \end{array}
\right)$
is unitary. Besides that, for every $\varepsilon\in U$ the solution $u_2^+(t,\varepsilon,f)$ is bounded as $t\rightarrow+\infty$, which follows from Lemma \ref{lem model problem individual asymptotics} and the relation \eqref{u 2 pm}. Thus we conclude that
\begin{equation*}
    \|u_2^+(t,\varepsilon,f)\|\rightarrow
    \frac{e^{c_{vp}(\beta,\gamma)}|\Phi(f)|}{2}\exp\left(\int_0^{t_0}\frac{\lambda_I(\tau)d\tau}{\varepsilon}\right)
    \left\|
    \left(
      \begin{array}{c}
        1+\delta_{15}(\varepsilon) \\
        1+\delta_{16}(\varepsilon) \\
      \end{array}
    \right)
    \right\|\text{ as }t\rightarrow+\infty
\end{equation*}
for every fixed sufficiently small $\varepsilon>0$. Now we can take $\varepsilon\rightarrow0^+$ to obtain the asymptotics. With the expressions \eqref{lambda I} for $\lambda_I$ and \eqref{C mp} for $C_{mp}(\beta,\gamma)$ this proves the formula \eqref{answer u 3}, and the assertion of Theorem \ref{thm model problem} for $\varepsilon\rightarrow0^+$ follows.

Consider now the solution $u_2^+(t,\varepsilon,f_-)$. Let us prove the estimate \eqref{answer u}. First use the relations \eqref{relation u IV+V} and \eqref{relation u IV-V} together with the facts that $\lambda_{IV}=\lambda_V$ are purely imaginary, $S_V$ is summable at infinity and that $\|u_V(t,\varepsilon,e_{\pm})\|\rightarrow\sqrt2$ as $t\rightarrow+\infty$ by Lemma \ref{lem V result for u V2} to conclude that
\begin{equation*}
    \lim_{t\rightarrow+\infty}\|u_{IV}^{\pm}(t,\varepsilon)\|=O\left(\exp\left(\int_{t_{III-IV}(\varepsilon)}^{t_{IV-V}}S_{IV,+}(\tau)d\tau\right)\right)
    \text{ as }\varepsilon\rightarrow0^+.
\end{equation*}
By Lemma \ref{lem matching II-IV} this means that
\begin{equation}\label{matching eq 1}
    \lim_{t\rightarrow+\infty}\|u_{II}^{\pm}(t,\varepsilon)\|=O\left(\exp\left(\int_{t_{III-IV}(\varepsilon)}^{t_{IV-V}}S_{IV,+}(\tau)d\tau\right)\right)
    \text{ as }\varepsilon\rightarrow0^+.
\end{equation}
Let us define two solutions $u_{I-II}^{\pm}$ of the system \eqref{system u 2} as
\begin{equation*}
    u_{I-II}^{\pm}(t,\varepsilon):=\frac1{a_{II}^{\pm}}\exp\left(\int_{t_{I-II}}^{t_{II-III}(\varepsilon)}
    \left(\pm\frac{\lambda_{II}(\tau)}{\varepsilon}+S_{II,\pm}(\tau)\right)d\tau\right)u_{II}^{\pm}(t,\varepsilon)
\end{equation*}
On the one hand, from \eqref{matching eq 1}, the equalities $\lambda_I=\lambda_{II}$, $S_I=S_{II}$, $S_{IV}=S_V$ and finiteness of the limit in \eqref{c 32}, we have
\begin{equation}\label{matching eq 2}
    \lim_{t\rightarrow+\infty}\|u_{I-II}^{\pm}(t,\varepsilon)\|
    =O\left(\exp\left(\int_{t_{I-II}}^{t_0}\frac{\lambda_I(\tau)d\tau}{\varepsilon}\right)\right)
    \text{ as }\varepsilon\rightarrow0^+.
\end{equation}
On the other hand, from the asymptotics \eqref{asympt u II pm}
\begin{equation}\label{matching eq 3}
    u_{I-II}^{\pm}(t_{I-II},\varepsilon)\rightarrow T_I(t_{I-II})e_{\pm}
    \text{ as }\varepsilon\rightarrow0^+.
\end{equation}
From this we will now estimate the growth of the norm of the fundamental solution. Define for every $h\in\mathbb C^2$ the solution $u_{I-II}(t,\varepsilon,h)$ of the system \eqref{system u 2} on the interval $[t_{I-II},+\infty)$ with the initial condition
\begin{equation*}
    u_{I-II}(t_{I-II},\varepsilon,h)=T_{I-II}(t_{I-II})h.
\end{equation*}
The fundamental solution $F$ of the system \eqref{system u 2} can be written as follows:
\begin{equation*}
    F(t,t_{I-II},\varepsilon)=\left(u_{I-II}(t,\varepsilon,e_+)|u_{I-II}(t,\varepsilon,e_-)\right)T_I^{-1}(t_{I-II}),
\end{equation*}
therefore its norm can be estimated by the norms of solutions $u_{I-II}(t,\varepsilon,e_{\pm})$. For them we have from \eqref{matching eq 3} by Lemma \ref{lem II tech 2}:
\begin{equation*}
    u_{I-II}(t,\varepsilon,e_+)=(1+\delta_{17}(\varepsilon))u_{I-II}^+(t,\varepsilon)+\delta_{18}(\varepsilon)u_{I-II}^-(t,\varepsilon),
\end{equation*}
\begin{equation*}
    u_{I-II}(t,\varepsilon,e_-)=\delta_{19}(\varepsilon)u_{I-II}^+(t,\varepsilon)+(1+\delta_{20}(\varepsilon))u_{I-II}^-(t,\varepsilon)
\end{equation*}
with some $\delta_{17}(\varepsilon),\delta_{18}(\varepsilon),\delta_{19}(\varepsilon),\delta_{20}(\varepsilon)\rightarrow0$ as $\varepsilon\rightarrow0^+$. Therefore \eqref{matching eq 2} means that
\begin{equation*}
    \limsup_{t\rightarrow+\infty}\|F(t,t_{I-II},\varepsilon)\|
    =O\left(\exp\left(\int_{t_{I-II}}^{t_0}\frac{\lambda_I(\tau)d\tau}{\varepsilon}\right)\right)
    \text{ as }\varepsilon\rightarrow0^+.
\end{equation*}
This estimate together with the estimate of $u_2^+(t_{I-II},\varepsilon,f_-)$ by Lemma \ref{lem I result} imply that
\begin{multline*}
    \lim_{t\rightarrow+\infty}\|u_2^+(t,\varepsilon,f_-)\|\le
    \\
    \|u_2^+(t_{I-II},\varepsilon,f_-)\|\limsup_{t\rightarrow+\infty}\|F(t,t_{I-II},\varepsilon)\|
    =o\left(\exp\left(\int_0^{t_0}\frac{\lambda_I(\tau)d\tau}{\varepsilon}\right)\right),
\end{multline*}
which is the desired estimate.

To prove \eqref{answer u} for $\varepsilon_0\rightarrow0^-$ consider the solution $u_2^-(t,\varepsilon,f)$ of the system
\begin{equation*}
    u_2^-\,'(t)
    =\left(
    \left(
      \begin{array}{cc}
        \frac{\beta}{t^{\gamma}} & \frac12 \\
        -\frac12 & -\frac{\beta}{t^{\gamma}} \\
      \end{array}
    \right)
    +R_2^-(t,\varepsilon)\right)u_2^-(t).
\end{equation*}
If one takes
$u_2^-(t)=
\left(
  \begin{array}{cc}
    1 & 0 \\
    0 & -1 \\
  \end{array}
\right)
u_2^+(t)$, this system turns into the following:
\begin{equation}\label{system u 2- transformed to +}
    u_2^+\,'(t)
    =\left(
    \left(
      \begin{array}{cc}
        \frac{\beta}{t^{\gamma}} & -\frac12 \\
        \frac12 & -\frac{\beta}{t^{\gamma}} \\
      \end{array}
    \right)
    +
    \left(
      \begin{array}{cc}
        1 & 0 \\
        0 & -1 \\
      \end{array}
    \right)
    R_2^-(t,\varepsilon)
    \left(
      \begin{array}{cc}
        1 & 0 \\
        0 & -1 \\
      \end{array}
    \right)\right)u_2^+(t).
\end{equation}
If one proves that the asymptotics as $\varepsilon\rightarrow0^+$ of the solution
\begin{equation*}
    \left(
      \begin{array}{cc}
        1 & 0 \\
        0 & -1 \\
      \end{array}
    \right)u_2^-(t,\varepsilon,f)
\end{equation*}
in the region $I$, namely at the point $t_{I-II}$, coincides with the asymptotics of the solution $u_2^+(t,\varepsilon,f)$ at the same point (which is given by Lemma \ref{lem I result}), then the rest follows automatically. This is because Lemmas \ref{lem V answer}, \ref{lem II result}, \ref{lem IV result} and \ref{lem III main} are directly applicable to the system \eqref{system u 2- transformed to +} and the matching procedure (or estimating, in the case $f=f_-$) works for the solution
$\left(
  \begin{array}{cc}
    1 & 0 \\
    0 & -1 \\
  \end{array}
\right)u_2^-$
literally as above. Hence the asymptotic behaviour of the norm is the same. In the argument of Section \ref{section I} every estimate and convergence remain the same as for the case of $u_2^+$, except the calculation \eqref{I eq free term}. In an analogue of that calculation one finally arrives at
\begin{multline*}
    \int_0^{\varepsilon_0^{-\frac1{\gamma}}t}
    \left(
      \begin{array}{cc}
        1 & 0 \\
        0 & \exp\left(-\frac2{\varepsilon}\int_{\varepsilon_0^{\frac1{\gamma}}x}^t\lambda_I\right) \\
      \end{array}
    \right)
    T_I^{-1}(\varepsilon_0^{\frac1{\gamma}}x)
    \left(
      \begin{array}{cc}
        1 & 0 \\
        0 & -1 \\
      \end{array}
    \right)
    \\
    \times
    \left(
    \begin{array}{cc}
    \cos\left(\frac{\varepsilon_0 x}{2}\right) & \sin\left(\frac{\varepsilon_0 x}{2}\right)
    \\
    \sin\left(\frac{\varepsilon_0 x}{2}\right) & -\cos\left(\frac{\varepsilon_0 x}{2}\right)
    \end{array}
    \right)
    R(x,\varepsilon_0)u(x,\varepsilon_0,f)
    \exp\left(-\frac1{\varepsilon}\int_0^{\varepsilon_0^{\frac1{\gamma}}x}\lambda_I\right)dx,
\end{multline*}
which differs from the result of the calculation \eqref{I eq free term} by presence of the matrix
$\left(
  \begin{array}{cc}
    1 & 0 \\
    0 & -1 \\
  \end{array}
\right)$. Since $T_I(t)\rightarrow I$ as $t\rightarrow0$, in the limit as $\varepsilon\rightarrow0^+$ this term affects the lower component which goes to zero for the same reason as in the case of $u_2^+$, so the result is still the same. This proves \eqref{answer u} for $\varepsilon_0\rightarrow0^-$ and completes the proof of Theorem \ref{thm model problem}.
\end{proof}

\section{Proof of the main result}\label{section proof}
In this section we prove Theorem \ref{thm gamma<1} putting together everything that was obtained in the previous sections.

\begin{proof}[Proof of Theorem \ref{thm gamma<1}]
Consider the critical point $\nu_{cr}$ and let $\alpha\neq\alpha_{cr}$. By Lemma \ref{lem reduction} there exists the neighbourhood $U_{cr}$, and there the eigenfunction equation for the operator $\mathcal L_{\alpha}$ is equivalent to the system
\begin{equation*}
    w_{cr}'(x)=
    \left(
    \frac{\beta_{cr}}{x^{\gamma}}
    \left(
    \begin{array}{cc}
    \cos(\varepsilon_{cr}(\lambda)x) & \sin(\varepsilon_{cr}(\lambda)x)
    \\
    \sin(\varepsilon_{cr}(\lambda)x) & -\cos(\varepsilon_{cr}(\lambda)x)
    \end{array}
    \right)
    +
    R_{cr}(x,\lambda)
    \right)
    w_{cr}(x).
\end{equation*}
Rewrite this system as
\begin{equation*}
    u'(x)=
    \left(
    \frac{\beta_{cr}}{x^{\gamma}}
    \left(
    \begin{array}{cc}
    \cos(\varepsilon_0x) & \sin(\varepsilon_0x)
    \\
    \sin(\varepsilon_0x) & -\cos(\varepsilon_0x)
    \end{array}
    \right)
    +
    R_{cr}(x,\lambda(\varepsilon_0))
    \right)
    u(x)
\end{equation*}
with $\lambda(\varepsilon_0)=\varepsilon_{cr}^{-1}(\varepsilon_0)$, which means that $\lambda$ is parametrised by $\varepsilon_0$ so that $\varepsilon_{cr}(\lambda)=\varepsilon_0$.
Properties of the remainder provided by Lemma \ref{lem reduction} are such that the conditions \eqref{model problem r} and \eqref{model problem conditions} are satisfied. By Lemma \ref{lem model problem individual asymptotics} there exists $f_{cr-}\in\mathbb C^2\backslash\{0\}$ such that the solution $u_{cr}(x,0,f)$ has the asymptotics
\begin{equation*}
    u_{cr}(x,0,f_{cr-})=\exp\left(-\frac{\beta_{cr}x^{1-\gamma}}{1-\gamma}\right)(e_-+o(1))
    \text{ as }x\rightarrow+\infty
\end{equation*}
and for $f\nparallel f_{cr-}$
\begin{equation*}
    u_{cr}(x,0,f)=\exp\left(\frac{\beta_{cr}x^{1-\gamma}}{1-\gamma}\right)(\Phi_{cr}(f)e_++o(1))
    \text{ as }x\rightarrow+\infty.
\end{equation*}
In fact, the matrix $R_{cr}$ has real entries, so $f_{cr-}\in\mathbb R^2\backslash\{0\}$. Moreover, comparing these asymptotics with \eqref{asympt w-} and \eqref{asympt w+} provided by Lemma \ref{lem reduction}, we see that
\begin{equation*}
    f_{cr-}=\frac{g_{cr,\alpha_{cr}}}{d_{cr-}}
\end{equation*}
and
\begin{equation}\label{relation Phi to sin}
    \Phi_{cr}(f_{cr,\alpha})=d_{cr+}\sin(\alpha-\alpha_{cr}).
\end{equation}
Since $\alpha\neq\alpha_{cr}$, the vectors $g_{cr,\alpha}$ and $f_{cr-}$ are linearly independent and so, by Lemma \ref{lem II tech 2},
\begin{equation*}
    w_{cr,\alpha}(0,\lambda(\varepsilon_0))=(1+\delta_{21}(\varepsilon_0))g_{cr,\alpha}+\delta_{22}(\varepsilon_0)f_{cr-}
\end{equation*}
with some $\delta_{21}(\varepsilon),\delta_{22}(\varepsilon)\rightarrow0$ as $\varepsilon\rightarrow0$. Therefore
\begin{equation}\label{relation w u}
    w_{cr,\alpha}(x,\lambda(\varepsilon_0))=(1+\delta_{21}(\varepsilon_0))u(x,\varepsilon_0,g_{cr,\alpha})
    +\delta_{22}(\varepsilon_0)u(x,\varepsilon_0,f_{cr-}).
\end{equation}
By Theorem \ref{thm model problem} using the relation \eqref{relation Phi to sin} we have
\begin{multline*}
    \lim_{x\rightarrow+\infty}\|u_{cr}(x,\varepsilon_0,g_{cr,\alpha})\|
    =
    C_{mp}(\beta_{cr},\gamma)\exp
    \left(
    \frac1{|\varepsilon_0|^{\frac{1-\gamma}{\gamma}}}\int_0^{(2\beta_{cr})^{\frac1{\gamma}}}\sqrt{\frac{\beta_{cr}^2}{t^{2\gamma}}-\frac14}dt
    \right)
    \\
    \times
    (|d_{cr+}\sin(\alpha-\alpha_{cr})|+o(1)),
\end{multline*}
and
\begin{equation*}
    \lim_{x\rightarrow+\infty}\|u_{cr}(x,\varepsilon_0,f_{cr-})\|
    =
    o\left(
    \exp
    \left(
    \frac1{|\varepsilon_0|^{\frac{1-\gamma}{\gamma}}}\int_0^{(2\beta_{cr})^{\frac1{\gamma}}}\sqrt{\frac{\beta_{cr}^2}{t^{2\gamma}}-\frac14}dt
    \right)
    \right)
\end{equation*}
as $\varepsilon_0\rightarrow0$. With these asymptotics it follows from \eqref{relation w u} that
\begin{multline*}
    \lim_{x\rightarrow+\infty}\|w_{cr,\alpha}(x,\lambda)\|
    =
    C_{mp}(\beta_{cr},\gamma)\exp
    \left(
    \frac1{|\varepsilon_{cr}(\lambda)|^{\frac{1-\gamma}{\gamma}}}\int_0^{(2\beta_{cr})^{\frac1{\gamma}}}\sqrt{\frac{\beta_{cr}^2}{t^{2\gamma}}-\frac14}dt
    \right)
    \\
    \times
    (|d_{cr+}\sin(\alpha-\alpha_{cr})|+o(1))
    \text{ as }\lambda\rightarrow\nu_{cr}.
\end{multline*}
Since, by Proposition \ref{prop Titchmarsh--Weyl formula} and the equality \eqref{A=eta},
\begin{equation*}
    \rho'_{\alpha}(\lambda)=\frac1{2\pi|W\{\psi_+,\psi_-\}(\lambda)| \; |A_{\alpha}(\lambda)|^2}
    =\frac1{2\pi|W\{\psi_+,\psi_-\}(\lambda)| \left\|\lim\limits_{x\rightarrow+\infty}w_{cr,\alpha}(x,\lambda)\right\|^2},
\end{equation*}
using continuity of $W\{\psi_+,\psi_-\}(\cdot)$ we have
\begin{equation}\label{proof eq 1}
    \rho'_{\alpha}(\lambda)=\frac{a_{cr}}{d_{cr+}^2\sin^2(\alpha-\alpha_{cr})}
    \exp
    \left(-
    \frac2{|\varepsilon_{cr}(\lambda)|^{\frac{1-\gamma}{\gamma}}}
    \int_0^{(2\beta_{cr})^{\frac1{\gamma}}}\sqrt{\frac{\beta_{cr}^2}{t^{2\gamma}}-\frac14}dt
    \right)(1+o(1)),
\end{equation}
where
\begin{equation*}
    a_{cr}=\frac1{2\pi|W\{\psi_+,\psi_-\}(\nu_{cr})|C_{mp}^2(\beta_{cr},\gamma)},
\end{equation*}
which coincides with \eqref{a cr} after substitution of $C_{mp}$ from \eqref{C mp}. Using the property \eqref{epsilon 0} we have
\begin{equation*}
    \frac1{|\varepsilon_{cr}(\lambda)|^{\frac{1-\gamma}{\gamma}}}
    =
    \frac1{|\lambda-\nu_{cr}|^{\frac{1-\gamma}{\gamma}}}
    \left(
    \frac a{2\pi k'(\nu_{cr})}
    \right)^{\frac{1-\gamma}{\gamma}}
    (1+O(|\lambda-\nu_{cr}|^{\frac1{\gamma}}))
    \text{ as }\lambda\rightarrow\nu_{cr}.
\end{equation*}
Substituting this and the result of the calculation \eqref{exponent} into \eqref{proof eq 1} we finally arrive at the asymptotics \eqref{answer} with $c_{cr}$ given by \eqref{c cr}. This completes the proof.
\end{proof}

\subsection*{Acknowledgments}
The author wishes to express his gratitude to Sergey Naboko for proposing this problem and his interest to this work and to Daphne Gilbert for her constant attention and help. This work was supported by the Irish Research Council (Government of Ireland Postdoctoral Fellowship in Science, Engineering and Technology).

\end{document}